\documentclass{amsart}

\usepackage{amsmath,amssymb,amsfonts,amsthm}
\usepackage{color}

\def\le{\leqslant}
\def\ge{\geqslant}

\newcommand{\beq}{\begin{equation}}
\newcommand{\eeq}{\end{equation}}
\newcommand{\beqst}{\begin{equation*}}
\newcommand{\eeqst}{\end{equation*}}
\newcommand{\bal}{\begin{align}}
\newcommand{\eal}{\end{align}}
\newcommand{\balst}{\begin{align*}}
\newcommand{\ealst}{\end{align*}}

\newtheorem{lemma}{Lemma}[section]
\newtheorem{theorem}[lemma]{Theorem}
\newtheorem{proposition}[lemma]{Proposition}

\newtheorem{corollary}[lemma]{Corollary}

\newtheorem{notation}[lemma]{Notation}

\newcommand{\N}{{\mathbb N }}

\newcommand{\R}{{\mathbb R}}
\newcommand{\C}{{\mathbb C}}

\newcommand{\eps}{{\varepsilon }}

\def\({\left(}
\def\){\right)}
\def\<{\left\langle}
\def\>{\right\rangle}
\def\O{\mathcal O}
\def\mcN{\mathcal N}

\newcommand{\Id}[1]{{\rm I\kern-2pt I_{#1}}}
\renewcommand{\hbar}{{\displaystyle\bar{\phantom{x}}\kern-6pt h}}

\numberwithin{equation}{section}

\newcommand\cc{\mathrm{c.c.}}
\newcommand{\pl}{\partial}

\newcommand{\omj}{\omega_j}
\newcommand{\omi}{\omega_i}
\newcommand{\omk}{\omega_k}
\def\om{{\omega}}
\def\Om{{\Omega}}

\def\mcN{\mathcal N}

\def\mcG{\mathcal G}
\def\mcR{\mathcal R}
\def\mcE{\mathcal E}
\def\mcH{\mathcal H}
\def\mfP{\mathfrak P}
\def\mfe{\mathfrak e}
\def\mfS{\mathfrak S}
\def\ds{\displaystyle}
\def\ep{\epsilon}
\def\ve{\varepsilon}

\def\bk{\xi}

\def\ii{\mathrm{i}}
\def\ee{\,\mathrm{e}}
\def\psia{\psi_{a}}
\def\zetaa{\zeta_{a}}
\def\Gm0{\mcG_{\mu0}}
\def\bond{ { \textstyle \frac1{\mathrm{Bo}}}}
\def\nablabo{ \nabla_{ \mathrm{Bo}\, } }

\def\ds{\displaystyle}
\def\ts{\textstyle}

\begin{document}

\title[Interaction of modulated gravity water waves of finite depth]{Interaction of modulated\\ gravity water waves of finite depth
%with flat bottom
}

\author{% W.~H.~Aschbacher and 
Ioannis Giannoulis }
\address{Department of Mathematics, University of Ioannina, GR-45110 Ioannina, Greece}
\email{giannoul@uoi.gr}

\begin{abstract} 
We consider the capillary-gravity water-waves problem of finite depth with a flat bottom of one or two horizontal dimensions.
We derive the modulation equations of leading and next-to-leading order in the hyperbolic scaling for three weakly amplitude-modulated plane-wave solutions of the linearized problem in the absence of quadratic and cubic resonances.
We fully justify the derived system of macroscopic equations in the case of pure gravity waves, i.e.\ in the case of zero surface tension, employing the stability of the water-waves problem on the 
%relevant 
time-scale $O(1/\ep)$ obtained by Alvarez-Samaniego and Lannes.
\end{abstract}
%, 35Q35}
%\keywords{...}
%\thanks{...} 

% \date{April 26, 2015}
\date\today
%\date{December 1, 2015}
% start: 250314

\maketitle
%%%%%%%%%%%%%%%%%%%%%%%%%%%%%%%%%%%%%%%%%%%%%%%%%%%%%%%%%%%%%%%%%%%%%

\section{Introduction} \label{SectionIntroduction}

A significant part in the research on water waves 
%, be it gravity or capillary-gravity waves, 
is based on the study of asymptotic limits derived from the original water-waves problem, 
which is considered as providing the complete description of their behavior. 
The benefit of this method is that these reduced models have, on the one hand, a clearer structure
%, and are thus easier to treat, 
and, on the other hand,  
% that, naturally, the derived asymptotic limits 
they 
highlight 
%qualitatively 
a particular qualitative feature of the wave evolution. 
Depending on the aspect of the nature of the water waves one is interested in, one has to employ the relevant asymptotic scaling, 
obtaining in each case a different macroscopic limit. 
Considering that from the outset the original water-waves problem has some fundamental characteristics, 
the result is a plethora of different equations which are presumed to describe approximatively the behavior of water waves in different situations and regarding different aspects.
%
%However, 
While the choice of the relevant asymptotic scaling requires a thorough understanding of the initial model, from an analytical point of view the crucial question is the justification of the derived model, i.e.\ the proof that its solutions indeed are approximations of solutions to the original problem. 

Concerning the initial set-up of the water waves problem one could distinguish roughly between (a) gravity or capillary-gravity waves (the latter ones taking into account together with gravity also the surface tension as driving forces for the evolution of the waves), (b) finite- or infinite-depth water, and (c) two- or three-dimensional space, where in the former case one considers a  vertical plane in the water domain that contains the (dominant) direction of evolution,  assuming that the waves are (nearly) constant in the direction normal 
% transversal
to the plane. 
Of course, this is only a very rough classification (e.g.\ in the case of finite depth one can consider shallow or deep water, flat bottoms or bottoms with some (smooth or rough) topography, or  even moving bottoms etc.), but it seems to be the prevalent one in the mathematical-analytical literature, where the fundamental question concerns the well-posedness of the water-waves problem and in particular the existence time of its solutions. 

%Here, 
In the case of gravity water waves in two dimensions and for infinite depth, first local well-posedness results were obtained by Nalimov \cite{Nalimov74} in 1974 for small Sobolev initial data, and by Shinbrot \cite{Shinbrot76}
% in 1976 
and Kano and Nishida \cite{KanoNishida79}
%in 1979 
for analytic initial data. 
The method of Nalimov was employed to prove local well-posedness in the case of finite depth by Yosihara \cite{Yosihara82} 
%in 1982 
and by Craig \cite{Craig85}, 
%in 1985 
who obtained also first rigorous justification results of the Korteweg-de Vries (KdV) and Boussinesq 
%%% ??? 
approximations. 
However, in the case of infinite depth, the crucial breakthrough
%most significant progress 
was made by the work of S.~Wu, who presented local well-posedness results for the two- and three-dimensional cases without smallness assumptions on the initial data in \cite{Wu97}, \cite{Wu99}, 
%in the late 90's, 
%and extended them to obtain almost global  and global existence results in \cite{Wu09}, \cite{Wu11} in 2009, 2011, respectively.  
which she extended to almost global and global existence, respectively, in \cite{Wu09}, \cite{Wu11}. 
Independently, global existence in three dimensions was shown by Germain, Masmoudi, and Shatah in \cite{GermainMasmoudiShatah12}.

In the case of finite depth the first general well-posedness result  for gravity waves in three dimensions was obtained by Lannes \cite{Lannes05} in 2005. 
This result was extended in \cite{ASL08}, where the characteristic dimensionless parameters of the original water-waves problem have been worked out in order to provide a stable fundament for the derivation and justification of various asymptotic limits. 
The results obtained, are presented in more 
detail in the survey \cite{Lannes}, to which we take explicit reference in the present paper. 
%
%In particular, our justification result, Theorem \ref{JustificationNonresonantInteraction}, 
%is a consequence of the stability property of the water-waves problem on macroscopic times of order $1/\ep$. 
%
% on which also the results obtained here are established. 
%
% Finally, concerning 
Concerning the well-posedness of capillary-gravity water-waves,
 we mention exemplarily only the more recent selection 
\cite{
%Yosihara83, 
BeyerGunther98,  
SchneiderWayne02,
Iguchi01,
ChistodoulouLindblad00, 
Schweizer05, 
CoutandShkoler07, 
%AmbroseMasmoudi05, 
AmbroseMasmoudi09, 
MingZhang09, 
AlazardBurgZuily11, 
GermainMasmoudiShatah15, 
Lannes13}, 
to which we refer for more details on the various results obtained, their development and their extensions.
% EVENTUELL DIESE LISTE AUSDUENNEN
%
%%% ONE COULD ALSO MENTION THE DIFFERENT FORMULATIONS OF THE PROBLEMS: 
%%% EULERIAN COORDINATES, LAGRANGIAN COORDINATES, GEODESIC FLOW, VARIATIONAL (LAGRANGIAN OF THE FLOW),
%%% AND PROBABLY ALSO OTHERS (SEE \cite{Lannes})

As mentioned 
%in the beginning, 
above, 
for each original water-wave problem with its own characteristics, different asymptotic limits can be obtained. 
%
 %Here, the 
 The
 main distinction of the derived models is with respect to the shallowness parameter $\mu = H_0^2/L^2$ of the original equations, 
where $H_0>0$ is the water depth and  $L=1$ is the characteristic horizontal length-scale. For $\mu\ll 1$ we speak of shallow water, 
while for $\mu\approx 1$ and $\mu\ge 1$ of deep water (with the limiting case of infinite depth as $\mu\to\infty$).
This classification is not arbitrary. 
Indeed, the main difference in the behavior of water-waves in these two cases is that for increasing water depth the r\^ole of dispersion
% effects become 
becomes more dominant, see, e.g., \cite[\S 1.3]{Lannes}.
Of course within each of these two main classes of models, finer distinctions can be, and indeed are, made.
Since in the present article we consider the deep (though finite) water case $\mu\ge 1$, 
% which allows for modulation equations, 
we refrain to mention any of the various shallow water models, but refer to the survey \cite{Lannes}, which seems to give a complete account of the "state-of-the-art" in 2013 concerning their derivation and justification.
%, and which can used as a guide through the relevant literature.
%
%
However, we would like to mention the justification of the celebrated KdV-equation in the two-dimensional case in \cite{SchneiderWayne00,SchneiderWayne02} after some first results in \cite{Craig85,KanoNishida86}, and with improvements in \cite{BonaColinLannes05,Iguchi07,Iguchi08}.
% NUR EINEN IGUCHI BEHALTEN

We will address modulation equations further below, after presenting in the following the capillary-gravity water waves equations.

\

% The water waves problem; well-posedness on the right time-scale; presentation of the theorem of Lannes; note that in this form it holds for finite depth; why three pulses; nonlinearities; resonances; the role of surface tension; spaces in which to estimate the residual; energy norm;...

The capillary-gravity water waves problem of finite depth 
%$\sqrt{\mu} \in (0,\infty)$
$0 < \sqrt{\mu} < \infty$ 
with a flat bottom extending over all of $\R^d$, $d=1,2$, 
can be written in the following non-dimensionalized form, due to Zakharov \cite{Zakharov68}, Craig, C.~Sulem, P.-L.~Sulem \cite{CS93,CSS92, SulemSulem}, and Alvarez-Samaniego, Lannes \cite{ASL08, Lannes}: 
\begin{align}\label{wwe}
\partial_t U +\mcN_{\ep, \sigma} (U)  = 0, \qquad U= (\zeta,\psi)^T, \qquad \mcN_{\ep, \sigma} = ( \mcN_{\ep, \sigma}^1,  \mcN_{\ep, \sigma}^2 )^T,
\end{align}
where
\begin{align} \label{mcNsigma1}
\mcN_{\ep, \sigma}^1 (U) & = -  \mcG[\ep \zeta]\psi, 
% \qquad \mcG[\ep \zeta] \psi = {\textstyle \frac1{\sqrt{\mu}}} \mcG_{\mu,1}\big[\textstyle{ \frac{\ep}{\sqrt{\mu}}}\zeta,0\big] \psi, 
\\
\label{mcNsigma2}
\mcN_{\ep, \sigma}^2 (U) & = \zeta 
 -  \bond \nabla\cdot \Big( \frac{\nabla \zeta}{\sqrt{1+\ep^2 |\nabla\zeta|^2} } \Big) 
%+ \frac\ep2 \left( |\nabla\psi|^2 - \frac{(\mcG[\ep \zeta]\psi + \ep \nabla\zeta\cdot \nabla \psi)^2}{1+\ep^2 |\nabla\zeta|^2} \right)
 + \frac\ep2 |\nabla\psi|^2 - \frac\ep2 \frac{(\mcG[\ep \zeta]\psi + \ep \nabla\zeta\cdot \nabla \psi)^2}{1+\ep^2 |\nabla\zeta|^2}.
\end{align} 
Here, the unknown functions of time $t\in[0, T)$, $T>0$, and space $X\in \R^d$ are the surface elevation $\zeta:[0,T)\times\R^d\to\R$ 
and the trace $\psi:[0,T)\times\R^d\to\R$ of the velocity potential of the fluid at the surface.  
The scaling parameter $0<\ep\le 1$ is the steepness of the wave, i.e., the ratio of the amplitude of the surface elevation above the still water level $\zeta=0$ to the characteristic horizontal length
%$L_x = L_y =1$.
$L=1$.
%, which we assume to equal $1$.

The second term in \eqref{mcNsigma2} corresponds to the surface tension, which is essentially the mean curvature of the surface scaled by the (inverse) Bond number $\bond  = \frac{\sigma}{\rho g}$, where $\sigma,\rho, g$ are the (dimensionless) coefficients of the surface tension, the fluid density and the gravity acceleration, respectively. 
When $\sigma=0$, this term is absent and we speak of gravity water waves.

The most important term in the above formulation is the Dirichlet-Neumann operator
\begin{align}\label{DNoperator}
\mcG[\ep \zeta]\psi 
%= \frac1{\sqrt{\mu}} \mcG_{\mu,1}\big[\textstyle{ \frac{\ep}{\sqrt{\mu}}}\zeta,0\big] \psi
%&=  \frac1{\sqrt{\mu}} \mcG_{\mu,1}\big[\textstyle{ \frac{\ep}{\sqrt{\mu}}}\zeta,0\big] \psi
%\\
%& = \frac1{\sqrt{\mu}} \sqrt{1+\frac{\ep^2}\mu |\nabla\zeta|^2} \pl_{\bf n} \Phi|_{z=\frac\ep{\sqrt{\mu}}\zeta}
%\\ 
%&=  \frac1{\sqrt{\mu}}  \pl_z \Phi|_{z=\frac\ep{\sqrt{\mu}}\zeta} - \ep \nabla\zeta \cdot \nabla_X \Phi|_{z=\frac{\ep}{\sqrt{\mu}} \zeta }     
%\\ 
%&=  \frac1{\sqrt{\mu}}  \pl_z \Phi(\cdot, \frac\ep{\sqrt{\mu}}\zeta) - \ep \nabla\zeta \cdot \nabla \Phi(\cdot, \frac{\ep}{\sqrt{\mu}} \zeta)   
%\\
&=  \pl_z \Phi(\cdot, \ep \zeta) - \nabla ( \ep \zeta ) \cdot \nabla \Phi(\cdot, \ep \zeta)
= \sqrt{1 + |\nabla(\ep\zeta)|^2} \pl_{\bf n} \Phi(\cdot,\ep\zeta) 
\end{align}
where the velocity potential of the fluid $\Phi$ solves the boundary value problem for 
the Laplace equation 
\begin{equation}\label{Laplace}
\begin{cases}
\Delta_{X,z} \Phi = 0, \quad  - \sqrt{\mu} \le z \le\ep\zeta,
\\ \Phi(\cdot, \ep\zeta) = \psi, \quad \pl_z \Phi (\cdot, - \sqrt{\mu}) =0
 \end{cases}
\end{equation}
in the fluid domain at time $t\ge 0$,
$$\Om_{\ep, t} =\{(X,z)\in \R^{d+1}: -\sqrt{\mu} \le z \le \ep \zeta(t, X) \},$$
with Dirichlet data $\psi$ at the surface and Neumann boundary data at the bottom. 
With ${\bf n}$ in \eqref{DNoperator} being the upward unit normal vector at the surface $\ep\zeta$, 
we see that the Dirichlet-Neumann operator $\mcG[\ep\zeta]$ relates the Dirichlet data $\psi$ 
to the normal derivative of the potential $\Phi$ at the surface, thus justifying its name.  
In particular, the first equation of the system \eqref{wwe}, 
\begin{equation}\label{wwe1}
\pl_t \zeta - \mcG[\ep\zeta]\psi = 0,
\end{equation}
codifies the physical assumption that fluid particles at the surface stay there for all times.  
We note also that the Dirichlet-Neumann operator is linear in $\psi$ but nonlinear in $\zeta$.

The second equation of the gravity water-waves problem \eqref{wwe}  (with $\sigma = 0$),
\begin{equation}\label{wwe2}
\pl_t \psi + \zeta 
 + \frac\ep2 |\nabla\psi|^2 - \frac\ep2 \frac{(\mcG[\ep \zeta]\psi + \ep \nabla\zeta\cdot \nabla \psi)^2}{1+\ep^2 |\nabla\zeta|^2} =0,
\end{equation}
originates from the Euler equation for the fluid velocity $\nabla_{X,z} (\ep \Phi) $ of an inviscid, homogeneous, incompressible, and irrotational fluid in $\Om_{\ep, t}$ under the influence of gravity and with constant external (atmospheric) pressure at the surface. 
Integrating the Euler equation over the space variables $(X,z)$ and evaluating it at the surface $z=\ep\zeta$, one gets 
\begin{align}\label{BernoulliSurface}
& \pl_t (\ep \Phi)(\cdot,\ep\zeta)  +  \frac12 |\nabla_{X,z} (\ep \Phi)(\cdot,\ep\zeta)|^2 + \ep \zeta =0.
%\\
%\Longleftrightarrow \quad & \pl_t \psi  + \zeta - \ep \pl_z\Phi (\cdot,\ep\zeta) 
%%\pl_t \zeta 
%%\mcG[\ep\zeta]\psi
%( \pl_z \Phi(\cdot, \ep \zeta) - \nabla ( \ep \zeta ) \cdot \nabla \Phi(\cdot, \ep \zeta))
% +  \frac{\ep}2 |\nabla \Phi(\cdot,\ep\zeta)|^2 
%+  \frac{\ep}2 ( \pl_z\Phi(\cdot,\ep\zeta))^2 =0
%\\
%\Longleftrightarrow \quad & \pl_t \psi  + \zeta 
%+ \ep^2 ( \pl_z\Phi (\cdot,\ep\zeta) ) (  \nabla \zeta  \cdot \nabla \Phi(\cdot, \ep \zeta))
% +  \frac{\ep}2 |\nabla \psi - \pl_z \Phi(\cdot,\ep\zeta) \ep\nabla\zeta|^2 
%-  \frac{\ep}2 ( \pl_z\Phi(\cdot,\ep\zeta))^2 =0
%\\
%\Longleftrightarrow \quad & \pl_t \psi  + \zeta 
%%+ \ep^2 ( \pl_z\Phi (\cdot,\ep\zeta) )  \nabla \zeta  \cdot \nabla \psi 
%% - \ep ( \pl_z\Phi (\cdot,\ep\zeta) )^2 ( |\ep \nabla\zeta|^2 )
% +  \frac{\ep}2 |\nabla \psi|^2 
% %-  \ep^2 \nabla \psi \cdot \pl_z \Phi(\cdot,\ep\zeta) \nabla\zeta 
% %+  
% -
% \frac{\ep}2 ( \pl_z \Phi(\cdot,\ep\zeta))^2  |\ep\nabla\zeta|^2 
%-  \frac{\ep}2 ( \pl_z\Phi(\cdot,\ep\zeta))^2 =0
%\\
%\Longleftrightarrow \quad & \pl_t \psi  + \zeta  +  \frac{\ep}2 |\nabla \psi|^2 
%- \frac{\ep}2 ( \pl_z \Phi(\cdot,\ep\zeta))^2  (1 + |\ep\nabla\zeta|^2) =0
%\\
%\Longleftrightarrow \quad &\pl_t \psi + \zeta 
% + \frac\ep2 |\nabla\psi|^2 - \frac\ep2 \frac{(\mcG[\ep \zeta]\psi + \ep \nabla\zeta\cdot \nabla \psi)^2}{1+\ep^2 |\nabla\zeta|^2} =0
\end{align}     
By use of the chain rule on $\psi = \Phi(\cdot, \ep\zeta)$, one obtains with \eqref{DNoperator} and \eqref{wwe1} that \eqref{BernoulliSurface} is equivalent to \eqref{wwe2}. 
Thus, determining $U=(\zeta,\psi)$ via the system \eqref{wwe1}, \eqref{wwe2}, we have all required data to solve \eqref{Laplace} for $\Phi$
(under reasonable regularity assumptions on $U$ and under the condition that the flow is at rest as $|(X,z)|\to\infty$). From the Euler equation we can then determine also the pressure of the fluid.
It was Zakharov who noted in \cite{Zakharov68} that the knowledge of $U=(\zeta,\psi)$ is sufficient for solving the water-waves problem in this way, while the use of the Dirichlet-Neumann operator \eqref{DNoperator} in the formulation of the system \eqref{wwe1}, \eqref{wwe2} is mainly due to Craig, C.~Sulem and P.-L.~Sulem in \cite{CS93,CSS92}. 

%%% DAS FOLGENDE AUF WIEDERHOLUNGEN UEBERPRUEFEN UND KUERZEN
% Finally, the 
The non-dimensionalized version \eqref{wwe} of the water-waves problem, that we use for the dispersive, deep water case relevant in this article, is relying on a more general one, derived by Alvarez-Samaniego and Lannes first in \cite{ASL08} and then presented in more detail in \cite{Lannes}, which works out all characteristic parameters of the water wave problem. 
This is particularly useful for a systematic and analytically reliable derivation of all possible asymptotic limits one may be interested in. 
Since the ultimate goal of the present article is the justification (see Section \ref{SectionJustification}) of the modulation equations formally derived in Section \ref{SectionFormalDerivation}, and since our justification result (Theorem \ref{JustificationNonresonantInteraction}) follows directly from the stability property of the water-waves problem as presented by Lannes in \cite{Lannes} (see here Theorem \ref{JustificationTheorem}), we chose to study the water-waves problem from the beginning in the form \eqref{wwe}. This is also the reason for the %  --- at a first glance surely irritating --- 
(at a first glance unusual)
notation of the water-depth by $\sqrt{\mu}$. 
For a full derivation of the water-waves problem in the form \eqref{wwe}, and an extended and detailed overview of its recent analytical state of the art, we refer the reader to \cite{Lannes}.

The water-waves problem \eqref{wwe} has the linearization around $(\zeta,\psi)=(0,0)$ 
\begin{equation}\label{linearization}
\begin{cases}
\pl_t \zeta - \mcG[0]\psi = 0,
\\
\pl_t \psi + \zeta - \bond \Delta\zeta =0,
\end{cases}
\end{equation}
with $\mcG[0]\psi = \pl_z \Phi(\cdot,0)$, where $\Phi$ solves \eqref{Laplace} with Dirichlet data $\Phi(\cdot,0)=\psi$ at the surface $\zeta=0$. 
Considering the Fourier transform 
% $\mathcal F$ 
of \eqref{Laplace} with respect to the horizontal variables $X\in\R^d$, we obtain for each $\xi\in\R^d$ a second-order ODE for $\Phi(\xi,\cdot)$ in the vertical variable $z$ with boundary values at $z=-\sqrt{\mu}$ and $z=0$, which can be solved uniquely, yielding
\begin{equation*} %\label{g0}
% \widehat{\mcG[0] \psi} (\xi) =
\widehat{\pl_z\Phi} (\xi, 0) =
 %|\xi| \tanh (\sqrt{\mu} |\xi|) 
 g_0(\xi)
\widehat{\psi}(\xi)
,\quad g_0 (\xi) = |\xi| \tanh (\sqrt{\mu} |\xi|) 
 ,\quad \xi\in\R^d.   
\end{equation*}
Thus, in the Fourier-multiplier notation
\begin{equation}\label{FourierMultiplier}
\widehat{f(D) u} (\xi) = f(\xi) \widehat{u}(\xi), \quad \xi\in\R^d, \quad \text{with}\quad D = - \ii \nabla,
\end{equation} 
we obtain 
\begin{equation*}
\mcG[0]\psi =  |D| \tanh (\sqrt{\mu} |D|) \psi.
\end{equation*}
Moreover, we obtain that  \eqref{linearization} allows for plane wave solutions of the form 
\begin{equation}\label{constantplanewaves}
%\underline{\psi}
%\begin{pmatrix} 
%%\frac{\ii \underline{\om}}{1 + \bond|\underline{\xi}|^2}
%\ii \underline{\om}/(1 + \bond|\underline{\xi}|^2)
% \\ 1 \end{pmatrix} 
%% \underline{\psi}
%\ee^{\ii (\underline{\xi} \cdot X - \underline{\om} t)} +\cc, 
%\qquad \underline{\xi}\in\R^d,\quad  \underline{\om}\in \R,\quad \underline{\psi} \in \C
\begin{pmatrix} \underline{\zeta} \\ \underline{\psi} \end{pmatrix} \ee^{\ii (\underline{\xi} \cdot X - \underline{\om} t)} +\cc, 
\qquad \underline{\xi}\in\R^d,\quad  \underline{\om}\in \R,\quad  \underline{\zeta}, \underline{\psi} \in \C
\end{equation}
(with $\cc$ denoting the complex conjugate of  the preceding term(s)),
provided %the polarization condition and 
the dispersion relation 
\begin{equation}\label{dispersionrelation}
\underline{\zeta} = \frac{\ii \underline{\om}}{1 + \bond|\underline{\xi}|^2} \underline{\psi}
\qquad\text{and}\qquad
\underline{\om}^2 = \om^2(\underline{\xi}),  
\end{equation}
%are 
is satisfied, with the dispersion function
\begin{equation} \label{dispersionfunction}
\om(\xi)= \sqrt{(1+\bond|\xi|^2) g_0 (\xi) },\qquad g_0 (\xi) = |\xi| \tanh (\sqrt{\mu} |\xi|),\qquad \xi\in\R^d.  
\end{equation}

\

In the case of linear systems one can construct more complicated solutions (wave packets) by superposition of plane waves via Fourier transformation. 
The analog to this in nonlinear systems is the consideration of modulated plane waves.  
In the most simple case of amplitude modulation we replace the constants $\underline{\zeta}, \underline{\psi}\in \C$  
in \eqref{constantplanewaves} by slowly varying functions
\begin{equation} \label{modulatedplanewave}
\begin{pmatrix} \underline{\zeta} (t', X') \\ \underline{\psi} (t',X') \end{pmatrix} \ee^{\ii (\underline{\xi} \cdot X - \underline{\om} t)} +\cc, 
\qquad \underline{\zeta}, \underline{\psi}:[0,\infty)\times\R^d \to \C,
\end{equation}
where $t' = \ep t$, $X'=\ep X$ with $0<\ep\le 1$ are new, macroscopic time- and space- variables. 
The question then is whether the nonlinear system
%, i.e.\ in our case \eqref{wwe}, 
allows
%, at least approximatively, 
approximatively 
for solutions of this form.
By this we mean solutions which \emph{maintain} the above form at least locally with respect to the macroscopic time $t'\le T$, or, equivalently, for $t\le T/\ep$. 
Typically, by inserting the two-scale ansatz \eqref{modulatedplanewave} into the nonlinear system, one obtains formally the necessary conditions, viz.\ the modulation equations, which the macroscopic functions $\underline{\zeta}, \underline{\psi}$ have to satisfy.
The modulation equations reveal some qualitative, macroscopic feature in the behavior of the nonlinear system under investigation, 
which depends of course strongly on the macroscopic scaling used for the modulation. 
This approach has been used widely in the physics literature
for all sorts of dispersive systems. 
Indeed, one of the oldest fields of application have been water waves, as is exemplified prominently in the work of Whitham \cite{Whitham}, 
to which we refer for a methodical exposition of the ideas behind modulation from the physical point of view.
%, as well as for the presentation of a variational method for the derivation of modulation equations. 

In nonlinear dispersive systems the central modulation equation is the nonlinear Schr\"odinger equation (nlS), since it captures the interplay between nonlinearity and dispersion governing the deformation of the envelopes of the wave packets 
(see  e.g.\ \cite{SulemSulem} for an overview).
For this, the right two-scale ansatz is not \eqref{modulatedplanewave} but rather
\begin{equation} \label{nlSscaling}
\begin{pmatrix} \underline{\zeta} (t'', X'') \\ \underline{\psi} (t'',X'') \end{pmatrix} \ee^{\ii (\underline{\xi} \cdot X - \underline{\om} t)} +\cc, 
\qquad \underline{\zeta}, \underline{\psi}:[0,\infty)\times\R^d \to \C,
\end{equation}
with $t'' = \ep t' $, $X''=X' - 
 \nabla\underline{\om}\, 
% c_g 
t'$, where 
% $c_g$
$\nabla\underline{\om}$
is the group velocity of the wave packet.  
Inserting this ansatz (with a corresponding polarization condition for $\underline{\zeta}$) into \eqref{wwe}, one obtains the nlS equation, 
which for two-dimensional gravity waves takes the form 
\begin{equation*}
\pl_t'' \underline{\psi} - \ii \frac12 \underline{\om}'' {\pl_x''}^2 \underline{\psi} + \ii c |\underline{\psi}|^2\underline{\psi} =0  
\end{equation*}
with $c\in\R$ depending on $\underline{\xi}, \underline{\om}$ and $\pl_t''$, $\pl_x''$ denoting differentiation with respect to $t''$, $x''$.
It was derived by Zakharov \cite{Zakharov68} for infinite depth 
%for two- and three-dimensions, 
% for capillary-gravity waves
and by Hasimoto and Ono \cite{HasimotoOno72} for finite depth.
% for gravity waves
In the three-dimensional case of finite depth instead of the nlS one obtains for the scaling \eqref{nlSscaling} 
the Davey-Stewartson (DS) system \cite{DaveyStewartson74} (see \cite{SulemSulem} for a detailed discussion of its properties). 
However, in infinite depth again the nlS is obtained as the modulation equation for the scaling \eqref{nlSscaling}. 
Concerning the justification of these modulation equations, this has been achieved 
%so far only 
for the two-dimensional gravity water-waves problem
by Totz and Wu \cite{TotzWu12} in the case of infinite depth and by D\"ull, Schneider, and Wayne \cite{DullSchneiderWayne12} in the case of finite depth. 
In the three-dimensional capillary-gravity case there exist consistency results for the nlS equation \cite{CraigSulemSulem92} 
and for the DS system \cite{CraigSchanzSulem97}, where consistency means that the amount by which the approximate solution fails to satisfy the original problem (i.e.\ the residual) tends to zero in the asymptotic limit with respect to some relevant norm. 

\

In the present article we use the hyperbolic scaling \eqref{modulatedplanewave} and consider three modulated plane waves of that form. 
We are interested in the modulation equations that govern the macroscopic dynamics of these waves 
not only in leading order but also with respect to their 
%macroscopic amplitudes 
% $\underline{\zeta}, \underline{\psi}$ 
% but also of their 
macroscopic corrections of order $O(\ep)$.
For the sake of clarity, we first present our exact assumptions, and discuss them afterwards.
%in the following. 

We make the two-scale ansatz for approximate solutions of \eqref{wwe}
\begin{equation}\label{Uapp}
U_a = \begin{pmatrix}  \zetaa \\ \psia \end{pmatrix} 
%= \sum_{m=0}^2 \ep^m \begin{pmatrix} \zeta_m \\ \psi_m \end{pmatrix} 
% FALLS WIR DAS FRUEHRE n AUF 2 FIXIEREN, KOENNEN WIR STATT m AUCH n BENUTZEN; ANSONSTEN BESSER WIE FOLGT:
= \begin{pmatrix} \zeta_0 \\ \psi_0 \end{pmatrix} +  \ep \begin{pmatrix} \zeta_1 \\ \psi_1 \end{pmatrix} + \ep^2 \begin{pmatrix} \zeta_2 \\ \psi_2 \end{pmatrix} 
%, \quad n=1,2,
\end{equation}
%where $0<\ep \le \ep_0$, $t\in[0,\infty)$, $X=(x,y)\in\R^2$ and
with
\begin{align*}
& \begin{pmatrix} \zeta_0 \\ \psi_0 \end{pmatrix} 
= \sum_j \begin{pmatrix} \zeta_{0j}  \\  \psi_{0j}  \end{pmatrix}  \ee_j 
+ \cc 
+ \begin{pmatrix} \zeta_{00} \\ \psi_{00} \end {pmatrix} ,
\\
& \begin{pmatrix} \zeta_1 \\ \psi_1 \end{pmatrix}  
= \sum_j  \begin{pmatrix} \zeta_{1j} \\ \psi_{1j} \end{pmatrix}  \ee_j
+ \sum_{ji}  \begin{pmatrix} \zeta_{1ji} \\ \psi_{1ji} \end{pmatrix}  \ee_{ji}
+ \cc 
+ \begin{pmatrix} \zeta_{10} \\ \psi_{10}  \end{pmatrix} , 
\\
& \begin{pmatrix} \zeta_2 \\ \psi_2 \end{pmatrix}  
= \sum_j  \begin{pmatrix} \zeta_{2j} \\ \psi_{2j} \end{pmatrix}   \ee_j
+ \sum_{ji}   \begin{pmatrix} \zeta_{2ji} \\ \psi_{2ji} \end{pmatrix}  \ee_{ji}
+ \sum_{jik}    \begin{pmatrix} \zeta_{2jik} \\ \psi_{2jik} \end{pmatrix} \ee_{jik}
+ \cc 
+  \begin{pmatrix} \zeta_{20} \\ \psi_{20}  \end{pmatrix},
\end{align*}
defined according to the following % assumptions and 
% notations: %\\
notations and assumptions: %\\

%\noindent {\bf Assumption 1.1.}\\[.5em]
\begin{notation} \label{approximationnotation}
%
%\noindent 
(1) All functions 
$\zeta_{\ldots}, \psi_{\ldots}$  on the right of $(\zeta_n, \psi_n)^T$, $n=0,1,2$, 
%above
are complex-va\-lued and depend only on the macroscopic time- and space-variables $0 \le t' = \ep t \le T$, $X' = \ep X = \ep (x,y) \in\R^d$
($d=1,2$ and $X=x\in\R$ if $d=1$), where $0<\ep\le 1$.
Differentiation with respect to $t'$ and $X'$ is denoted by $\pl_t'$ and $\nabla'$.
The abbreviation $\cc$ denotes the complex conjugate of all preceding terms.\\[.5em]
(2)
We introduce the index-sets 
\begin{equation*}
J = \{1,2,3\},
\end{equation*}
\begin{equation*} %\label{I1} 
I=\{  (1,1),  \  (2,2),  \ (3,3),  \ (1, \pm 2), \  (1, \pm 3), \ (2, \pm 3) \} \supset \{ (1,2), (1,3), (2,3)  \} = I_<,
\end{equation*}
\begin{align*} %\label{I2}
K =\{  & (1,1,1), \ (2,2,2), \ (3,3,3),  
\\ & (1,1, \pm 2),  \ (1,1, \pm 3), \ (2,2, \pm 3),  (2,2,\pm1), \  (3,3,\pm1),\  (3,3, \pm2), 
\\ & (1,2,3),   \ (1,2,-3), \ (1,3,-2), \ (2,3,-1) \}.
\end{align*}
We denote summation over these index-sets by 
\begin{align*}
\sum_j:= \sum_{j\in J}\ ,\qquad
\sum_{ji}:= \sum_{(j,i)\in I}\ ,\qquad 
\sum_{jik}:= \sum_{(j,i,k)\in K} \ .
\end{align*}\\[.5em]
(3a)
The functions $\ee_{\pm j}$ for $j\in J$, $\ee_{ji}$ for $(j,i)\in I$ and $\ee_{jik}$ for $(j,i,k)\in K$ are defined through
\begin{align*}%\label{ej} 
\ee_{\pm j} (t,X)= \ee^{\pm \ii (\xi_j\cdot X - \om_j t)}, \qquad  \ee_{ji} = \ee_j\ee_i,\qquad \ee_{jik} = \ee_j\ee_i\ee_k,
\end{align*}
where the wave-vectors $\xi_j\in\R^d\setminus\{0\}$ and the frequencies $\om_j=\om(\bk_j)>0$ satisfy for each $j\in J$ 
the dispersion relation $\om_j^2 = \om^2(\xi_j)$ with the dispersion function \eqref{dispersionfunction}.
%\begin{equation*}%\label{deffreq}
%% \om(\xi) = \left( (1+\bond|\xi|^2) g_0(\xi)\right)^{1/2}, 
%\om(\xi) = \sqrt{ \left(1+\bond|\xi|^2\right) g_0(\xi) }, 
%\qquad g_0(\xi) = |\xi| \tanh(\sqrt{\mu}|\xi|), \qquad \xi\in\R^d.
%\end{equation*}
%Moreover, 
%We assume that all $\ee_{\pm j}$, $j\in J$, are mutually different, i.e.\ 
%\begin{equation*}
%(\xi_j,\om_j) \neq \pm (\xi_i,\om_i)\quad\forall\ j,i\in J, \ j\neq i.
%\end{equation*}
We assume that the plane waves $\ee_{j}$, $j\in J$, are mutually different, i.e.\ 
\begin{equation*}
(\xi_j,\om_j) \neq (\xi_i,\om_i)\quad\forall\ j,i\in J, \ j\neq i.
\end{equation*}
Occasionally, we will refer to $\ee_j$, $\ee_{ji}$ and $\ee_{jik}$ as the
 %"first-", "second-" and "third-order harmonics",  
 first-, second- and third-order harmonics, respectively, and to $1=\ee^0$ as the zeroth-order harmonic.\\[.5em]    
(3b) In analogy to the index-notation for the harmonics, we use %also 
the abbreviations 
\begin{alignat*}{4}
\xi_{\pm j}  & = \pm \xi_j,\qquad  & \xi_{ji} & =  \xi_j+\xi_i, \qquad   & \xi_{jik} & =  \xi_j+\xi_i+\xi_k,
\\
\om_{\pm j} & = \pm\om_j, \qquad  & \om_{ji} & =  \omj+\omi,  \qquad   & \om_{jik} & = \omj+\omi+\omk,
\intertext{and}
%\\
b_j  & = 1 + \bond   |\bk_j|^2, \qquad & b_{ji}  & = 1 + \bond   |\bk_{ji}|^2, \qquad &  b_{jik}  & = 1 + \bond   |\bk_{jik}|^2,
\\
g_j &=  g_0(\bk_j), \qquad  &  g_{ji} &=  g_0(\bk_{ji}), \qquad &  g_{jik} & =  g_0(\bk_{jik}).
 \end{alignat*}
Finally, we denote 
\begin{align*} 
 g'_j =  \nabla g_0(\bk_j), \qquad  g'_{ji} & =  \nabla g_0(\bk_{ji}), \qquad {\mathrm H}_j =  {\ts \frac 12} \nabla' \cdot \mcH_{g_0} ( \bk_j) \nabla' , 
\end{align*}
where $\mcH_{g_0}(\xi)$ is the Hessian matrix of the function $g_0$ at $\xi\in\R^d$, see \eqref{dispersionfunction}.\\
(Note, in particular, $g_0(0,0) =0$, $\nabla g_0(0,0) = (0,0)$ and $\mcH_{g_0}(0,0) = 2\sqrt{\mu} I$.)  

In this notation the following identities hold true:
\begin{equation}  \label{nablaomega}
2\om_j\nabla\om_j = b_j g'_j+\bond 2\bk_j g_j, \quad \text{where}\quad \nabla\om_j = \nabla\om(\bk_j),
\end{equation} 
%where $\nabla\om_j = \nabla\om(\bk_j)$,
and
\begin{equation} \label{Hessiannew}
%%\frac{b_j}{2\om_j}\nabla' \cdot \mcH_{g_0}(\bk_j) \nabla'  \psi_{0j}
%\frac{b_j}{\om_j} \big( {\mathrm H}_j \psi_{0j} -  \nablabo\om_j \cdot \nabla'  (  \nabla\om_j   \cdot \nabla'  \psi_{0j} ) \big) 
%\\ 
%= \nabla' \cdot \mcH_\om(\bk_j) \nabla' \psi_{0j} -  \bond  \Big( 2 \bk_j \cdot \nabla' ( \nablabo\om_j  \cdot \nabla'  \psi_{0j} ) + \frac{\om_j}{b_j}\Delta'  \psi_{0j} \Big),
\nabla' \cdot \mcH_\om(\bk_j) \nabla' \psi_{0j}  = \frac{b_j}{\om_j}  {\mathrm H}_j \psi_{0j} 
-  \frac1{\om_j}  ( b_j \nablabo\om_j \cdot \nabla' )^2 \psi_{0j} + \bond \frac{\om_j}{b_j}\Delta'  \psi_{0j}, 
\end{equation}
%with $a\cdot\nabla(b\cdot\nabla f) = b\cdot\nabla(a\cdot\nabla f)$.
where
\begin{align}\label{nablabo}
\nablabo\om_j =\frac1{b_j}\Big( \nabla\om_j  - \bond 2\frac{  \om_j}{b_j}  \bk_j \Big).
\end{align}
(3c) The plane waves $\ee_j$, $j\in J$, satisfy the 
\emph{non-resonance 
% or closedness 
conditions}
\begin{equation*}
\om_{ji}^2 \neq  \om^2 (\xi_{ji})  = b_{ji} g_{ji} %=  \om^2 (\xi_{ji}) 
\quad \forall\ (j,i)\in I
\end{equation*}
and
\begin{equation*}
\om_{jik}^2 \neq  \om^2 (\xi_{jik})  = b_{jik} g_{jik} % =  \om^2 (\xi_{jik})
\quad \forall\ (j,i,k)\in K.
\end{equation*}
%or, equivalently, 
%%or, equivalently (cf.\ \eqref{dispersionfunction}),
%\begin{equation*}
%\om_{ji}^2 \neq  \om^2 (\xi_{ji}) 
%\quad \forall\ (j,i)\in I, 
%% \quad 
%\qquad
%\om_{jik}^2 \neq \om^2 (\xi_{jik})
%\quad \forall\ (j,i,k)\in K.
%\end{equation*}
%
(4) 
For $u\in H^s (\R^d)$, $s\in \R$, see \eqref{Hs}, of the form 
\begin{equation*}%\label{formu}
u(X) = \sum_{i=1}^k  \tilde u_i (X') \ee^{\ii \xi_i \cdot X}
\end{equation*} 
we use the notations
%\begin{equation*} % \label{formu}
$
|\tilde u|_{H^s} = \sum_{i=1}^k |\tilde u_i|_{H^s}
$
%\end{equation*}
and
\begin{equation*}  \label{notationprime}
u' (X) = \sum_{i=1}^k \ii \xi_i  \tilde u_i (\ep X) \ee^{\ii \xi_i \cdot X}, 
\qquad  
u'' (X) = \sum_{i=1}^k ( - |\xi_i|^2) \tilde u_i (\ep X) \ee^{\ii \xi_i \cdot X}, 
\end{equation*}
such that (with $\Delta' = \nabla'\cdot\nabla'$) 
\begin{equation} \label{spacediffexpansion}
\nabla u  =  u' + \ep \nabla' u 
\quad\text{and}\quad  
\Delta u = u'' + 2 \ep  \nabla' \cdot u' + \ep^2 \Delta' u.
\end{equation}
\end{notation}

\

The motivation for the special form of the ansatz \eqref{Uapp} is that we want to include in our formal expansion of the capillary-gravity water-waves equation \eqref{wwe} the case of quadratic interaction of two modulated plane waves. 
By quadratic interaction we mean the situation where two such waves generate by multiplication a third plane wave through the (quadratic) resonance of their phases, e.g.\ $\ee_1\ee_2 = \ee_3$. 
In this case one has to consider from the outset all three involved modulated plane waves in order to obtain a closed system of macroscopic equations, 
and the interaction is manifested macroscopically in leading order, 
which means that an expansion up to $O(\ep)$-terms in \eqref{Uapp} would be sufficient. 
% This motivates the ansatz \eqref{Uapp}.

However, as 
will be explained below, in the case of pure gravity waves, which is our main focus in the present paper, no such quadratic resonances arise. Such resonances exist only if surface tension is included in the original water-waves equation, see \cite{SchneiderWayne03} for the two-dimensional case ($d=1$).  
Then, naturally, the question arises, whether even in this quadratically non-resonant case any macroscopic coupling can be detected in the next-to-leading-order correction of the leading order amplitudes or in the non-oscillating mean field generated by the waves. 
Wanting to perform the (unsurprisingly, very cumbersome)
% Nevertheless, in order to perform the (typically very cumbersome) 
formal expansion of the water waves equation for three modulated pulses only once,
we chose the ansatz \eqref{Uapp}, which is usefull in both cases (quadratically resonant and non-resonant) and for waves with or without surface tension. 

As expected, indeed also in the quadratically non-resonant case, the interaction of modulated waves can be traced in the second-order macroscopic system. 
%
% This turns out to be indeed the case. 
%, see \eqref{mainsystem}.
More precisely, we obtain in Section \ref{SectionFormalDerivation} that in order for the approximation $U_a$ of \eqref{Uapp} to satisfy formally the water waves equation \eqref{wwe} up to residual terms of order $O(\ep^3)$, i.e.\
\begin{equation} \label{UaWweResiduals}
\partial_t U_a +\mcN_{\ep, \sigma} (U_a)  = \ep^3 (r_2^1, r_2^2)^T, 
\end{equation}
the macroscopic modulation equations
\begin{align} \tag{\ref{mainsystem}}
\begin{cases}
\pl_t' \psi_{0j}   +  \nabla\om_j \cdot   \nabla' \psi_{0j} = 0,
\\
{\pl_t'}^2  \psi_{00} -  \sqrt{\mu} \Delta' \psi_{00}   =  \sum_j  \Big(  (g_j^2 - |\xi_j|^2) \pl_t' + 2  \frac{\om_j}{b_j} \xi_j  \cdot  \nabla'  \Big)  |\psi_{0j}|^2,
\\
\pl_t'   \psi_{1j} + \nabla\om_j \cdot \nabla' \psi_{1j}   = E_j
\end{cases}
\end{align}
with $j\in J=\{1,2,3\}$ and
\begin{align} \tag{\ref{tildeEj}}
 E_j 
= &\  \ii \frac12   \nabla' \cdot \mcH_\om(\bk_j) \nabla' \psi_{0j} 
-  \ii  \psi_{0j} \Big( \frac{ b_j }{2\om_j}  (g_j^2 - |\xi_j|^2)   \pl_t' +  \bk_j  \cdot  \nabla'\Big) \psi_{00}
+ \tilde E_j
\end{align} 
have to be satisfied, where $\tilde E_j$ consists of cubic products of the leading-order amplitudes $\psi_{0j}$, 
see \eqref{truetildeEj}.
The other macroscopic functions appearing in $U_a$ can be determined via $\psi_{0j}, \psi_{00}, \psi_{1j}$ or can be chosen arbitrarily.

Of course, more generally one could consider also an arbitrary number of $N\in\N$, $N\ge 3$, % non-resonant 
modulated plane waves, as was done for other physical settings, e.g.\ in \cite{GiannoulisSparberMielke08,Giannoulis10}. 
In order to keep the presentation more simple and explicit, we chose, however, to consider here only three pulses.  
Note that the results derived in the present paper can be used in order to obtain the (non-resonant) macroscopic dynamics up to next-to-leading order for two pulses or even for a single pulse, by equating the macroscopic coefficients of the superfluous waves to zero, cf.\ 
for two pulses e.g.\ \cite{HammackHendersonSegur05} and for a single pulse 
\cite{BR} (see also Remark \ref{SectionFormalDerivation}.3) and \cite[\S 11.1]{SulemSulem}, \cite[\S 8.2.5]{Lannes}.

The ansatz \eqref{Uapp} consists of all first- and higher-order harmonic terms expected to arise up to order $\ep^2$ due to the nonlinear nature of the water waves problem, including the non-oscillating terms for the mean field which arise from the interaction of a plane wave with its complex conjugate. Note here, that the index sets $I$ and $K$ of Notation \ref{approximationnotation}(2) represent all possibly different second- and third-order harmonics. 
In this context, we point out that actually we are interested in the weakly nonlinear approximation of 
solutions to the water-waves problem \eqref{wwe} with $\ep=1$, 
% DAS VIELLEICHT AUCH DORT ERWAEHNEN
that is to say in an approximation of the form 
% $\ep U_{a,1}$, where $U_{a,1}$ consists of the first two terms of \eqref{Uapp}. 
$\ep U_a$ with $U_a$ as in \eqref{Uapp}.
%, where $U_{a,1}$ consists of the first two terms of \eqref{Uapp}. 
%
However, the expansion of the water waves equation in Section \ref{SectionExpansion} and the derivation of the modulation equations in Section \ref{SectionFormalDerivation} are performed for the equation in the form \eqref{wwe} and the ansatz $U_a$ in \eqref{Uapp}.

The reason why the ansatz $U_a$ includes also terms of order $\ep^2$, although we are interested only in an approximation $U_{a,1}$ of up to next-to-leading-order terms of order $\ep$ (i.e.\ consisting only of the first two terms of $U_a$), 
is that the determining equations for the functions $\zeta_{1j}, \psi_{1j}$ arise at order $O(\ep^2)$, cf.\ Corollary \ref{wwexpansion2}.
Moreover, for the justification of the approximation $U_{a,1}$ over time scales of order $O(1/\ep)$  we need first to consider an approximation that satisfies \eqref{wwe} formally up to residual terms of order $O(\ep^3)$, see Section \ref{SectionJustification} for the details, and \cite{KirrmannSchneiderMielke92} for a more general presentation of this approach.

%%% IRGENDWIE IST DAS FOLGENDE NICH BESONDERS KLAR FORMULIERT, UND ZEUGT VON MANGELNDEM VERSTAENDNISS
%Our aim is to derive modulation equations that describe the macroscopic dynamics of the first two terms 
%(we denote their sum by $U_{a,1}$) on the right hand side of $U_a$ in \eqref{Uapp}. 
%%
%Moreover, we want to show that the $U_{a,1}$ with the macroscopic coefficients determined by the derived modulation equations remains close to some solution of the water-waves problem \eqref{wwe} 
%%with $\ep=1$ 
%over time-scales of order $ t=O(1/\ep)$.
%%
%This requires in particular that the residual terms, i.e.\ the terms by which our approximation fails to satisfy \eqref{wwe}, 
%are formally of order $O(\ep^3)$, see Section \ref{SectionJustification}.
%%
%In order to achieve this we have to insert into \eqref{wwe} rather the \emph{improved approximation} \eqref{Uapp}, see Section \ref{SectionFormalDerivation} for the details, and \cite{KirrmannSchneiderMielke92} for the general principle. 
%% Eliminating throughout the factor $\ep$ in $\ep U_a$ leads to the form of $U_a$ in \eqref{Uapp}. 
%%

The two non-resonance conditions of Notation \ref{approximationnotation}(3c) imply that none of the higher-order harmonic terms are plane waves (or non-oscillating). 
The first set of conditions is essential for the form of the modulation equations, see \eqref{mainsystem}, yielding that in leading order the macroscopic amplitudes $\psi_{0j}$ are just transported with the group velocity $\nabla\om_j$ of the wave packet, without any macroscopic interaction.
In the opposite case, when quadratic resonances appear, one obtains the 'three-wave-interaction equations', a coupled system of three semilinear transport equations for $\psi_{0j}$ containing for each $j\in J$ quadratic products of the other two amplitudes, according to the existing resonances, see \cite{SchneiderWayne03}. 
As stated above, for pure gravity water waves this quadratic non-resonance condition is always satisfied. This is known since the 1960s, see \cite{Phillips60}, while the existence of quadratic resonances in the case of capillary-gravity waves of infinite depth was first proven in \cite{McGoldrick65}. For a more general discussion of resonances of water waves we refer to \cite{HammackHenderson93,SchneiderWayne03} and the references given therein. 
For the sake of completeness we give  in the following Remark \ref{SectionIntroduction}.1 a short analytical proof of the non-existence of quadratic resonances for gravity water waves of finite depth.

\

{\sc Remark \ref{SectionIntroduction}.1.}\ 
% Indeed, assume 
Assume there are $\xi_1,\xi_2,\xi_1+\xi_2\in\R^d\setminus\{0\}$, such that
\begin{equation*}% \label{rc}
(\om(\xi_1) \pm \om(\xi_2))^2 = \om^2(\bk_1+\bk_2)
\quad \text{with}\quad  
%\om(\xi)= \big( |\xi| \tanh( \sqrt{\mu} |\xi| ) \big)^{1/2}.
\om(\xi)= \sqrt{ |\xi| \tanh( \sqrt{\mu} |\xi| )}.
\end{equation*}
Since $\om$ is a radial function, taking square-roots and possibly considering the opposite of some wave-vector $\xi_j$ and relabeling, 
the equation on the left can always be written in the form 
\begin{equation*} %\tag{\ref{rc}$^\prime$} \label{rcprime}
\om(\bk_1)+ \om(\bk_2) = \om(\bk_1+\bk_2), \qquad\bk_1,\bk_2\in\R^d\setminus\{0\}.
\end{equation*}
Mutiplying % this equation 
by $\mu^{\frac14}>0$ and setting
\begin{equation*}
x= \sqrt{\mu} |\bk_1|>0,
\qquad  \lambda=\frac{|\bk_2|}{|\bk_1|}>0,
\qquad 
c %=\cos\vartheta 
= \frac{\bk_1\cdot\bk_2}{ |\bk_1| |\bk_2|} \in [-1,1],
\end{equation*}
solving this equation for $\bk_1,\bk_2\in\R^d\setminus\{0\}$ is equivalent to finding roots 
%to solve this equation for $\bk_1,\bk_2\in\R^d\setminus\{0\}$ is equivalent to finding roots 
$(x, \lambda, c)\in (0,\infty)\times (0,\infty)\times[-1,1]$ 
of 
 the function
\begin{equation*}
r_0(x,\lambda, c)= 
\sqrt{ (1+\lambda^2 + 2 c \lambda) \tanh ( x \sqrt{1+\lambda^2 + 2 c \lambda}) }
- \sqrt{\lambda  \tanh (x \lambda) }
- \sqrt{\tanh x}.
\end{equation*}
%for $(x, \lambda, c)\in (0,\infty)\times (0,\infty)\times[-1,1]$.
%
But $r_0(x,\lambda, c) \le  r_0(x,\lambda, 1)$ 
% for $c\in[-1,1]$ 
and  
\begin{align*}
r_0(x,\lambda,1) = \frac{g \left(x(1{+}\lambda)\right) - g(x\lambda)}{\sqrt{x}} 
- \sqrt{\tanh x},
\qquad \text{where}\quad g(y) = \sqrt{y \tanh y},
\end{align*}
with  $r_0(x,0,1)=0$ and 
\begin{align*}
\frac{d}{d\lambda} r_0 (x,\lambda, 1) = \sqrt{x} \left( g'\left(x(1{+}\lambda)\right) - g'(x\lambda) \right) <0 \quad \forall\  \lambda \ge 0,
\end{align*}
the latter due to the strict decreasing of 
\begin{align*}
g'(y) = \frac{\tanh y + y(1-\tanh^2 y) }{2 \sqrt{y\tanh y}}, \quad y\ge 0.
\end{align*}
Hence,  we conclude 
%by stating that 
%\begin{equation*}
 $
r_0(x,\lambda, c) < 0
$
 %\qquad \forall \   
 for all $
 (x, \lambda, c)\in (0,\infty)\times (0,\infty)\times[-1,1].
 $
%\end{equation*}
\hfill$\square$

\

While the first (quadratic) non-resonance condition of Notation \ref{approximationnotation}(3c) is essential for the form of the derived modulation equations, the second (cubic) non-resonance condition is much less so. Indeed, if the first non-resonance condition holds true,
cubic resonances do not change the form of the modulation equations, but merely contribute additional cubic products of the leading order amplitudes $\psi_{0j}$ to the source term of one of the equations for the next-to-leading order amplitudes $\psi_{1j}$, if a third-order harmonic is equal to one of the three considered plane waves. (Note, that in general if plane waves are generated through resonant interaction one always has to include the generated wave in the original ansatz, here \eqref{Uapp}, in order to obtain a closed system of modulation equations.) 
For the sake of simplicity we do not consider these cases explicitly here, but prefer to impose the cubic non-resonance condition of Notation \ref{approximationnotation}(3c) instead. 
Nevertheless, since the same justification result holds true for these modified modulation equations, the present paper covers completely the justification of the modulation equations up to next-to-leading order for three weakly amplitude-modulated gravity water waves, provided possibly existing cubic resonances generate only one of the three plane waves considered. 

\

We close this introduction by outlining the structure of the article and commenting on its main results.  
%
% WAS HIER STEHT, SOLLTE EVENTUELL NICHT MEHR BEIM BEGINN DER EINZELNEN SEKTIONEN STEHEN
% ZU JEDEM ABSCHNITT AUCH DIE BESONDEREN BEMERKUNGEN SCHREIBEN, Z.B. APPENDIX, BR, SECULAR GROWTH, JUSTIFICATION FOR THE FIRST ORDER EQUATIONS, CONTRIBUTIONS OF LANNES
%
In the following Section \ref{SectionExpansion}, after inserting the ansatz \eqref{Uapp} for the approximation $U_a$ into the water waves equation \eqref{wwe}, we expand with respect to the steepness parameter $0<\ep\ll 1$ up to residual terms of formal order $O(\ep^3)$, for which we give estimates in the Sobolev norms used for the justification of the modulation equations \eqref{mainsystem} in Section \ref{SectionJustification}. 
% LANNES ERWAEHNEN ? VIELLEICHT BESSER FUER DAS GANZE ZUSAMMEN
%
The precise formulas for the more involved, though structurally simple, macroscopic coefficients are given in an Appendix.
Thus, as a byproduct, we provide a complete formal explicit expansion including all terms of order $\ve^2$ for three amplitude modulated plane waves with the hyperbolic scaling $t' = \ep t$, $X' = \ep X$ for the capillary-gravity water waves problem of finite depth, 
that can be used independently, also for only one or two waves.

Then, in Section \ref{SectionFormalDerivation}, we derive %formally 
the necessary conditions on the macroscopic coefficients of $U_a$ in order for the latter to satisfy \eqref{wwe} up to the residual terms of order $O(\ep^3)$, see \eqref{UaWweResiduals}. 
In particular we obtain the modulation equations \eqref{mainsystem}. 
%
% \
%
% \emph{(tbc)}
%
Finally, in Section \ref{SectionJustification} we justify the derived modulation equations as a macroscopic limit to the gravity water waves problem 
\eqref{wwe} with $\ep=1$, viz.\
\begin{equation*}
\begin{cases}
\ds \partial_t \zeta - \mcG[\zeta]\psi = 0,  
\\
\ds \partial_t \psi + \zeta 
+ \frac12 |\nabla\psi|^2 - \frac{( \mcG[\zeta]\psi + \nabla\zeta\cdot \nabla \psi)^2}{ 2 (1 + |\nabla\zeta|^2)} = 0,
\end{cases}
% \qquad  U(0) = \ep U^0,
\end{equation*} 
over the macroscopic time $T/\ep$, that is to say, we show that the approximation $\ep U_{a,1}$, which consists of only the first two terms on the right hand side of \eqref{Uapp}, with the macroscopic functions given through the solutions of \eqref{mainsystem} up to the time $T>0$,
maintain a distance of order $\ep^{3-d/2}$ to the solution $U$ of the original problem over this time interval and with respect to a suitable Sobolev norm, 
% provided this is true for the initial time $t=0$ 
if their distance is of this order at the initial time $t=0$
(for the precise result, see Theorem \ref{JustificationNonresonantInteraction}).
Here, the reduction of the order, compared to the formal one, is due to the scaling of the macroscopic time and space variables.
Note, that the approximation $\ep U_{a,1}$ contains  terms of order $\ep$ and $\ep^2$, and hence the obtained result is clearly more valuable in the one-dimensional case $d=1$, 
% establishing the validity of 
fully justifying 
% the macroscopic correction to the leading-order transport equations, 
% which shows 
the macroscopic interaction in next-to-leading order for three 
% quadratically non-resonant 
weakly modulated 
gravity water waves. 
%
% For $d=2$, 
%
In order to improve the result in the case $d=2$, one could try to adapt methods used in nonlinear optics (see, e.g., \cite{JolyMetivierRauch98,ColinLannes09}, as pointed out in \cite[fn.\ 10, p.\ 232]{Lannes}), 
which is left open here for future consideration. 
%
%%% 
%%% ICH WUERDE ALS NACHFOLGEPAPERS ZUNAECHST EHER DEN CAPILLARY-GRAVITY CASE MIR VORNEHMEN, 
%%% UND DANN INFINITE DEPTH, SOWIE BR, UND ZULETZT DAS. 
%

As already mentioned, our justification result of Theorem \ref{JustificationNonresonantInteraction} is in principle an application of the
well-posedness 
% stability 
result of Alvarez-Samaniego and Lannes \cite{Lannes05, ASL08, Lannes} on gravity water waves of finite depth for times of order $O(1/\ep)$,
which is exactly the hyperbolic time-scale of the macroscopic limit considered in the present paper.
This result was extended to the case of two-fluid interfaces with surface tension in \cite{Lannes13},
which contains as a special case capillary-gravity water waves, see also \cite[Ch.\ 9]{Lannes}.
However, in this case the energy norm used includes also higher-order time derivatives. 
Thus, we chose to treat in the present paper only the justification of the macroscopic interaction equations \eqref{mainsystem} 
in the case without surface tension, postponing to future work the treatment of the capillary-gravity case, 
which anyway allows also for quadratic resonances, as explained above.
Note, that in the resonant, one-dimensional case $d=1$ of  finite depth, 
the leading-order macroscopic 'three-wave interaction equations' have been justified by Schneider and Wayne in \cite{SchneiderWayne03}, using Lagrangian coordinates. 
Moreover, it is expected that an analogous approach as the one presented here, can be performed also for the case of infinite depth.
%
% Concluding, we point out that the justification result obtained here, uses the stability of the original water-waves problem.  

Concluding, we would like to mention that the macroscopic limit \eqref{mainsystem} derived and justified in the present paper is an alternative to a three-wave generalization of the Benney-Roskes system \cite{BR}, see Remark \ref{SectionFormalDerivation}.3.
The latter is the relevant one with respect to the longer dispersive time-scale $t''=\ep t' = \ep^2 t$. 
However, on such a long time-scale there exist no well-posedness results up to now, neither for the original water-waves problem of finite depth nor for the Benney-Roskes system,
in contrast to the situation here, where both the original and the derived models are well-posed on the relevant hyperbolic time-scale 
$t'=\ep t$.

%%% maybe after the theorem: 
% comment on the specificities of theorem 4.2, e.g. that the norms are different  

% cubic resonances in theorem 4.2

\

\emph{Acknowledgement:} 
I thank Walter H.~Aschbacher for bringing this problem to my attention and for the stimulating discussions on water waves.

%%%%%%%%%%%%%%%%%%%%%%%%%%%%%%%%%%%%%%%%%%%%%%%%%%%%%%%%%%%%%
\section{Two-scale expansion and estimates for the residuals} \label{SectionExpansion}

The formal derivation of the modulation equations consists in plugging the ansatz \eqref{Uapp} 
% for an approximate solution 
into the original water waves problem \eqref{wwe}, expanding with respect to $\ep$, and equating all terms up to order $\ep^2$ to $0$.
In the present section we perform the first two steps by writing out the terms up to order $\ep^2$ and 
giving estimates with respect to the $H^s(\R^d)$-norm $|\cdot|_{H^s}$, see \eqref{Hs}, for the residual terms of formal order $O(\ep^3)$.
The third step, i.e.\ the actual derivation of the modulation equations, is performed in the next section.

%%  
%The outcome of this procedure are certain necessary conditions that this ansatz has to fullfill in order for the approximate solution to satisfy the original problem up to residual terms of order $\ep^3$, i.e., 
%\begin{equation}\label{formalconsistency}
%\partial_t U_a +\mcN_{\ep,\sigma} (U_a) = \ep^{n+1} r_n, \qquad r_n=(r_n^1,r_n^2)^T,   \qquad n=1,2.
%\end{equation}
%In the case where the approximate solution has the form of \eqref{Uapp} these necessary conditions take the form of modulation equations for the macroscopic coefficients of the plane wave solutions satisfying the linearization of the problem. 

% POSSIBLY TO REFORMULATE (DONE 19.9.15)

Concerning the time- and space-derivatives of the approximation $U_a = (\zetaa,\psia)^T$ in \eqref{Uapp},  
we obtain immediately their expansions with respect to $\ep$, viz.\
\begin{align}\label{expansiontimezeta} 
 \pl_t \zetaa
= & \sum_j  (- \ii \omj ) \zeta_{0j} \ee_j + \cc 
\displaybreak[0]\\&\notag 
+ \ep \Big( \sum_j   ( \pl_t'  \zeta_{0j} - \ii \omj  \zeta_{1j} )  \ee_j  + \sum_{ji}  (- \ii \om_{ji} )  \zeta_{1ji}  \ee_{ji} + \cc + \pl_t' \zeta_{00} \Big)
\displaybreak[0]\\ & \notag 
+ \ep^2 \Big( 
 \sum_j   ( \pl_t' \zeta_{1j} - \ii \omj  \zeta_{2j}  ) \ee_j
+  \sum_{ji}  ( \pl_t'   \zeta_{1ji} - \ii \om_{ji}    \zeta_{2ji} )  \ee_{ji} 
\\& \notag \hspace*{2.5em} + \sum_{jik}  (- \ii \om_{jik} )  \zeta_{2jik}  \ee_{jik} 
+ \cc + \pl_t'  \zeta_{10}  \Big)  
\displaybreak[0] \\& \notag 
+ \ep^3 \Big( \sum_j   \pl_t' \zeta_{2j}  \ee_j + \sum_{ji} \pl_t' \zeta_{2ji}  \ee_{ji}  + \sum_{jik} \pl_t'   \zeta_{2jik}  \ee_{jik}  + \cc +  \pl_t'  \zeta_{20}   \Big),
\end{align}
% and analogously for $\pl_t \psia$, 
where we recall that $\pl_t'$ denotes differentiation with respect to %the macroscopic time variable 
$t'=\ep t$, 
and
\begin{align*} 
 \nabla \zetaa
= &  \sum_j   \ii \bk_j \zeta_{0j} \ee_j + \cc 
\displaybreak[0]\\&
+ \ep \Big( \sum_j   ( \nabla' \zeta_{0j} + \ii \bk_j  \zeta_{1j} )  \ee_j + \sum_{ji}   \ii \bk_{ji}   \zeta_{1ji}  \ee_{ji} + \cc + \nabla' \zeta_{00} \Big)
\displaybreak[0]\\ &
+ \ep^2 \Big(  \sum_j   ( \nabla' \zeta_{1j} + \ii \bk_j  \zeta_{2j}  ) \ee_j  +  \sum_{ji}  ( \nabla'   \zeta_{1ji} +  \ii \bk_{ji}    \zeta_{2ji} )  \ee_{ji} 
\\& \hspace*{2.5em} + \sum_{jik}  \ii \bk_{jik}   \zeta_{2jik}  \ee_{jik} + \cc + \nabla' \zeta_{10}  \Big)  
\displaybreak[0]\\& 
+ \ep^3 \Big( \sum_j   \nabla' \zeta_{2j}  \ee_j + \sum_{ji} \nabla' \zeta_{2ji}  \ee_{ji} + \sum_{jik} \nabla' \zeta_{2jik}  \ee_{jik} + \cc +  \nabla' \zeta_{20}   \Big)  ,
\end{align*}
where $\nabla'$ denotes differentiation with respect to %the macroscopic space variable 
$X' = \ep X = \ep (x,y)$.
Analogous expansions hold true for $\pl_t\psia$ and $\nabla\psia$.

 \

In order to obtain an expansion in terms of $\ep$ for $\mcN_{\ep, \sigma}(U_a)$ in \eqref{wwe}, given through \eqref{mcNsigma1}, \eqref{mcNsigma2}, we obviously need first of all an expansion of  
% $\mcG[\ep\zetaa]\psi_a =  {\textstyle \frac1{\sqrt{\mu}}} \mcG_{\mu,1}\big[\textstyle{ \frac{\ep}{\sqrt{\mu}}}\zetaa,0\big] \psia$.
$\mcG[\ep\zetaa]\psi_a$, defined by \eqref{DNoperator}.
For this, we can rely on the expansions of the Dirichlet-Neumann operator 
given in \cite[Lemmata 8.11, 8.12]{Lannes}, which we adapt to the present situation.  
\begin{proposition} \label{LannesResiduals}
\begin{enumerate} 
\item \label{propGexpansion}
Let $\zeta,\psi$ be of the form in Notation \ref{approximationnotation}.(4) 
with $\tilde\zeta_i\in H^{s +1+ t_0}(\R^d)$, $\tilde \psi_i \in H^{s+1}(\R^d)$, 
where $s\ge 0$, $t_0>d/2$, 
% and $s\vee t_0 =\max\{s, t_0\}$, 
and
\begin{equation*}
1 - \ep |\zeta|_\infty \ge h_{\min} > 0, \qquad 0< \ep \le 1.
\end{equation*}
Then, for $\mcG[\ep\zeta]\psi$ given in \eqref{DNoperator} with $1\le \mu \le \mu_{\max}<\infty$, we have
\begin{align}\label{Gexpansion1}
\mcG[\ep \zeta] \psi  
& =  \mcG_0 \psi  + \sum_{m=1}^n  \ep^m \mcG_m[\zeta] \psi  + \ep^{n+1} \mcR_n[\zeta]\psi,  \qquad n=0,1,2,
%%% caution: $\ep^n = \ep^{n+1- d/2}$ for $d=2$
\end{align}
with 
\begin{align*}
\notag
\mcG_0 \psi 
& = \mcG[0]\psi =  |D| \tanh(\sqrt{\mu}|D|))\psi,
\displaybreak[0] \\
\mcG_1[\zeta] \psi 
& = - \mcG_0(\zeta \mcG_0 \psi) - \nabla\cdot (\zeta \nabla\psi),
\displaybreak[0] \\
\mcG_2[\zeta] \psi 
& =   \mcG_0 (\zeta \mcG_0 (\zeta\mcG_0\psi) ) + {\ts\frac12}\Delta(\zeta^2\mcG_0\psi) + {\ts\frac12} \mcG_0(\zeta^2\Delta\psi)
\end{align*}
and
\begin{align} \label{mcRn}
|\mcG[\ep\zeta] \psi  |_{H^s} \le \ep^{-d/2} M(s \vee t_0 +1, \tilde \zeta)  |\tilde \psi |_{H^{s+1}},
\end{align}
\begin{align} \label{mcRn_new}
|\mcR_n[\zeta]\psi|_{H^s}  \le \ep^{-d/2} M(s+1 + t_0, \tilde \zeta)  |\tilde \psi |_{H^{s+1}},
\end{align}
where $s\vee t_0 =\max\{s, t_0\}$ and 
\begin{equation*}
M(s, \tilde \zeta) =  C \big( h_{\min}^{-1} , \mu_{\max}, |\xi_i|,  |\tilde \zeta|_{H^s} \big)
\end{equation*}
is a nondecreasing function of each of its arguments.
%%% DAS MUSS MAN SICH NOCH GENAUER ANSCHAUEN !!!!!
%
%\begin{equation}
%|\mcR_n[\zeta]\psi|_{H^s} \le \ep^{-d/2} M(s+1, \tilde \zeta)  |\tilde \psi |_{H^{s+1}}, \qquad n=0,1,2.
%% HIER BRAUCHT MAN NOCH DIE ABHAENGIGKEIT VON \ep_{\max}
%\end{equation}
%
\item \label{propG0expansion}
Let $u$ be as in Notation \ref{approximationnotation}.(4) with $\tilde u_i \in H^{s+n+1}(\R^d)$ and $\mcG_0$ as above. 
Then, with the Notation \ref{approximationnotation}.(3b), we have
\begin{equation}  \label{G0expansion}
\mcG_0 u =    \sum_{m=0}^n \ep^m G_{m} u + \ep^{n+1} R_n u , \qquad  n=0,1,2, 
%%% caution: $\ep^n = \ep^{n+1- d/2}$ for $d=2$
\end{equation}
with
\begin{align*}
G_{0} u (X) & =   \sum_{i=1}^k   g_0(\bk_i ) \tilde u_i (X') \ee^{\ii \xi_i \cdot X},
\displaybreak[0]\\
G_{1} u (X) & = - \ii \sum_{i=1}^k  \nabla g_0(\bk_i)  \cdot \nabla' \tilde u_i(X') \ee^{\ii \xi_i \cdot X},
\displaybreak[0]\\
G_{2} u (X) & = - \frac 12 \sum_{i=1}^k    \nabla' \cdot \mcH_{g_0} ( \bk_i) \nabla'  \tilde u_i (X')  \ee^{\ii \xi_i \cdot X}
\end{align*}
%%% DIE XI_I HIER SIND NICHT UNBEDINGT DIE DREI WELLENVEKTOREN DER BETRACHTETEN EBENEN WELLEN, 
%%% UND KOENNEN DAHER NICHT WIE IN  Notation \ref{approximationnotation}.(3b) ABGEKUERZT WERDEN
and
\begin{equation*}
|R_n u|_{H^s} \le \ep^{-d/2} C(\mu_{\max}, |\xi_i| ) |\tilde u|_{H^{s+n+1}}.
\end{equation*}
\end{enumerate}
\end{proposition}
\begin{proof}
The estimate \eqref{mcRn} is obtained through the identification
\begin{equation}\label{identification}
\mcG[\ep\zeta]\psi =  {\textstyle \frac1{\sqrt{\mu}}} \mcG_{\mu,1}\big[\textstyle{ \frac{1}{ \sqrt{\mu} } } \ep \zeta,0\big] \psi, 
\qquad 1\le \mu \le \mu_{\max} <\infty, 
\end{equation}
% DAS MUSS MAN NATUERLICH GENAUER EINFUEHREN UND ERKLAEREN
from \cite[Theorem 3.15~(1)]{Lannes} and \eqref{Pmu_above}, in the form
\begin{align} \label{mcGestimate}
|\mcG[\ep\zeta] \psi  |_{H^s} \le C \left( h_{\min}^{-1} ,  \mu_{\max} , | \ep \zeta|_{H^{s\vee t_0 + 1}} \right)  |\nabla \psi |_{H^s},
\end{align}
%Moreover, it can be shown that the appearing norm on $\zeta$ can be replaced by the same norm on $\ep \zeta$, provided $(\sqrt{\mu}h_{\min})^{-1}$ remains bounded.
%%
%This is possible, since $\mcG$ is independent of the diffeomorphism $\Sigma: {\mathcal S} \to \Omega$ (see \cite[Lemma 3.1]{Lannes}),  by choosing the regularizing diffeomorhism 
%\begin{equation}\label{regularizing}
%\Sigma (X,z) = \left(X, z + {\ts \frac{\ep}{\sqrt{\mu}} } (z+1) \chi (\delta z |D| ) \zeta (X,z) \right) \quad\text{for}\quad (X,z)\in {\mathcal S},
%\end{equation} 
%where $\chi:\R\to\R$ is a positive, compactly supported, smooth, even function equal to $1$ in a neighborhood of the origin, and 
%\begin{equation*} 
%\delta = \min\left\{ 1 , \frac{\mu^{3/2} h_{\min}}{2 C(\chi) | \ep \zeta|_{H^{t_0+1}}} \right\},\quad C^2(\chi) = |\chi'|_\infty^2 \int_{\R^d} (1+|\xi|^2)^{-t_0} d\xi.
%\end{equation*}
%For the details of this procedure we refer to \cite[\S 2.2.2]{Lannes}. 
together with  the estimate 
\begin{equation} \label{ModulationEstimate}
% |a (\ep \cdot) \ee^{\ii \xi \cdot \cdot }|_{H^s} \le C (|\xi|,\ep) \ep^{-d/2} |a|_{H^s}\quad\text{for}\quad s\ge 0  
|u|_{H^s} \le C (|\xi_i|) \ep^{-d/2} |\tilde u|_{H^s}, \qquad s\ge 0, \qquad 0<\ep \le 1,   
\end{equation}
% leads to the first estimate in \eqref{mcRn}, since $\zeta$, $\psi$ are of the form in Notation \ref{approximationnotation}.(4) and $d=1,2$.
for functions as 
% of the form  
in Notation \ref{approximationnotation}.(4),  
% leads to the first estimate in \eqref{mcRn},
exploiting $d=1,2$.

Similarly, the expansion \eqref{Gexpansion1} 
(based on a Taylor-expansion of $\mcG[\ep \zeta]\psi$ around $\zeta=0$ in the direction $\zeta$ and on the analyticity of the Dirichlet-Neumann operator) 
and \eqref{mcRn_new} 
%  for $\zeta$, $\psi$ as in Notation \ref{approximationnotation}.(4)
follow by 
% the identification 
\eqref{identification} and \eqref{ModulationEstimate}
from \cite[Proposition 3.44 (for $k=1$)]{Lannes}, see also \cite[Remark 3.47 and Lemma 8.11]{Lannes}.
We require here a higher regularity of $\tilde\zeta$,
in line with \cite{CraigSulemSulem92, CraigSchanzSulem97} 
(see, in particular, \cite[Theorem 4.7]{CraigSchanzSulem97} for $s\in\N_0$ and $d=2$), in order to obtain \eqref{mcRn_new} in a more straightforward manner.
%
% DAS IST DER ERSTE PUNKT, DER DIE AN \zetaa GEFORDERTE REGULARITAET ERHOEHT (REGULARITY)
%
% However, these stronger assumptions do not change the regularity required in our justification result.
%
% I think they change the regularity. 
% 
%%% DES RAETSELS LOESUNG: MAN MUSS IN LANNES BEI DEN BEWEISEN DIE L\INFTY NORM NEHMEN

The second point follows from 
% The proof of \eqref{propG0expansion} follows from is a straightforward adaptation of 
\cite[Lemma 8.12]{Lannes}, 
see also \cite{CraigSulemSulem92, CraigSchanzSulem97}.
% DIESEN PUNKT HABE ICH AUCH NICHT HUNDERTPROZENTIG GECHECKT (REGULARITY)
\end{proof}

 % VIELLEICHT NUETZLICH
%\begin{align*} 
%(\zeta,\psi) = (\zetaa,\psia) = ( \zeta_0+ \ep \zeta_1+ \ep^2 \zeta_2,\psi_0+ \ep \psi_1+ \ep^2 \psi_2)
%\end{align*}

We use the previous proposition in order to expand $\mcN_{\ep,\sigma} (U_a)$ with respect to $\ep$, 
writing out explicitly the terms of orders up to $\ep^2$ 
and providing $H^s$-norm estimates for the residual terms of formal order 
$O(\ep^3)$. 
\begin{corollary} \label{epexpansion}
For $\mcN_{\ep,\sigma}(U_a)$ of \eqref{mcNsigma1},  \eqref{mcNsigma2},  
%  with $U_a$ as in 
\eqref{Uapp}, 
the notations of Notation \ref{approximationnotation} and Proposition \ref{LannesResiduals} with $t_0=3/2$, $s\ge 1$, 
and $\mfP$ as in \eqref{mfP}, 
we have
\begin{align*} %\label{Gexpansion5}
\mcG[\ep \zetaa]\psia = & \, G_0 \psi_0 
\displaybreak[0] \\&\, 
+ \ep \big( G_1 \psi_0 + G_0 \psi_1  - G_0(\zeta_0 G_0 \psi_0)  - \zeta_0'  \cdot \psi_0'  -  \zeta_0  \psi_0''  \big) 
\displaybreak[0] \\ &\, 
+ \ep^2 \big( G_2 \psi_0  +  G_1 \psi_1  + G_0 \psi_2  - G_1 (\zeta_0 G_0 \psi_0)  - G_0(\zeta_0 G_1 \psi_0) 
\\& \phantom{ \, + \ep^2 \big( \ \,  } 
-  G_0(\zeta_1G_0 \psi_0)   - G_0 (\zeta_0 G_0 \psi_1)   +  G_0 ( \zeta_0  G_0 ( \zeta_0  G_0 \psi_0) )   
\\&\phantom{\, + \ep^2 \big( \  \, } 
- \zeta_0'  \cdot \nabla'\psi_0  -  \nabla' \zeta_0 \cdot  \psi_0'  -  2 \zeta_0 \nabla' \cdot \psi_0' 
 \\&\phantom{ \, + \ep^2 \big( \  \, } 
 -  \zeta_1'  \cdot \psi_0'  -  \zeta_1  \psi_0''   -  \zeta_0'  \cdot \psi_1'  -  \zeta_0  \psi_1''   +  {\ts\frac12} (  \zeta_0^2  G_0 \psi_0)''   +  {\ts\frac12} G_0 ( \zeta_0^2 \psi_0'') \big)
\displaybreak[0]\\&\, 
+ \ep^3 R_2^1
\end{align*}
with 
%\begin{equation*}
%R_2^1 = P^a_2 + R^a_2 + \mcR^a_2.
%\end{equation*}
%for which we obtain
 \begin{align} \label{estimateR21}
 |R_2^1|_{H^s} \le   \ep^{-d/2} M(s+5/2, \tilde\zetaa) \big( |\tilde\psi_0|_{H^{s+3}}+| \tilde \psi_1 |_{H^{s+2}} + | \tilde \psi_2 |_{H^{s+1}}   \big)
 \end{align}
and
\begin{align*} % \label{expansionNsigma2} 
- \mcN_{\ep, \sigma}^2 (U_a) 
 = &\, - \zeta_0 + \bond  \zeta_0'' 
\displaybreak[0] \\&\,
+ \ep \big( -  \zeta_1 -  {\ts\frac12} |\psi_0' |^2 +  {\ts \frac12} (G_0\psi_0)^2  +  \bond  ( 2 \nabla'\cdot \zeta_0' +  \zeta_1'' ) \big) 
\displaybreak[0] \\ &\,
+ \ep^2 \big( - \zeta_2  -  \ \psi_0' \cdot (\nabla'\psi_0 + \psi_1' ) 
\\ &\phantom{\, + \ep^2 \big( \  \, }  
+  \big( G_1\psi_0  + G_0 \psi_1 - G_0(\zeta_0 G_0 \psi_0) 
% - \zeta_0'  \cdot \psi_0' 
 -  \zeta_0  \psi_0''  
% +  \zeta_0'\cdot \psi_0'  
\big) G_0\psi_0 
\\ &\phantom{\, + \ep^2 \big(  \ \, } 
+  \bond \big( \Delta'\zeta_0 + 2 \nabla'\cdot \zeta_1'   +  \zeta_2'' - { \ts \frac12 } ( |\zeta_0'|^2 \zeta_0')' \big)
 \big)
\displaybreak[0] \\ &\,
+ \ep^3 R_2^2 
\end{align*}
with
%\begin{equation*}
%R_2^2 = - {\ts \frac12}  \rho^a_1 +  { \ts \frac12 } r^a_1  +  \bond   s^a_2. 
%\end{equation*}
%and hence
\begin{equation}\label{estimateR22}
% | \nabla R_2^2 |_{H^s} 
| \mfP R_2^2 |_{H^s} 
\le \ep^{-d/2}  M(s+3,\tilde\zetaa) \left( C \big( | \tilde \psi_0|_{H^{s+7/2}},  | \tilde \psi_1|_{H^{s+5/2}},  |\tilde \psi_2|_{H^{s+3/2}} \big) +  \bond \right). 
%%% DIE ZUSAETZLICHE ABLEITUNG KOMMT DADURCH ZUSTANDE, DASS MAN BEI r_a^1 DIE RESIDUEN-TERME, 
%%% DIE NICHT DIE STRUKTURIERTE FORM HABEN, "AUFMACHT", D.H. UM EINS ENTWICKELT. 
%%% DAS KANN MAN WAHRSCHEINLICH BESSER MACHEN 
\end{equation}
\end{corollary}
%%%%%%%%%%%%%%%%%%%%%%%%%%%%%%%%%%%%%%%%%%%%%%%%%%%%%%%%%%%%%%%%%%%
\begin{proof}
We start with the expansion of $\mcG[\ep \zetaa]\psia$.  
First, we use \eqref{Gexpansion1} and obtain
\begin{align*}% \label{Gexpansion3}
\mcG[\ep \zetaa]\psia = 
& \,  \mcG_0 \psi_0  + \ep \big(  \mcG_0 \psi_1 - \mcG_0(\zeta_0 \mcG_0 \psi_0)  - \nabla\cdot (\zeta_0 \nabla\psi_0) \big)
\displaybreak[0] \\&\,  + \ep^2 \big(  \mcG_0 \psi_2 -  \mcG_0(\zeta_1 \mcG_0\psi_0)  - \mcG_0(\zeta_0 \mcG_0 \psi_1)  -  \nabla\cdot (\zeta_1 \nabla\psi_0 + \zeta_0 \nabla\psi_1)  
\\ &\, \phantom{ +\ep^2 \big( \ \,}   
+   \mcG_0 ( \zeta_0  \mcG_0( \zeta_0  \mcG_0\psi_0) ) +  {\ts\frac12}\Delta(  \zeta_0^2  \mcG_0\psi_0) +  {\ts\frac12} \mcG_0( \zeta_0^2  \Delta\psi_0)  \big)
\displaybreak[0] \\&\, 
+\ep^3 \mcR^a_2
\end{align*}
with
\begin{align*}
\mcR^a_2 = &\,  \mcR_0[\zetaa]\psi_2 + \mcR_1[\zetaa]\psi_1  + \mcR_2[\zetaa]\psi_0 +\mcG_1[\zeta_1+\ep\zeta_2]\psi_1 +\mcG_1 [\zeta_2] \psi_0
\displaybreak[0]\\ & \,
 + \mcG_0 ((\zeta_1+\ep\zeta_2) \mcG_0 (\zeta_0\mcG_0\psi_0) )  + \mcG_0 (\zetaa \mcG_0 ((\zeta_1+\ep\zeta_2)\mcG_0\psi_0) ) 
\displaybreak[0]\\& \,
+ {\ts\frac12}\Delta( (\zeta_1+\ep\zeta_2) ( \zeta_0  + \zetaa)   \mcG_0\psi_0) + {\ts\frac12} \mcG_0( (\zeta_1+\ep\zeta_2) ( \zeta_0  + \zetaa)   \Delta\psi_0 ).
 \end{align*}
%%%%%%%%%%%%%%%%%%%%%%%%%%%%%%%%%%%%%%%% %%%%%
% and hence 
% \begin{align*}
% |\mcR^a_2|_{H^s} \le   \ep^{-d/2} M(s+2, \tilde\zetaa) \big( |\tilde\psi_0|_{H^{s+3}}+| \tilde \psi_1 |_{H^{s+2}} + | \tilde \psi_2 |_{H^{s+1}}   \big)
% \end{align*}
% where we used ...
 %%%%%%%%%%%%%%%%%%%%%%%%%%%%%%%%%%%%%%%%%%%%%
 Then, we expand $\mcG_0$ 
 % appearing in $\mcG[\ep \zetaa]\psia$ 
 according to \eqref{G0expansion}, getting
 %. Thus, we get
\begin{align*} %\label{Gexpansion4}
\mcG[\ep \zetaa]\psia = &\,  G_0 \psi_0 + \ep \big( G_1 \psi_0 + G_0 \psi_1 - G_0(\zeta_0 G_0 \psi_0) - \nabla\cdot (\zeta_0 \nabla\psi_0) \big) 
\displaybreak[0] \\ \notag &\, 
+ \ep^2 \big( G_2 \psi_0  + G_1 \psi_1  + G_0 \psi_2  - G_1 (\zeta_0 G_0 \psi_0)  -  G_0(\zeta_0 G_1 \psi_0)  
\\ \notag &\, \phantom{ +\ep^2 \big( \ \,}  
- G_0(\zeta_1G_0 \psi_0)  - G_0 (\zeta_0 G_0 \psi_1) +  G_0 ( \zeta_0  G_0 ( \zeta_0  G_0 \psi_0) )  
\\ \notag & \,  \phantom{ +\ep^2 \big( \ \,}  
- \nabla\cdot (\zeta_1 \nabla\psi_0+ \zeta_0 \nabla\psi_1)  +  {\ts\frac12}\Delta(  \zeta_0^2  G_0 \psi_0)  +  {\ts\frac12} G_0 ( \zeta_0^2  \Delta\psi_0) \big)
\displaybreak[0] \\ \notag &\, 
+ \ep^3 R^a_2 
+ \ep^3  \mcR^a_2
\end{align*}
with
\begin{align*}
R^a_2 = &\,  R_2 \psi_0   +  R_1 \psi_1  +  R_0 \psi_2  
% \\&\  
- R_1 (\zeta_0 G_0 \psi_0)  -  R_0 (\zeta_0 G_1 \psi_0)  -  \mcG_0 (\zeta_0 R_1 \psi_0)
\displaybreak[0]\\&\,   
-  R_0 (\zeta_1G_0 \psi_0) -  \mcG_0 (\zeta_1 R_0 \psi_0) 
- R_0 (\zeta_0 G_0 \psi_1)  -  \mcG_0 (\zeta_0 R_0 \psi_1) 
\displaybreak[0]\\&\,
  +  R_0 ( \zeta_0  G_0 ( \zeta_0  G_0 \psi_0) )  +  \mcG_0( \zeta_0  R_0 ( \zeta_0  G_0 \psi_0) )   +  \mcG_0 ( \zeta_0  \mcG_0 ( \zeta_0 R_0 \psi_0) )  
\\& \,
  + {\ts\frac12}\Delta(  \zeta_0^2  R_0 \psi_0) + {\ts\frac12} R_0 (\zeta_0^2  \Delta\psi_0).
\end{align*}
%%%%%%%%%%%%%%%%%%%%%%%%%%%%%%%%%%%%%%%%%%%%%%%%%%%%%%%%%%%%%%%%
%and hence 
%\begin{align*}
%\ep^{d/2} |R^a_2|_{H^s} &\  \le   C(\mu_{\max}, |\xi_j|, |\xi_{ji}|, |\xi_{jik}|) \big( |\tilde \psi_0|_{H^{s+3}} +  |\tilde \psi_1|_{H^{s+2}} + |\tilde \psi_2|_{H^{s+1}} \big)  
% \\&\  
%+ M(s+2, \tilde\zeta_0)  \big(  |\tilde\psi_0|_{H^{s+3}} +  |\tilde\psi_1|_{H^{s+2}} \big) + M(s+1, \tilde\zeta_1)   |\tilde\psi_0|_{H^{s+2}} 
%\end{align*}
%%%%%%%%%%%%%%%%%%%%%%%%%%%%%%%%%%%%%%%%%%%%%%%%%%%%%%%%%%%%%%%%
Finally, we expand the $\nabla$- and $\Delta$-operators 
%(differentiation with respect to $X$) 
according to \eqref{spacediffexpansion}, obtaining
\begin{align*}
& - \ep  \nabla 
 \cdot (\zeta_0 \nabla\psi_0) 
 + \ep^2 \big( - \nabla\cdot (\zeta_1 \nabla\psi_0+ \zeta_0 \nabla\psi_1)  + {\ts\frac12}\Delta(  \zeta_0^2  G_0 \psi_0) + {\ts\frac12} G_0 ( \zeta_0^2  \Delta\psi_0) \big)
\displaybreak[0] \\ \ & =
 - \ep ( \zeta_0'  \cdot \psi_0'  +  \zeta_0  \psi_0''  )
 - \ep^2 ( \zeta_0'  \cdot \nabla'\psi_0  +  \nabla' \zeta_0 \cdot  \psi_0'  +  2 \zeta_0 \nabla' \cdot \psi_0' ) 
\displaybreak[0] \\ & \ \phantom{=\ }
+ \ep^2  \big( - \zeta_1'  \cdot \psi_0'  -  \zeta_1  \psi_0''   -  \zeta_0'  \cdot \psi_1'  -  \zeta_0  \psi_1''  + {\ts\frac12} (  \zeta_0^2  G_0 \psi_0)''   +  {\ts\frac12} G_0 ( \zeta_0^2 \psi_0'') \big)
+ \ep^3 P^a_2 
\end{align*}
with
\begin{align*}
P^a_2 = &  -  (\nabla' \zeta_0 \cdot \nabla'\psi_0 +  \zeta_0 \Delta' \psi_0 )
\displaybreak[0] \\ &  
-  \big( \nabla' \zeta_1 \cdot \nabla \psi_0   +  \zeta_1'  \cdot \nabla'\psi_0  +    \zeta_1 ( 2 \nabla' \cdot \psi_0' +  \ep  \Delta' \psi_0 ) \big) 
\displaybreak[0]\\ &   
-  \big( \nabla' \zeta_0 \cdot  \nabla \psi_1  +  \zeta_0'  \cdot \nabla'\psi_1  +    \zeta_0 ( 2 \nabla' \cdot \psi_1' +  \ep  \Delta' \psi_1) \big) 
\displaybreak[0]\\&    
+ \nabla' \cdot (  \zeta_0^2  G_0 \psi_0)'  +  \ep {\ts\frac12}\Delta'(  \zeta_0^2  G_0 \psi_0) 
+  G_0 ( \zeta_0^2 \nabla' \cdot \psi_0'  ) +  \ep {\ts\frac12} G_0 ( \zeta_0^2  \Delta'  \psi_0 ).
\end{align*}
%%%%%%%%%%%%%%%%%%%%%%%%%%%%%%%%%%%%%%%%%%
%and hence
%\begin{align*}
%\ep^{d/2} | P^a_2 |_{H^s} & \le  
%M(s+2, \tilde\zeta_0)   \big(  |\tilde\psi_0|_{H^{s+2}} +  |\tilde\psi_1|_{H^{s+2}} \big)  + M(s+1, \tilde\zeta_1)   |\tilde\psi_0|_{H^{s+2}} 
%\end{align*}
%%%%%%%%%%%%%%%%%%%%%%%%%%%%%%%%%%%%%%%%%

Altogether, we obtain $\mcG[\ep \zetaa]\psia$ as in the statement of the corollary, with  
\begin{equation*}
R_2^1 = P^a_2 + R^a_2 + \mcR^a_2,
\end{equation*}
for which we obtain the estimate \eqref{estimateR21}, by 
using the estimates (and expansions) of Proposition \ref{LannesResiduals}, 
the product estimates of Lemma \ref{ProductEstimates} below, and 
\eqref{ModulationEstimate}. 

\

%Although in the present paper we focus on second-order approximations, if one is interested only in first-order ones, one can use the expansion.
Analogously, the (shorter) first-order approximation of $\mcG[\ep \zetaa]\psia$ reads
\begin{align*} % \label{mcGmuep2resid}
\mcG[\ep \zetaa]\psia = 
%& \  \mcG_0 \psi_0  + \ep \big(  \mcG_0 \psi_1 - \mcG_0(\zeta_0 \mcG_0 \psi_0)  - \nabla\cdot (\zeta_0 \nabla\psi_0) \big)
%+\ep^2 \mcR^a_1
%& \  G_0 \psi_0  + \ep \big( G_1\psi_0  + G_0 \psi_1 - G_0(\zeta_0 G_0 \psi_0)  - \nabla\cdot (\zeta_0 \nabla\psi_0) \big)
%+ \ep^2 R^a_1 + \ep^2 \mcR^a_1
& \, G_0 \psi_0  + \ep \big( G_1\psi_0  + G_0 \psi_1 - G_0(\zeta_0 G_0 \psi_0)  - \zeta_0'  \cdot \psi_0'  -  \zeta_0  \psi_0''   \big)
% \\ \notag &\,
+ \ep^2 R^1_1
\end{align*}
with
\begin{align}\label{R11}
R^1_1 = 
&\, \mcG[\ep \zetaa]\psi_2 + \mcR_0[\zetaa]\psi_1 + \mcR_1[\zetaa]\psi_0 + \mcG_1[\zeta_1+\ep\zeta_2]\psi_0 
\displaybreak[0] \\ \notag &\,  
+  R_1 \psi_0   +  R_0 \psi_1  - R_0 (\zeta_0 G_0 \psi_0)  -  \mcG_0 (\zeta_0 R_0 \psi_0) 
\displaybreak[0] \\ \notag &\, 
 -  ( \zeta_0'  \cdot \nabla'\psi_0  +  \nabla' \zeta_0 \cdot \nabla \psi_0  +  2 \zeta_0 \nabla' \cdot \psi_0' + \ep  \zeta_0 \Delta' \psi_0 ),
  \end{align}
from which we obtain as above
\begin{equation}\label{estimateR11}
|R^1_1|_{H^s} \le  \ep^{-d/2} M(s+5/2, \tilde\zetaa)  \big(   |\tilde\psi_0|_{H^{s+2}} +| \tilde \psi_1 |_{H^{s+1}} + | \tilde \psi_2 |_{H^{s+1}}  \big);
\end{equation}
while the zeroth-order approximation of $\mcG[\ep \zetaa]\psia$ is given by
\begin{equation*} %\label{mcGmuepresid}
\mcG[\ep \zetaa]\psia = G_0 \psi_0  + \ep R^1_0
\end{equation*}
with
\begin{equation*}
R^1_0 =  \mcG[\ep \zetaa](\psi_1 + \ep \psi_2)  + \mcR_0[\zetaa]\psi_0 +R_0\psi_0,
 \end{equation*}
which leads to
\begin{equation}\label{estimateR01}
|R^1_0|_{H^s}  \le   \ep^{-d/2}  M(s+5/2, \tilde\zetaa)  |\tilde\psia|_{H^{s+1}}. 
\end{equation}

\

We turn now to the expansion of $\mcN_{\ep,\sigma}^2(U_a)$. First, we have
\begin{equation*}
|\nabla\psia|^2 = |\psi_0' |^2 +  \ep 2  \psi_0' \cdot (\nabla'\psi_0 + \psi_1' )+  \ep^2 \rho^a_1
\end{equation*}
with
\begin{align*}
&\rho^a_1 =  2  \psi_0' \cdot ( \nabla' \psi_1+  \nabla \psi_2 ) +  | \nabla'\psi_0 + \nabla(\psi_1+\ep\psi_2)  |^2.
%\\
%&\rho^a_1 = (\nabla' \psi_0 + \psi_1') \cdot (\nabla'\psi_0 + \nabla (\psi_1 + \ep \psi_2) ) + (\nabla'\psi_1 + \nabla\psi_2)\cdot(\nabla\psia + \psi_0')
\end{align*}
%and hence
%\begin{align*}
% & 
%| \nabla \rho^a_1 |_{H^s} \le   \ep^{- d/2} C( |\tilde \psia|_{H^{s+2}} )
%\end{align*}
%
%and, analogously, 
%\begin{equation*}
%|\nabla\psia|^2 = |\psi_0' |^2 + \ep \rho^a_0
%\end{equation*}
%with
%\begin{align*}
%%&\rho^a_0 = 2 \psi_0' \cdot (\nabla' \psi_0 + \nabla(\psi_1+\ep\psi_2) ) +   \ep |\nabla' \psi_0 + \nabla(\psi_1+\ep\psi_2)|^2  
%&\rho^a_0 = (  \psi_0' + \nabla \psia ) \cdot (\nabla' \psi_0 + \nabla(\psi_1+\ep\psi_2) ) 
%\end{align*}
%
Then, 
% we use 
using the first- and zeroth-order expansions of  $\mcG[\ep \zetaa]\psia$ 
% presented 
above,
% (in order to write the residual terms in a shorter form) and 
we
obtain 
%with \eqref{mcGmuepresid}, \eqref{mcGmuep2resid} 
%\begin{equation*}
%\mcG[\ep \zetaa]\psia = \gamma_0 + \ep R_0^1= \gamma_0 +\ep \gamma_1 +\ep^2 R_1^1 
%\end{equation*}
%
%\begin{align*}
%& \frac{(\mcG_\mu^a\psia + \ep \nabla\zetaa\cdot \nabla \psia)^2}{1+\ep^2 |\nabla\zetaa|^2} 
%=  \frac{\gamma_0^2 + \ep 2 \gamma_0 ( \gamma_1+  \zeta_0'\cdot \psi_0'  )  + \ep^2 r_4 }{1+\ep^2 |\nabla\zetaa|^2} 
%\end{align*}
%with 
%\begin{align*}
%r_4 =  & \ 2 \gamma_0 R_1^1 + (R_0^1)^2 + 2 R_0^1 \nabla\zetaa\cdot \nabla \psia  + (\nabla\zetaa\cdot \nabla \psia)^2
%\\& \ 
%+ 2 \gamma_0  ( \zeta_0'\cdot (\nabla'\psi_0 +\nabla(\psi_1+\ep\psi_2)) + (\nabla'\zeta_0 + \nabla(\zeta_1+\ep\zeta_2))\cdot \nabla\psia )   
%\end{align*}
\begin{multline*}
{ \frac{(\mcG[\ep \zetaa]\psia + \ep \nabla\zetaa\cdot \nabla \psia)^2}{1+\ep^2 |\nabla\zetaa|^2} }
 \\
 = %&\,  
 (G_0\psi_0)^2 
% \\&\, 
+ 2 \ep \big( G_1\psi_0  + G_0 \psi_1 - G_0(\zeta_0 G_0 \psi_0) 
% - \zeta_0'  \cdot \psi_0'  
-  \zeta_0  \psi_0''   
% +  \zeta_0'\cdot \psi_0'  
\big)  G_0\psi_0  
% \\&\, 
+ \ep^2 r^a_1
\end{multline*}
with
\begin{align*}
r^a_1 = &\,  2 \big[  R_1^1  + \zeta_0'\cdot (\nabla'\psi_0 +\nabla(\psi_1+\ep\psi_2)) + (\nabla'\zeta_0 + \nabla(\zeta_1+\ep\zeta_2))\cdot \nabla\psia \big] G_0\psi_0   
\displaybreak[0] \\&\,
 + (R_0^1+\nabla\zetaa\cdot \nabla \psia)^2 -  (\mcG[\ep \zetaa]\psia + \ep \nabla\zetaa\cdot \nabla \psia)^2 \frac{  |\nabla\zetaa|^2  }{1+\ep^2 |\nabla\zetaa|^2}.
\end{align*}
Finally,  we get
\begin{align*}
\nabla\cdot \Big( \frac{\nabla\zetaa}{\sqrt{1+ \ep^2 |\nabla\zetaa|^2}} \Big) 
= &\,  \zeta_0'' + \ep (2 \nabla'\cdot\zeta_0'  + \zeta_1'' ) 
\displaybreak[0] \\&\, + \ep^2 \big(  \Delta' \zeta_0  + 2 \nabla'\cdot\zeta_1' +  \zeta_2'' -  {\ts \frac12}   ( |\zeta_0'|^2  \zeta_0')'  \big)
% CAUTION:  ( |\zeta_0'|^2  \zeta_0')'   \neq  |\zeta_0'|^2  \zeta_0'' (since the norm is the Euclidean one in \R^d)
+ \ep^3 s^a_2 
\end{align*}
with 
\begin{align*}
s^a_2=&\,  \Delta' \zeta_1  + \nabla'\cdot \zeta_2'  + \nabla\cdot\nabla' \zeta_2 -  \frac12 \nabla' \cdot ( |\zeta_0'|^2  \zeta_0') 
\displaybreak[0] \\&\,
-  \frac12 \nabla\cdot \big[  \big(  ( \nabla'\zeta_0 + \nabla(\zeta_1+\ep\zeta_2) ) \cdot (\zeta_0' + \nabla\zetaa) \big)   \zeta_0'  +  |\nabla\zetaa|^2 ( \nabla' \zeta_0 +  \nabla(\zeta_1+\ep \zeta_2) )   \big]  
\displaybreak[0] \\&\,
+\ep \nabla\cdot \Big( \frac{ ( 4  + r^2  (1 + r ) ) |\nabla\zetaa|^4  \nabla\zetaa} { 2 r  (1 + r ) (2 + r  ( 1 + r^2) )} \Big) , \quad \text{where}\quad  r = \sqrt{1 + \ep^2 |\nabla\zetaa|^2}.
\end{align*}
%and hence
%\begin{align*}
%| \nabla s^a_2 |_{H^s} \le   \ep^{- d/2} M(s+3, \tilde \zetaa)
%%  C(|\tilde\zetaa|_{H^{s+3}}) 
%% +  \ep \left|  \frac{ ( 4  + r^2  (1 + r ) ) |\nabla\zetaa|^4  \nabla\zetaa}{2 r  (1 + r ) (2 + r  ( 1 + r^2) )}\right|_{H^{s+2}} 
%%+  \ep |  F(\nabla\zetaa )  \nabla\zetaa|_{H^{s+2}} 
%% +  \ep C( |\nabla\zetaa|_\infty )  | \nabla\zetaa |_{H^{s+2}}    \ep^{- d/2} |  \tilde \zetaa |_{H^{(s+2)\vee t_0 +1}} 
%%+  \ep C( |\tilde \zetaa|_{H^{t_0+1}}) \, \ep^{-d/2}   | \tilde\zetaa|_{H^{s+3}}    \, \ep^{- d/2} | \tilde \zetaa|_{H^{(s+3)\vee ( t_0 +1 )}} 
%\end{align*}
%with %(see (2.2) in ASL, Large time... 2008)
%$$
%F(x) = \frac{  ( 4  + (1+\ep^2 |x|^2 )  (1 + \sqrt{1+ \ep^2 |x^2|} ) ) |x|^4 }{2 \sqrt{1+\ep^2 |x|^2}  (1 + \sqrt{1+\ep^2 |x|^2} ) (2 + \sqrt{1+\ep^2 |x|^2}  ( 2+ \ep^2 |x|^2) )}, \quad x\in\R^d,
%$$
%where we used ... 
%and, analogously, 
%\begin{align*}
%\nabla\cdot \Big( \frac{\nabla\zetaa}{\sqrt{1+\ep^2 |\nabla\zetaa|^2}} \Big)
%= 
%%\Delta \zeta_0 + \ep  \Delta \zeta_1  
%\zeta_0'' + \ep( 2  \nabla'\cdot\zeta_0'  + \zeta_1'' )  
%+ \ep^2 s^a_1 
%\end{align*}
%with
%\begin{align*}
%s^a_1= &\ \Delta' \zeta_0 + \nabla'\cdot\zeta_1'  + \nabla\cdot \nabla'\zeta_1 + \Delta \zeta_2
%% \\&\ 
%% -  \nabla\cdot \Big( \frac{ |\nabla\zetaa|^2 \nabla\zetaa}{\sqrt{1+\ep^2 |\nabla\zetaa|^2}\big(1+ \sqrt{1+\ep^2 |\nabla\zetaa|^2}\big)} \Big)
%-  \nabla\cdot \Big( \frac{ |\nabla\zetaa|^2 \nabla\zetaa}{r (1+ r )} \Big), \quad r = \sqrt{1 + \ep^2 |\nabla\zetaa|^2}. 
%\end{align*}

% Summarily, 
Summarizing, we get the expansion of $\mcN_{\ep,\sigma}^2(U_a)$ presented in the statement, with
\begin{equation*}
R_2^2 = - {\ts \frac12}  \rho^a_1 +  { \ts \frac12 } r^a_1  +  \bond   s^a_2. 
\end{equation*}
Using 
%once 
again Proposition \ref{LannesResiduals}, Lemma \ref{ProductEstimates} below, \eqref{ModulationEstimate}, 
and the estimates \eqref{estimateR11}, \eqref{estimateR01} with \eqref{Pmu_above}, 
% and the explicit forms of $R_1^1$ and $R_0^1$, 
we obtain for $s\ge 1$
% in a similar way as before
the estimate \eqref{estimateR22}.
Note, that we estimated $|r_1^a|_{H^{s+1/2}}$, by expanding $R_0^1$, $R_1^1$, $\mcG[\ep \zetaa]\psia$ 
up to residual terms of order $O(\ep)$, in order to control the appearing products with respect to $\ep$. 
%
%%% FALLS GEFRAGT, SAGEN DAS WIR DEN ESTIMATE (2.8) ODER SO (SPECIAL FORM FUNCTIONS) NUR EINMAL 
%%% VERWENDEN KOENNEN 
%
This leads to an increase by one order of the regularity required for $\tilde \psia$. 
%
%%% DAS KANN MAN EVENTUELL VERMEIDEN, SIEHE DIE MULTISKALEN ABSCHAETZUNGEN FUER FOURIER-MULTIPLIER 
%%% VON CSS97  (REGULARITY)

\

For future use, we present also the first-order expansion of $\mcN_{\ep,\sigma}^2(U_a)$  
\begin{align*}
- \mcN_{\ep, \sigma}^2 (U_a) 
 = & - \zeta_0 + \bond  \zeta_0'' 
+ \ep \big( -  \zeta_1 -  {\ts\frac12} |\psi_0' |^2 +  {\ts \frac12} (G_0\psi_0)^2  +  \bond  ( 2 \nabla'\cdot \zeta_0' +  \zeta_1'' ) \big) 
\displaybreak[0] \\ &
+ \ep^2 R_1^2 
\end{align*}
with
\begin{equation}\label{R12}
R_1^2 = -\zeta_2 - {\ts \frac12}  \rho^a_0 +  { \ts \frac12 } r^a_0  +  \bond   s^a_1, 
\end{equation}
where
\begin{align*}
\rho^a_0 = &\, (  \psi_0' + \nabla \psia ) \cdot (\nabla' \psi_0 + \nabla(\psi_1+\ep\psi_2) ),
\displaybreak[0] \\
r^a_0 = &\,  (R_0^1+ \nabla\zetaa\cdot \nabla \psia) ( G_0\psi_0+  \mcG[\ep \zetaa]\psia + \ep \nabla\zetaa\cdot \nabla \psia)
\\ &\,  -  \ep (\mcG[\ep \zetaa]\psia + \ep \nabla\zetaa\cdot \nabla \psia)^2 \frac{  |\nabla\zetaa|^2  }{1+\ep^2 |\nabla\zetaa|^2},
\displaybreak[0] \\
s^a_1= &\,  \Delta' \zeta_0 + \nabla'\cdot\zeta_1'  + \nabla\cdot \nabla'\zeta_1 + \Delta \zeta_2
-  \nabla\cdot \Big( \frac{ |\nabla\zetaa|^2 \nabla\zetaa}{r (1+ r )} \Big), \quad r = \sqrt{1 + \ep^2 |\nabla\zetaa|^2}, 
\end{align*}
which leads to
\begin{equation}\label{estimateR12}
| \mfP R_1^2 |_{H^s} \le \ep^{-d/2} M(s+3 ,\tilde\zetaa)  
\left(C \big( | \tilde \psi_0|_{H^{s+5/2}},  |\tilde \psi_1 + \ep \tilde \psi_2|_{H^{s+3/2}} \big) +  \bond \right) .
\end{equation}
% for $s\ge 1$, $t_0 = 3/2$
%
%%% DIE ZUSAETZLICHE ABLEITUNG KOMMT DADURCH ZUSTANDE, DASS MAN BEI r_a^1 DIE RESIDUEN-TERME, 
%%% DIE NICHT DIE STRUKTURIERTE FORM HABEN, "AUFMACHT", D.H. UM EINS ENTWICKELT. 
%%% DAS KANN MAN WAHRSCHEINLICH BESSER MACHEN 
%
%with % (see \cite[(2.2)]{ASL08}) 
%\begin{equation*}
%G(x) = \frac{|x|^2}{\sqrt{1+\ep^2|x|^2} (1+\sqrt{1+\ep^2|x|^2})}, \quad x\in\R^d,
%\end{equation*}
\end{proof}

\

In the following lemma we list the product estimates for functions in $H^s(\R^d)$ that we used in order to obtain the estimates 
%\eqref{estimateR21}, \eqref{estimateR22}, \eqref{estimateR11}, \eqref{estimateR01}, \eqref{estimateR12}. 
\eqref{estimateR21} -- \eqref{estimateR12}.
For their proof we refer to \cite[Appendix B.1.1]{Lannes} and \cite[(2.2)]{ASL08}, and the references given therein. Recall also that  $H^{t_0} (\R^d) \subset L^\infty (\R^d)$  for $t_0>d/2$. 
We use the notation
\begin{equation*}
A_s + \langle B_s \rangle_{s>r} = \begin{cases} A_s & \text{if $s\le r$,} \\ A_s + B_s  & \text{if $s>r$.}\end{cases} 
%\qquad ($s,r\in\R$)
\end{equation*}
\begin{lemma} \label{ProductEstimates}
Let $t_0>d/2$. Then, for $f,g\in H^s(\R^d)$, $s\in\R$,  the following estimates hold true:
\begin{enumerate}
\item if $s_1,s_2\in \R$, $s_1,s_2\ge s$, $ 0 \vee (s + t_0)  \le s_1 + s_2 $, then $|fg|_{H^s}\le C |f|_{H^{s_1}}|g|_{H^{s_2}}$;
\item if $s\ge 0$, then  $|fg|_{H^s}\le C ( |f|_\infty  |g|_{H^s}+ |f|_{H^s} |g|_\infty)$;
\item if $s\ge 0$ and $F\in C^\infty(\R^n;\R^m)$, $F(0)=0$, then $|F(u)|_{H^s} \le C(|u|_\infty)|u|_{H^s}$;
\item if $s\ge - t_0$ and $1+g(X)\ge c_0>0\ \forall\ X\in\R^d$, then $$\left|\frac{f}{1+g}\right|_{H^s}\le C({\ts \frac1{c_0} }, |g|_{H^{t_0}} ) (|f|_{H^s} + \langle |f|_{H^{t_0}} |g|_{H^s} \rangle_{s>t_0}).$$  
\end{enumerate}
\end{lemma}

\

% REFORMULIEREN, WENIGER POETISCH (DONE 19.9.15)

We can now insert into the expansions of $\mcG[\ep\zetaa]\psia$ and $\mcN_{\ep, \sigma}^2(U_a)$ 
with respect to 
% in 
$\ep$ (Corollary \ref{epexpansion})
the decompositions of the $(\zeta_n, \psi_n)^T$-terms 
%, $n=0,1,2$, 
into their harmonics, see after \eqref{Uapp},
and expand at each order of $\ep$ with respect to the different harmonics.
Due to the nonlinearity of the terms of the $\ep$-expansions and the fact that the inserted terms already contain higher-order harmonics, which moreover result from three different carrier waves, this leads to very involved formulas. 
The following proposition gives the relevant structure of these expansions containing the information needed for the derivation and characterization of the modulation equations, 
and we refer to the Appendix for exact formulas of the more involved expressions.
%
%while an algorithm for the determination of the exact form of the more involved terms is deferred to the Appendix.   
%
%%%%%%%%%%%%%%%%%%%%%%%%%%%%%%%%%%%%%%%%%%%%%%%%%%%%%%%%%%%%%%%%%%%%
\begin{proposition} \label{finalexpansion}
With Notation \ref{approximationnotation} and $(\zetaa,\psia)^T$ as in \eqref{Uapp}, 
%and the residuals $R_2^1$, $R_2^2$ as estimated in Corollary \ref{epexpansion}, the expansions there
%of the latter 
% for $\mcG[\ep\zetaa]\psia = -\mcN_\sigma^1(U_2^a)$ and $\mcN_\sigma^2(U_2^a)$ 
%take the following form:
the expansions of Corollary \ref{epexpansion}
take the form
% AUCH DIE INDEX UND SUMMATION NOTATION CLARIFIEN NACHDEM MAN DIE AM ANFANG ODER SONSTWO KONZIS AUFGESCHRIEBEN HAT
\begin{align*}% \label{expansionNsigma1} 
& 
% - \mcN_\sigma^1 (U^a_2) = 
 \mcG[\ep \zetaa]\psia 
=
 \sum_j  g_j  \psi_{0j} \ee_j  +\cc 
\displaybreak[0] \\& 
+  \ep \Big( 
\sum_j \big(  g_j  \psi_{1j}  - \big( \ii g'_j  \cdot   \nabla'+ (g_j^2 - |\xi_j|^2) \zeta_{00} \big) \psi_{0j}  \big) \ee_j 
+ \sum_{ji}  ( g_{ji}  \psi_{1ji}  +  A_{ji}  ) \ee_{ji} + \cc 
\Big) 
\displaybreak[0] \\&
+  \ep^2 \Big( 
\sum_j \big( g_j  \psi_{2j}  - \ii g'_j \cdot \nabla' \psi_{1j}  
% - \big( {\mathrm H}_j + (g_j^2 - |\xi_j|^2) \zeta_{10} \big)  \psi_{0j}  - \ii \bk_j  \zeta_{0j}  {\cdot} \nabla' \psi_{00}  + C_j  
- P_j 
\big) \ee_j
 + \sum_{ji} ( g_{ji} \psi_{2ji} - \ii  g'_{ji}  \cdot  \nabla'   \psi_{1ji}  + C_{ji}  ) \ee_{ji} 
\\& \phantom{ +\ep^2\Big(\ }
+ \sum_{jik} ( g_{jik} \psi_{2jik} + C_{jik} ) \ee_{jik} 
+\cc 
%  \\&  \phantom{ +\ep^2\Big(\ }
- \sqrt{\mu} \Delta' \psi_{00}   + C_0
\Big)
% \\&  
+ \ep^3 R_2^1
%%%%%%%%%%%%%%%%%%%
\\ \intertext{and}
& - \mcN_{\ep, \sigma}^2 (U^a_2) 
=  \sum_j  ( - b_j \zeta_{0j} ) \ee_j + \cc  - \zeta_{00} 
\displaybreak[0] \\&
+ \ep \Big( 
\sum_j ( - b_j \zeta_{1j} +  \bond 2\ii\bk_j \cdot \nabla' \zeta_{0j} )  \ee_j
+ \sum_{ji} ( - b_{ji}  \zeta_{1ji}  + B_{ji} ) \ee_{ji} 
+ \cc 
 - \zeta_{10}  + B_0  \Big)
\displaybreak[0] \\&
+\ep^2 \Big(
\sum_j  ( - b_j \zeta_{2j}   + 
%\bond   2 \ii
{ \ts \frac{2\ii}{\mathrm{Bo}} }
 \bk_j \cdot \nabla' \zeta_{1j} 
% +  \bond   \Delta' \zeta_{0j}   -  \ii \bk_j \psi_{0j} \cdot  \nabla' \psi_{00} +  D_j   
- Q_j 
) \ee_j
 %  \\&  \phantom{ +\ep^2\Big(\ }
+ \sum_{ji} (- b_{ji} \zeta_{2ji}  +  
% \bond 2 \ii 
{\ts \frac{2\ii}{\mathrm{Bo}} }
\bk_{ji} \cdot  \nabla'   \zeta_{1ji}  +  D_{ji} ) \ee_{ji} 
 \\&   \phantom{ +\ep^2\Big(\ }
+ \sum_{jik}  ( - b_{jik} \zeta_{2jik}  + D_{jik} ) \ee_{jik} 
+\cc
%\\&   \phantom{ +\ep^2\Big(\ }
-   \zeta_{20}  + D_0
\Big)
% \\&
+ \ep^3 R_2^2 .
\end{align*}
Here, $R_2^1$, $R_2^2$ are the residual terms of Corollary \ref{epexpansion}, 
\begin{align*}
 A_{jj} & =
 % - a_{jj,-j} 
-(g_{jj}g_j - 2 |\xi_j|^2)
 \zeta_{0j} \psi_{0j}, 
\qquad j\in J, 
\displaybreak[0] \\  
A_{ji} & = 
%- a_{ji,-i}
-(g_{ji} g_i - \xi_{ji} \cdot \xi_i)
\zeta_{0j} \psi_{0i} 
%- a_{ji,-j} 
-(g_{ji} g_j - \xi_{ji} \cdot \xi_j)
\zeta_{0i} \psi_{0j},  
\qquad (j,i)\in I_<,
\displaybreak[0] \\
A_{j,-i} &= 
% - a_{j,-i,i} 
- (g_{j,-i}g_i + \xi_{j,-i} \cdot \xi_i)
\zeta_{0j} \overline{\psi_{0i}} 
%- a_{j,-i,-j} 
- (g_{j,-i}g_j-\xi_{j,-i} \cdot \xi_j)
\overline{\zeta_{0i}}\psi_{0j},  
\qquad (j,i)\in I_<,
\end{align*}
% and
\begin{align*}
B_0 &  =  {\ts \sum_j }  (g_j^2 - |\xi_j|^2)  |\psi_{0j}|^2,
\qquad 
B_{jj} =  {\ts \frac12 }  (g_j^2 + |\xi_j|^2) \psi_{0j}^2, 
\qquad j\in J,
\displaybreak[0] \\
B_{ji} & = (g_jg_i + \xi_j \cdot \xi_i) \psi_{0j}  \psi_{0i}
\qquad
B_{j,-i}  =  (g_jg_i - \xi_j \cdot \xi_i) \psi_{0j}  \overline{\psi_{0i}}, 
\qquad (j,i) \in I_<
\end{align*}
and
\begin{align*}
P_j & =   \big( {\mathrm H}_j + (g_j^2 - |\xi_j|^2) \zeta_{10} \big)  \psi_{0j}  + \ii \bk_j  \zeta_{0j}  {\cdot} \nabla' \psi_{00}  - C_j,
\displaybreak[0] \\
Q_j & =  - \bond   \Delta' \zeta_{0j}   + \ii \bk_j \psi_{0j} \cdot  \nabla' \psi_{00} -  D_j. 
\end{align*}
The exact formulas for the functions $C, D$ are given in the Appendix.
In particular, $C_j, D_j$ consist of 
% cubically coupled terms 
cubic products 
of $\zeta_{0j}, \psi_{0j}$ and 
% quadratically coupled terms 
quadratic products 
of $\zeta_{1ji},\psi_{1ji}$ with $\zeta_{0k},\psi_{0k}$.
\end{proposition}

%%%%%%%%%%%%%%%%%%%%%%%%%%%%%%%%%%%%%%

From this proposition and the expansions of the time-derivatives of $U_a$ as exemplified in \eqref{expansiontimezeta}, 
we obtain finally our full expansion of the water waves equation \eqref{wwe} with respect to $\ep$.  
\begin{corollary}\label{wwexpansion2}
With Notation \ref{approximationnotation} and that of Proposition \ref{finalexpansion}, and with \eqref{expansiontimezeta},
the water waves equation \eqref{wwe} with the ansatz \eqref{Uapp} takes the form
\begin{align*}% \label{expansionNsigma1} 
& \pl_t \zetaa - \mcG[\ep \zetaa]\psia 
 =
 \sum_j  (  - \ii \omj  \zeta_{0j} - g_j  \psi_{0j} ) \ee_j  +\cc 
\displaybreak[0]\\& 
+  \ep \Big( 
\sum_j \big( \pl_t'  \zeta_{0j} - \ii \omj  \zeta_{1j} - g_j  \psi_{1j}  +  \big( \ii g'_j  \cdot   \nabla'+ (g_j^2 - |\xi_j|^2) \zeta_{00} \big) \psi_{0j}  \big) \ee_j 
\\&   \phantom{ +\ep \Big(\ }
+ \sum_{ji}  ( - \ii \om_{ji}  \zeta_{1ji}  -  g_{ji}  \psi_{1ji}  -  A_{ji}  ) \ee_{ji} + \cc 
+ \pl_t' \zeta_{00} 
\Big) 
\displaybreak[0]\\&
+  \ep^2 \Big( 
\sum_j \big( \pl_t' \zeta_{1j} - \ii \omj  \zeta_{2j}  - g_j  \psi_{2j}  +  \ii g'_j \cdot \nabla' \psi_{1j}  + P_j \big) \ee_j
 \\& \phantom{ +\ep^2\Big(\ }
+ \sum_{ji} (  \pl_t'   \zeta_{1ji} - \ii \om_{ji}    \zeta_{2ji} - g_{ji} \psi_{2ji} +  \ii  g'_{ji}  \cdot  \nabla'   \psi_{1ji}  -  C_{ji}  ) \ee_{ji} 
\\&   \phantom{ +\ep^2\Big(\ }
+ \sum_{jik} ( - \ii \om_{jik} \zeta_{2jik} - g_{jik} \psi_{2jik} -  C_{jik} ) \ee_{jik} 
+\cc 
%  \\&  \phantom{ +\ep^2\Big(\ }
+ \pl_t'  \zeta_{10}  +  \sqrt{\mu} \Delta' \psi_{00}   -  C_0  
\Big)
\displaybreak[0]\\&  
+ \ep^3 r_2^1
% , \quad r_2^1 =  \sum_j   \pl_t' \zeta_{2j}  \ee_j + \sum_{ji} \pl_t' \zeta_{2ji}  \ee_{ji}  + \sum_{jik} \pl_t'   \zeta_{2jik}  \ee_{jik}  + \cc +  \pl_t'  \zeta_{20}   - R_2^1,
%%%%%%%%%%%
\\ \intertext{and}
& \pl_t \psia + \mcN_{\ep, \sigma}^2 (U^a_2) 
=    \sum_j (  - \ii \omj \psi_{0j} + b_j \zeta_{0j} ) \ee_j + \cc  +  \zeta_{00} 
\displaybreak[0] \\&
+ \ep \Big( 
\sum_j (  \pl_t'  \psi_{0j} - \ii \omj  \psi_{1j}  + b_j \zeta_{1j} -  \bond 2\ii\bk_j \cdot \nabla' \zeta_{0j} )  \ee_j
\\&   \phantom{ +\ep \Big(\ }
+ \sum_{ji} ( - \ii \om_{ji} \psi_{1ji} +  b_{ji}  \zeta_{1ji}  - B_{ji} ) \ee_{ji} 
+ \cc 
+ \pl_t' \psi_{00}  + \zeta_{10}  - B_0 
 \Big)
\displaybreak[0] \\&
+\ep^2 \Big(
\sum_j  (   \pl_t' \psi_{1j} - \ii \omj  \psi_{2j}  + b_j \zeta_{2j}   - \bond   2 \ii \bk_j \cdot \nabla' \zeta_{1j} + Q_j ) \ee_j
 \\&  \phantom{ +\ep^2\Big(\ }
+ \sum_{ji} (  \pl_t'   \psi_{1ji} - \ii \om_{ji} \psi_{2ji} + b_{ji} \zeta_{2ji}  -  \bond  2 \ii \bk_{ji} \cdot  \nabla'   \zeta_{1ji}  - D_{ji} ) \ee_{ji} 
\\&   \phantom{ +\ep^2\Big(\ }
+ \sum_{jik}  ( - \ii \om_{jik} \psi_{2jik} + b_{jik} \zeta_{2jik}  -  D_{jik} ) \ee_{jik} 
+\cc
% \\&   \phantom{ +\ep^2\Big(\ }
+ \pl_t'  \psi_{10} + \zeta_{20}  -  D_0 
\Big)
\displaybreak[0]\\&
+ \ep^3 r_2^2
% ,\quad r_2^2 = \sum_j   \pl_t' \psi_{2j}  \ee_j + \sum_{ji} \pl_t' \psi_{2ji}  \ee_{ji}  + \sum_{jik} \pl_t' \psi_{2jik}  \ee_{jik}  + \cc +  \pl_t'  \psi_{20}  - R_2^2.
\end{align*}
with
\begin{align*}
r_2^1 & =  \sum_j   \pl_t' \zeta_{2j}  \ee_j + \sum_{ji} \pl_t' \zeta_{2ji}  \ee_{ji}  + \sum_{jik} \pl_t'   \zeta_{2jik}  \ee_{jik}  + \cc +  \pl_t'  \zeta_{20}   - R_2^1,
\displaybreak[0] \\
r_2^2 & = \sum_j   \pl_t' \psi_{2j}  \ee_j + \sum_{ji} \pl_t' \psi_{2ji}  \ee_{ji}  + \sum_{jik} \pl_t' \psi_{2jik}  \ee_{jik}  + \cc +  \pl_t'  \psi_{20}  - R_2^2.
\end{align*}
\end{corollary}

%%%%%%%%%%%%%%%%%%%%%%%%%%%%%%%%%%%%%%%%%%%%%%%%%%%%%%%%%%%%%%%%%%%%%%
% \newpage

\section{Formal derivation of the modulation equations} \label{SectionFormalDerivation}
%\section{Resonances and formal derivation} \label{SectionFormalDerivation}

Having determined the expansion in terms of $\ep$ of the left-hand side of \eqref{UaWweResiduals},  
we see that the proposed approximation \eqref{Uapp} satisfies formally the water waves problem \eqref{wwe} 
up to residual terms of order $\ep^3$, i.e., \eqref{UaWweResiduals} holds true, 
if and only if all terms of the expansion of orders up to $\ep^2$ vanish identically.  
This means that the macroscopic coefficients of each one of the mutually different harmonics have to vanish separately.  
%
% In particular, according 
According 
to Notation \ref{approximationnotation}(3a), which stipulates that all plane-waves $\ee_j$, $j\in J$, are mutually different and also different from the zeroth-order harmonic $\ee^0=1$, and our closedness assumption 
(Notation \ref{approximationnotation}(3c)) that all higher-order harmonics are neither plane-waves nor equal to $1$, 
this implies that each of the coefficients to $\ee_j$, $1=\ee^0$ and the different higher harmonics of orders up to $\ep^2$
has to vanish identically.
This necessary condition, leads to equations for the corresponding macroscopic coefficients, which are called modulation equations. 

The typical way to obtain these equations is to proceed step by step from lower to higher orders of $\ep^n$, $n=0,1,2$, and require at each step that the coefficients to each different harmonic vanishes. 
However, one could as well add for each harmonic some of the macroscopic coefficients multiplied by their orders $\ep^n$ and require that the sum vanishes up to residuals of a higher order. 
Since at each step the obtained macroscopic equations are typically undetermined, in the sense that they contain terms which are determined at a higher order of $\ep^n$, the second approach allows for a different choice of the, at order $\ep^2$, still undetermined coefficients. 
However, with both approaches we obtain \eqref{UaWweResiduals}. 
For more details on the second approach we refer to Remark \ref{SectionFormalDerivation}.3.
In the present paper, we follow the first, more standard approach.

%The standard way to obtain this is to proceed hierarchically from lower to higher orders of $\ep$ and require at each order that the corresponding coefficient vanishes. 
%%
%% The outcome of this procedure are the modulation equations for the macroscopic functions of the two-scale ansatz \eqref{Uapp}. 
%%
%%This is how we will proceed here. 
%%
%This is what we do in the following. 
%(For a modification of this approach, see Remark \ref{SectionFormalDerivation}.3).
%%However, in principle one could consider for each different harmonic the expansion of its coefficients with respect to $\ep$ as a whole and require that this expansion 
%% vanishes up to residual terms of order $\ep^3$. We present this approach in Remark 2 at the end of this section. 
%% 
%% IN DER TAT, WARUM MACHT MAN DAS NICHT AUCH UNTER EINBEZIEHUNG DER \ep^0 TERME UND FUER ALLE HARMONISCHEN ?
%% 
%%%% NEEDS PROBABLY REFORMULATION (I THINK, DONE AT 23.1.2016)

\

%%%%%%%%%%%%%%%%%%%%%%%
%%% 0order

Starting from the terms of order $\ep^0$, and according to our assumptions,
we obtain immediately from Corollary \ref{wwexpansion2} the macroscopic equations 
\begin{equation} \label{0order}
\zeta_{0j} =  \ii \frac{ \om_j}{b_j} \psi_{0j} , \quad  \omj^2   = b_j  g_j \quad \text{for $j\in J$,} 
\quad\text{and}\quad  
%\qquad 
\zeta_{00} = 0.
\end{equation}
% DIE DRITTE GLEICHUNG ABER NUR WENN MAN NICHT-TRIVIALE MODULATIONSFUNKTIONEN  HABEN WILL 
%
% COMMENT THAT THIS MEANS FINDING THE PLANE WAVE SOLUTIONS OF THE LINEARIZED WATER WAVES PROBLEM: Hence, at leading or zero order we just obtain the dispersion relation
%

\

%%%%%%%%%%%%%%%%%%%%%%%%%%%%%%%%%%%
%%% 1order

%% PART OF THE FOLLOWING SHOULD BE IN THE INTRODUCTION
%As for the $\ep^0$-case treated before we would like to derive conditions on the (macroscopic) coefficients in front of the (microscopic) harmonic functions $\ee_j$, $\ee_{ji}$ in order for the equations \eqref{full1order} to be satisfied identically. However, the difference between \eqref{full0order} and \eqref{full1order} is that now also the second order harmonics $\ee_{ji}=\ee_j\ee_i=\ee^{\ii (\theta_j+\theta_i)}$
%appear in the latter.
%For these second order harmonics the following dichotomy holds: either they are plane waves solving the linearized water waves problem or not. 
%Equivalently, either the dispersion relation $\om_{ji}^2 = b_{ji} g_{ji} $ is satisfied, or not.  
%In the former case, we say that the plane waves $\ee_j$,  $\ee_i$ are in (second-order-)resonance and generate the plane wave $\ee_{ji}$. 
%Although this is a very interesting case, it is also exceptional. It turns out that the resonance of the carrier waves manifests itself also in the macroscopic modulation equations (three-wave-interaction, see for $d=1$ \cite{SchneiderWayne}). 

At the next order $\ep^1$, by requiring that the coefficients of $\ee_j$ 
%, $j\in J$, 
vanish, 
%, using \eqref{0order} and due to Assumption \eqref{nonresonance}, 
and using
% from Corollary \ref{wwexpansion2} 
 \eqref{0order} and the identity \eqref{nablaomega}, we obtain 
% for  the coefficients of  the first order harmonics $\ee_j$, $j\in J$, 
by elimination of 
$$b_j\zeta_{1j} -\ii \om_j \psi_{1j} = \frac{b_j}{\ii \om_j} (\ii \om_j \zeta_{1j} + g_j \psi_{1j})$$
(due to $\om_j^2=b_j g_j$) the equations
%\begin{align*}
%& \ii \frac{\om_j}{b_j} \pl_t' \psi_{0j} - \ii \omj  \zeta_{1j} - g_j  \psi_{1j}  +  \ii g'_j  \cdot   \nabla' \psi_{0j} = 0,
%\\
%&   \pl_t'  \psi_{0j} - \ii \omj  \psi_{1j}  + b_j \zeta_{1j} + \bond 2 \frac{\om_j}{b_j} \bk_j \cdot \nabla' \psi_{0j} =0.
%  \end{align*}
%Since, according to \eqref{deffreq}, 
%\begin{equation}  \label{nablaomega}
%2\om_j\nabla\om_j = b_j g'_j+\bond 2\bk_j g_j , \qquad \nabla\om_j = \nabla\om(\bk_j),
%\end{equation} 
%% and \ref{AssumptionPlaneWaves}
%% MAN KOENNTE/SOLLTE DIE DEFINITION VON \nabla \om_j BEI DEN EIGENSCHAFTEN VON \om SCHREIBEN/EINFUEHREN
%and taking into account also $\om_j^2 = b_j g_j$ of \eqref{0order}, the above system is equivalent to 
% as necessary conditions for the consistency of our approximation up to residual terms of order $\ep^2$
\begin{align}
& \label{nonrestransport}
\pl_t' \psi_{0j}   +  \nabla\om_j \cdot   \nabla' \psi_{0j} = 0,
\displaybreak[0] \\
& 
% - 2 g_j   \psi_{1j}  -2\ii \om_j \zeta_{1j}  - \ii 2 \frac{\om_j}{b_j} \nabla\om_j \cdot \nabla' \psi_{0j} + 2  \ii g'_j  \cdot   \nabla' \psi_{0j}=0.
%  \zeta_{1j}  = \ii \frac{\om_j}{b_j}\psi_{1j} -  \frac{1}{b_j} \nabla\om_j \cdot \nabla' \psi_{0j}  +  \frac1{\om_j}g'_j  \cdot   \nabla' \psi_{0j}.
% \zeta_{1j}  = \ii \frac{\om_j}{b_j}\psi_{1j} +   \frac{1}{b_j} ( \nabla\om_j - \bond 2 \frac{\om_j}{b_j} \xi_j ) \cdot \nabla'\psi_{0j}.
 \label{zeta1j}
\zeta_{1j} = \ii \frac{ \om_j }{b_j} \psi_{1j} + \nablabo\om_j  \cdot \nabla' \psi_{0j}
% , \qquad \nablabo\om_j =\frac1{b_j}\Big( \nabla\om_j  - \bond 2\frac{  \om_j}{b_j}  \bk_j \Big).
\end{align}
(with the notation \eqref{nablabo}).
For the coefficient of $\ee^0=1$ we get
\begin{align}\label{zeta10}
 & \pl_t' \psi_{00} + \zeta_{10}  = B_0 = \sum_j  ( g_j^2 - |\xi_j|^2 ) |\psi_{0j}|^2
\end{align}
(see Proposition \ref{finalexpansion} for the definition of $B_0$). 
Moreover, for the coefficients of the second-order harmonics $\ee_{ji}$, $(j,i)\in I$, which surely differ from the first- and zero-order harmonics by the first non-resonance condition of Notation \ref{approximationnotation}(3c), we obtain the equations
\begin{align}\label{zetapsi1ji}
% (- \ii \om_{ji} )  \zeta_{1ji}  = g_{ji}  \psi_{1ji}  + A_{ji},  \qquad (- \ii \om_{ji} )  \psi_{1ji} =  -  b_{ji} \zeta_{1ji}  +  B_{ji},
% \begin{pmatrix} - \ii \om_{ji}  & - g_{ji} \\ b_{ji} & -\ii \om_{ji} \end{pmatrix} \begin{pmatrix} \zeta_{1ji}  \\ \psi_{1ji} \end{pmatrix} = \begin{pmatrix} A_{ji} \\ B_{ji} \end{pmatrix},
 \begin{pmatrix} \zeta_{1ji}  \\ \psi_{1ji} \end{pmatrix} = \frac1{\om_{ji}^2 - b_{ji} g_{ji}} \begin{pmatrix}  \ii \om_{ji}  & - g_{ji} \\ b_{ji} &  \ii \om_{ji} \end{pmatrix} \begin{pmatrix} A_{ji} \\ B_{ji} \end{pmatrix}
\end{align}
%that determine 
%uniquely 
%the macroscopic functions 
%$\zeta_{1ji}$, $\psi_{1ji}$ (again due to Assumption \ref{nonresonance}) 
with $A_{ji}, B_{ji}$ as in Proposition \ref{finalexpansion}. In particular, due to \eqref{0order}, we have
\begin{align} \label{As1order}
 A_{jj} & = - \ii (g_{jj}g_j -  2 |\xi_j|^2) \frac{\om_j}{b_j} \psi_{0j}^2, \quad j\in J,
\displaybreak[0]\\ \notag
A_{ji} & = - \ii \Big(  (g_{ji}g_i - \xi_{ji} \cdot \xi_i) \frac{\om_j}{b_j}  + (g_{ji}g_j - \xi_{ji} \cdot \xi_j) \frac{\om_i}{b_i} \Big) \psi_{0j} \psi_{0i}, \quad  (j,i)\in I_<,
\displaybreak[0]\\ \notag
A_{j,-i} &=  - \ii \Big( (g_{j,-i}g_i + \xi_{j,-i} \cdot \xi_i) \frac{\om_j}{b_j}  -  (g_{j,-i}g_j - \xi_{j,-i} \cdot  \xi_j) \frac{\om_i}{b_i}  \Big) \psi_{0j} \overline{\psi_{0i}},  \quad (j,i) \in I_<. 
\end{align}

\

{\sc Remark \ref{SectionFormalDerivation}.1.}\
Note, that we do not require that all second-order harmonics are mutually different, since if two or more coincide, corresponding to a subset $\Lambda\subset I$ of indices,  
and are different from all the other, then we consider in % the ansatz 
\eqref{Uapp} only one representative $(\ell,m)\in \Lambda$ and replace on the right-hand side of \eqref{zetapsi1ji} $(A_{\ell m}, B_{\ell m})^T$  by  $\sum_{(\lambda,\mu)\in\Lambda}(A_{\lambda\mu}, B_{\lambda\mu})^T$.  
\hfill$\square$

\

Commenting on the results obtained at the order $\ep^1$,
% or: so far
we note that \eqref{nonrestransport} describes only the macroscopic transport of the leading-($\ep^0$-)order amplitudes of the plane-waves $\ee_j$, $j\in J$, with the group velocity of the wave $\nabla\om_j$. 
Hence, at leading order the macroscopic amplitudes do not interact.
Of course, this is due to our non-resonance condition, viz.\ the first inequality in Notation \ref{approximationnotation}(3c). 
%
% Da muss man irgendwie auf den resonanzfall verweisen, z.b. mit Schneider Wayne, Giannoulis, und upcoming work

Quadratic interaction of the plane-wave-amplitudes can be detected only in the zeroth- and second-order harmonics of order $\ep$ via \eqref{zeta10}, \eqref{zetapsi1ji}.
Moreover, we see that \eqref{zeta10} contains also the time derivative of the leading-order non-oscillating part of the velocity potential at the surface. 
However,  we realize that this equation is not yet closed at the level $\ep$. The same applies for \eqref{zeta1j}, which relates the $\ep$-order corrections of the amplitudes to their leading order transport term.   

\

{\sc Remark \ref{SectionFormalDerivation}.2.}\  
Up to now the $\ep^2$-terms of $U_a$ in \eqref{Uapp} did not contribute to the expansion of the water waves problem as given in Corollary \ref{wwexpansion2}.
The same holds true for $\psi_{10}$. 
Hence, if we are interested only in the leading-order modulation equations we can consider the approximation $U_{a,1}$, defined as $U_a$ but without $\ep^2$-terms, and stop the derivation procedure here, choosing to set $\psi_{00}=\psi_{1j}=\psi_{10}=0$.
Thus, we have obtained up to now that the approximation 
\begin{align*}
U_{a,1}
%=  \sum_j \begin{pmatrix} \zeta_{0j}  \\  \psi_{0j}  \end{pmatrix}  \ee_j 
%+ \cc 
%+\ep \Big(  \sum_j  \begin{pmatrix} \zeta_{1j} \\ \psi_{1j} \end{pmatrix}  \ee_j
%+ \sum_{ji}  \begin{pmatrix} \zeta_{1ji} \\ \psi_{1ji} \end{pmatrix}  \ee_{ji}
%+ \cc 
%+ \begin{pmatrix} \zeta_{10} \\ \psi_{10}  \end{pmatrix} , 
%\Big)
= & \sum_j \begin{pmatrix} \ii\frac{\om_j}{b_j} \\  1  \end{pmatrix} \psi_{0j}  \ee_j 
+ \cc 
% \\&
+\ep 
\bigg(  \sum_j  \begin{pmatrix} \zeta_{1j} \\ 0 \end{pmatrix}  \ee_j + \sum_{ji}  \begin{pmatrix} \zeta_{1ji} \\ \psi_{1ji} \end{pmatrix}  \ee_{ji} + \cc + \begin{pmatrix} \zeta_{10} \\ 0   \end{pmatrix} \bigg)
\end{align*}
with the appearing functions determined 
%as above (and with the choices mentioned) 
by \eqref{nonrestransport}---\eqref{zetapsi1ji} (with $\psi_{1j}=\psi_{00}=0$)
%with $\psi_{0j}$, $\zeta_{1j}$, $\zeta_{10}$, $(\zeta_{1ji}, \psi_{1ji})^T$ given
%by \eqref{nonrestransport}, \eqref{zeta1j} (with $\psi_{1j}=0$), \eqref{zeta10} (with $\psi_{00}=0$), \eqref{zetapsi1ji}, respectively, 
satisfies the water waves problem \eqref{wwe} 
up to residual terms of order $\ep^2$, i.e. 
% in the sense of \eqref{formalconsistency}, i.e.
\begin{equation*} %\label{formalconsistency}
\partial_t U_{a,1} +\mcN_\sigma (U_{a,1}) 
%= \ep^2 \begin{pmatrix} r_1^1 \\ r_1^2 \end{pmatrix}
= \ep^2 (r_1^1,r_1^2)^T
\end{equation*}
with 
\begin{align*}
r_1^1 & =  \sum_j \pl_t' \zeta_{1j}  \ee_j +  \sum_{ji}  \pl_t'   \zeta_{1ji}   \ee_{ji}  + \cc + \pl_t'  \zeta_{10} - R_1^1, 
\displaybreak[0] \\
r_1^2 & = \sum_{ji}  \pl_t'   \psi_{1ji}   \ee_{ji}  + \cc  - R_1^2,
\end{align*}
and where $R_1^1$, $R_1^2$ are defined by \eqref{R11}, \eqref{R12} in the proof of Corollary \ref{epexpansion}. 
\hfill$\square$

\

%%%%%%%%%%%%%%%%%%%%%%%%%%%%%%%%%%%%%%%%%%%%%%%%%
%%% 2order

According to the comments made above, we expect that the interaction of three non-resonant modulated waves appears macroscopically at the first-order corrections to the leading-order amplitudes. Remark \ref{SectionFormalDerivation}.2.\ implies that in order to obtain the corresponding modulation equations, we have to carry our derivation procedure to the next order, i.e.\ we need to eliminate  also the $\ep^2$-terms of the expansions given in Corollary \ref{wwexpansion2}. 
% PUT DIFFERENTLY, BUT NOT TO BE PREFERED: 
%However, if one wants to obtain a better consistency result  with residual terms of formal order $\ep^3$ one has to eliminate identically also the $\ep^2$-terms in the expansion of  
%$\partial_t U_a +\mcN_\sigma (U_a) $ in Corollary \ref{wwexpansion2}.
%
%
Due to our assumptions in Notation \ref{approximationnotation}(3a,c), this has to be done separately for each $\ee_j$, $j\in J$, the zeroth-order harmonic and the higher-order harmonics. 
% WARUM DAS SO IST, WURDE EXPLIZIT IMMER NOCH NICHT ERKLAERT

Equating the coefficients of the first-order harmonics $\ee_j$ of order $\ep^2$ to zero, 
and using the equations \eqref{0order},  \eqref{nonrestransport}, \eqref{zeta1j}, \eqref{zeta10} 
and the identities \eqref{nablaomega},  \eqref{Hessiannew} with the notation \eqref{nablabo}, 
we obtain by elimination of $b_j\zeta_{2j} -\ii \om_j \psi_{2j}$, 
that the two equations can be written equivalently in the form 
%\begin{align*}
%%& \pl_t' \zeta_{1j} - \ii \omj  \zeta_{2j}  - g_j  \psi_{2j}  +  \ii g'_j \cdot \nabla' \psi_{1j}  + \tilde P_j =0,
%&  \ii \frac{ \om_j }{b_j}  \pl_t'   \psi_{1j} 
% % + \nablabo\om_j  \cdot \nabla'  \pl_t'   \psi_{0j}  
% - \nablabo\om_j  \cdot \nabla' ( \nabla\om_j \cdot   \nabla' \psi_{0j} )
%  - \ii \omj  \zeta_{2j}  - g_j  \psi_{2j}  +  \ii g'_j \cdot \nabla' \psi_{1j}  + \tilde P_j =0,
%\\
%%& \pl_t' \psi_{1j} - \ii \omj  \psi_{2j}  + b_j \zeta_{2j}   - \bond   2 \ii \bk_j \cdot \nabla' \zeta_{1j} + \tilde Q_j =0.
%& \pl_t' \psi_{1j} - \ii \omj  \psi_{2j}  + b_j \zeta_{2j}   + \bond   2 \frac{ \om_j }{b_j} \bk_j \cdot \nabla' \psi_{1j} - \bond   2 \ii \bk_j \cdot \nabla' ( \nablabo\om_j  \cdot \nabla' \psi_{0j}  ) + \tilde Q_j =0,
%\end{align*}
%which, with the same method as for the order $\ep$, i.e.\ using \eqref{nablaomega} and $\om_j^2 = b_j g_j$, and the abbreviation \eqref{nablabo}, can be written equivalently in the form
\begin{align}
\label{psi1j}
 &  \pl_t'   \psi_{1j} + \nabla\om_j \cdot \nabla' \psi_{1j}   = E_j ,
\displaybreak[0] \\
\label{zeta2j}
 &  \zeta_{2j}  =  \ii \frac{\om_j}{b_j}  \psi_{2j}  +  \nablabo\om_j  \cdot \nabla' \psi_{1j} + F_j ,
\end{align}
%% where we used \eqref{nablabo}, 
%with
%\begin{align*}
%& \tilde E_j 
%=  \ii \frac{ b_j }{2\om_j} \big( - \nablabo\om_j  \cdot \nabla' ( \nabla\om_j \cdot   \nabla' \psi_{0j} ) + \tilde P_j \big)   +  \bond  \ii \bk_j \cdot \nabla' ( \nablabo\om_j  \cdot \nabla' \psi_{0j} ) 
%- \frac12 \tilde Q_j,
% \\&
%\tilde F_j 
%=  - \ii \frac{1 }{2 \om_j} ( - \nablabo\om_j  \cdot \nabla' ( \nabla\om_j \cdot   \nabla' \psi_{0j} ) + \tilde P_j )   + \bond  \ii  \frac1{b_j}  \bk_j \cdot \nabla' ( \nablabo\om_j  \cdot \nabla' \psi_{0j}  ) 
%-  \frac1{2b_j}  \tilde Q_j .
%\end{align*}
%%%% MIT \zeta_{10}
%\begin{align*}
%& \tilde E_j 
%=  \ii \frac12   \nabla' \cdot \mcH_\om(\bk_j) \nabla' \psi_{0j} 
%+  \ii \frac{ b_j }{2\om_j}  (g_j^2 - |\xi_j|^2) \zeta_{10}  \psi_{0j} 
% -  \ii \bk_j \psi_{0j} \cdot  \nabla' \psi_{00} 
%-  \ii \frac{ b_j }{2\om_j}  C_j   + \frac12 D_j,
% \\&
%\tilde F_j 
%=  - \ii \frac{1 }{2 b_j }  \nabla' \cdot \mcH_\om(\bk_j) \nabla' \psi_{0j} 
%+  \ii \frac{1 }{b_j }  \bond  \Big( 2 \bk_j \cdot \nabla' ( \nablabo\om_j  \cdot \nabla'  \psi_{0j} ) + \frac{\om_j}{b_j}\Delta'  \psi_{0j} \Big) 
%\\&\qquad
%- \ii \frac{1 }{2 \om_j}  (g_j^2 - |\xi_j|^2) \zeta_{10}  \psi_{0j}  
%+  \ii \frac{1 }{2 \om_j}  C_j  +  \frac1{2b_j}  D_j .
%\end{align*}
%%% OHNE \zeta_{10}
with
%\begin{align} \label{tildeEj}
% E_j 
%= &\  \ii \frac12   \nabla' \cdot \mcH_\om(\bk_j) \nabla' \psi_{0j} 
%-  \ii \frac{ b_j }{2\om_j}  (g_j^2 - |\xi_j|^2) \psi_{0j}   \pl_t' \psi_{00} 
% -  \ii \bk_j \psi_{0j} \cdot  \nabla' \psi_{00} 
%\\&\  \notag
%+  \ii \frac{ b_j }{2\om_j}  (g_j^2 - |\xi_j|^2) B_0 \psi_{0j} 
%-  \ii \frac{ b_j }{2\om_j}  C_j   
%+ \frac12 D_j,
%% \\ \notag
%\end{align} 
\begin{align} \label{tildeEj}
 E_j 
= &\  \ii \frac12   \nabla' \cdot \mcH_\om(\bk_j) \nabla' \psi_{0j} 
-  \ii  \psi_{0j} \Big( \frac{ b_j }{2\om_j}  (g_j^2 - |\xi_j|^2)   \pl_t' +  \bk_j  \cdot  \nabla'\Big) \psi_{00}
+ \tilde E_j
\end{align} 
where
\begin{align}\label{truetildeEj}
\tilde E_j =  \ii \frac{ b_j }{2\om_j}  (g_j^2 - |\xi_j|^2) B_0 \psi_{0j} 
-  \ii \frac{ b_j }{2\om_j}  C_j   
+ \frac12 D_j,
\end{align}
and 
\begin{align*}
F_j 
= & - \ii \frac{1 }{2 b_j }  \nabla' \cdot \mcH_\om(\bk_j) \nabla' \psi_{0j} 
+  \ii \frac{1 }{b_j }  \bond  \Big( 2 \bk_j \cdot \nabla' ( \nablabo\om_j  \cdot \nabla'  \psi_{0j} ) + \frac{\om_j}{b_j}\Delta'  \psi_{0j} \Big) 
\displaybreak[0] \\& \notag
+  \ii \frac{1 }{2 \om_j}  (g_j^2 - |\xi_j|^2)   \psi_{0j}   \pl_t' \psi_{00}
- \ii \frac{1 }{2 \om_j}  (g_j^2 - |\xi_j|^2)  B_0   \psi_{0j}  
+  \ii \frac{1 }{2 \om_j}  C_j  
+  \frac1{2b_j}  D_j .
% \end{align}
\end{align*}
We note that the solution of \eqref{psi1j} closes now equation \eqref{zeta1j} and determines $\zeta_{1j}$.
%, which does not appear in \eqref{psi1j}. 
In particular, $\tilde E_j$ and 
the three last terms of $F_j$ consist of cubic products of $\psi_{0j}$ (cf.\ Proposition \ref{finalexpansion}, \eqref{0order}, \eqref{zetapsi1ji} and
\eqref{As1order}).
Hence, $E_j$, $F_j$ depend only on $\psi_{0j}$ and $\psi_{00}$. 

While $\psi_{0j}$ is determined by \eqref{nonrestransport}, we obtain from the coefficient of the $\ee^0$-term of order $\ep^2$ in the 
% $\psi_{00}$ is determined by the equation corresponding to the zeroth-order harmonic of order $\ep^2$ of the 
first expansion of Corollary \ref{wwexpansion2} 
and \eqref{zeta10},  
%which gives, 
by elimination of $\pl_t'\zeta_{10}$, 
the second-order inhomogeneous wave equation
\begin{equation}\label{psi00}
{\pl_t'}^2  \psi_{00} -  \sqrt{\mu} \Delta' \psi_{00}   =  \pl_t'   B_0 - C_0   =   \sum_j  \Big(  (g_j^2 - |\xi_j|^2) \pl_t' + 2  \frac{\om_j}{b_j} \xi_j  \cdot  \nabla'  \Big)  |\psi_{0j}|^2.
\end{equation}
% NAME: MEAN FLOW !!! EXPLAIN !!!
%with 
%\begin{equation*}
% \pl_t'   B_0  - C_0 =   \sum_j  \Big(  (g_j^2 - |\xi_j|^2) \pl_t' + 2  \frac{\om_j}{b_j} \xi_j  \cdot  \nabla'  \Big)  |\psi_{0j}|^2.
%\end{equation*}
Its solution determines then $\zeta_{10}$ through \eqref{zeta10}. 
Moreover, for the coefficient of $\ee^0$  at order $\ep^2$ in the second expansion of Corollary \ref{wwexpansion2} we obtain
\begin{equation}\label{zeta20}
\pl_t' \psi_{10} + \zeta_{20} = D_0 = \sum_j  \big(     \ii    ( \xi_j - g_j g'_j ) \cdot \nabla' \psi_{0j}  + ( g_j^2  - |\xi_j|^2 )  \psi_{1j} \big)  \overline{ \psi_{0j} }   + \cc.
\end{equation}  

Finally, equating the coefficients of the higher harmonics  to zero, we obtain 
\begin{align}\label{zetapsi2higher}
\begin{pmatrix}  \zeta_{2ji} \\ \psi_{2ji} \end{pmatrix}  
&=  \frac1{\om_{ji}^2 - b_{ji}g_{ji}} \begin{pmatrix} \ii \om_{ji} & - g_{ji}  \\ b_{ji}  &  \ii \om_{ji} \end{pmatrix}  
\begin{pmatrix} C_{ji} - \pl_t'   \zeta_{1ji} -  \ii  g'_{ji}  \cdot  \nabla'   \psi_{1ji} \\  D_{ji} - \pl_t'   \psi_{1ji} +  \bond  2 \ii \bk_{ji} \cdot  \nabla'   \zeta_{1ji} \end{pmatrix}, 
\displaybreak[0] \\ \notag
\begin{pmatrix}  \zeta_{2jik} \\  \psi_{2jik} \end{pmatrix} 
&= \frac1{\om_{jik}^2 - b_{jik}g_{jik}} \begin{pmatrix} \ii \om_{jik} & - g_{jik}  \\ b_{jik}  &  \ii \om_{jik} \end{pmatrix}  
\begin{pmatrix}  C_{jik} \\  D_{jik} \end{pmatrix} ,
\end{align}
where the right-hand sides depend only on the functions $\psi_{0j}$ (and its first order derivatives)  and $\psi_{1j}$, see Appendix and \eqref{zeta1j}, \eqref{zetapsi1ji}. 
The case where two or more higher harmonics coincide is treated as in Remark \ref{SectionFormalDerivation}.1.
In the case where only the first non-resonance condition of Notation \ref{approximationnotation}(3c)
%, but not the second 
holds true, i.e.\ when some third-order harmonic $\ee_{jik}$ equals $\ee_j$, the corresponding $C_{jik}$, $D_{jik}$ of the expansions in Corollary  \ref{wwexpansion2} contribute to \eqref{psi1j}, \eqref{zeta2j} in the same way as the $C_j$, $D_j$, and the $\zeta_{2jik}, \psi_{2jik}$ can be set equal to zero. 
According to the Appendix, the terms $C_{jik}, D_{jik}$ are cubic products of the leading-order amplitudes $\psi_{0j}$, $j\in J$.

\

In analogy to the procedure in Remark \ref{SectionFormalDerivation}.2.\ at the level $\ep$, we can close the system of modulation equations derived so far, 
by setting  $\psi_{10}= \psi_{2j} = \psi_{20} =0$.
%
%Summarizing, 
Thus, 
we have obtained that the second order approximation
\begin{align}\label{Uappfinal}
U_a = 
& \sum_j \begin{pmatrix} \ii\frac{\om_j}{b_j}  \\  1 \end{pmatrix} \psi_{0j} \ee_j 
+ \cc 
+ \begin{pmatrix} 0 \\ \psi_{00} \end {pmatrix}
\displaybreak[0] \\ \notag & 
+ \ep \Big( \sum_j  \begin{pmatrix} \zeta_{1j} \\ \psi_{1j} \end{pmatrix}  \ee_j
+ \sum_{ji}  \begin{pmatrix} \zeta_{1ji} \\ \psi_{1ji} \end{pmatrix}  \ee_{ji}
+ \cc 
+ \begin{pmatrix} \zeta_{10} \\ 0 \end{pmatrix} \Big) 
\displaybreak[0]\\ \notag &
+ \ep^2 \Big( \sum_j  \begin{pmatrix} \zeta_{2j} \\ 0 \end{pmatrix}   \ee_j
+ \sum_{ji}   \begin{pmatrix} \zeta_{2ji} \\ \psi_{2ji} \end{pmatrix}  \ee_{ji}
+ \sum_{jik}    \begin{pmatrix} \zeta_{2jik} \\ \psi_{2jik} \end{pmatrix} \ee_{jik}
+ \cc 
+  \begin{pmatrix} \zeta_{20} \\ 0 \end{pmatrix} \Big)
\end{align}
with the appearing macroscopic functions determined as above (with $\psi_{2j}=0$ in \eqref{zeta2j} and $\psi_{10}=0$ in \eqref{zeta20}) 
is consistent with the water waves problem \eqref{wwe}, 
in the sense that it satisfies
\begin{equation} \label{UaConsistency}
\partial_t U_a +\mcN_{\ep,\sigma} (U_a) 
= \ep^3 ( r_2^1 , r_2^2 )^T, 
%= \ep^3 \begin{pmatrix} r_2^1 \\ r_2^2 \end{pmatrix},
\end{equation}
where
\begin{align}\label{r21}
r_2^1 & = \sum_j   \pl_t' \zeta_{2j}  \ee_j + \sum_{ji} \pl_t' \zeta_{2ji}  \ee_{ji}  + \sum_{jik} \pl_t'   \zeta_{2jik}  \ee_{jik}  + \cc +  \pl_t'  \zeta_{20}  - R_2^1, 
\displaybreak[0] \\\notag
r_2^2 & = \sum_{ji} \pl_t' \psi_{2ji}  \ee_{ji}  + \sum_{jik} \pl_t'   \psi_{2jik}  \ee_{jik}  + \cc  - R_2^2
\end{align}
with $R_2^1$, $R_2^2$ as in Corollary \ref{epexpansion}.

\

In particular, all macroscopic functions can be determined if we solve the system 
\begin{align} \label{mainsystem}
\begin{cases}
\pl_t' \psi_{0j}   +  \nabla\om_j \cdot   \nabla' \psi_{0j} = 0,
\\
{\pl_t'}^2  \psi_{00} -  \sqrt{\mu} \Delta' \psi_{00}   =  \sum_j  \Big(  (g_j^2 - |\xi_j|^2) \pl_t' + 2  \frac{\om_j}{b_j} \xi_j  \cdot  \nabla'  \Big)  |\psi_{0j}|^2,
\\
\pl_t'   \psi_{1j} + \nabla\om_j \cdot \nabla' \psi_{1j}   = E_j
\end{cases}
\end{align}
with $j\in J=\{1,2,3\}$ and 
% $E_j$ as in \eqref{tildeEj}. 
\begin{align} \tag{\ref{tildeEj}}
 E_j 
= &\  \ii \frac12   \nabla' \cdot \mcH_\om(\bk_j) \nabla' \psi_{0j} 
-  \ii  \psi_{0j} \Big( \frac{ b_j }{2\om_j}  (g_j^2 - |\xi_j|^2)   \pl_t' +  \bk_j  \cdot  \nabla'\Big) \psi_{00}
+ \tilde E_j
\end{align} 
with $\tilde E_j$ as in \eqref{truetildeEj}.

% COMMENTS ON THE $\ep^0$ AND $\ep^2$ EQUATIONS AS PERFORMED FOR $\ep^1$

% MUSS MAN HIER NOCHMAL DIE RESONANZ-ASSUMPTIONS ERWAEHNEN ???

\

We close this section with the remark mentioned at the beginning of the section, concerning an alternative derivation procedure 
in comparison to the one presented here. 
% mentioned at the beginning of this section.
%
% SAGEN HIER ODER IN DER INTRODUCTION, DASS DIESER APPROACH FUER DIE HERLEITUNG VON DAVEY-STEWARTSON VON INTERESSE  IST. 

\

{\sc Remark \ref{SectionFormalDerivation}.3.}\ 
In the approach presented above, we derived \emph{consecutively} the macroscopic equations that make the first-order-harmonic terms vanish \emph{at each} order $\ep^1$ and $\ep^2$ \emph{separately}.
% and completely}.
%
In a different approach, which in the case of a single carrier wave leads to the so-called \emph{Benney-Roskes system} (\cite{BR}, cf.\ also \cite[\S 8.2.5]{Lannes} and the references given therein), 
one considers \emph{jointly} the coefficients of the first-order-harmonics of orders $\ep^1$ and $\ep^2$ and derives conditions such that these joint coefficients vanish up to terms of order $O(\ep^3)$. 
More precisely, starting from the coefficients of the first-order harmonics of orders $\ep^1$, $\ep^2$ of Corollary \ref{wwexpansion2}, with \eqref{0order}, \eqref{nablaomega}, \eqref{nablabo} 
and \eqref{Hessiannew} one can write equivalently 
\begin{multline*}
( \pl_t'  + \nabla\om_j  \cdot   \nabla' ) \psi_{0j} 
- \ep \ii \frac{1}{2 \om_j}  ( \pl_t' + \nabla\om_j \cdot \nabla' ) ( b_j  \zeta_{1j}  + \ii \om_j \psi_{1j} ) 
\\
+  \ep \ii \frac{b_j}{2 \om_j} \nablabo\om_j \cdot \nabla' ( b_j \zeta_{1j} - \ii \om_j \psi_{1j}  - b_j \nablabo\om_j \cdot   \nabla' \psi_{0j}  ) - \ep E_j  = O(\ep^2), 
\end{multline*}
\begin{multline*}
( \pl_t'  + \nabla\om_j \cdot   \nabla' ) \psi_{0j} + b_j \zeta_{1j}  - \ii \omj  \psi_{1j}  -  b_j \nablabo\om_j \cdot \nabla' \psi_{0j} 
\\
+  \ep ( b_j \zeta_{2j}  - \ii \omj  \psi_{2j}  )  + \ep ( \pl_t' \psi_{1j}   -  \bond   2 \ii \bk_j \cdot \nabla' \zeta_{1j}  + Q_j )  = O(\ep^2)
\end{multline*}
with $E_j$ as in \eqref{tildeEj} (although typically with the second and fourth term written together as  $ \ii \frac{ b_j }{2\om_j}  (g_j^2 - |\xi_j|^2) \zeta_{10} \psi_{0j}$, see \eqref{zeta10}).
%\begin{align*}
% \tilde E_j 
% &=  \ii \frac{ b_j }{2\om_j} \big( - \nablabo\om_j  \cdot \nabla' ( \nabla\om_j \cdot   \nabla' \psi_{0j} ) + \tilde P_j \big)   +  \bond  \ii \bk_j \cdot \nabla' ( \nablabo\om_j  \cdot \nabla' \psi_{0j} ) - \frac12 \tilde Q_j,
%\\&
%= -  \ii \frac{ b_j }{2\om_j} \nablabo\om_j  \cdot \nabla' ( \nabla\om_j \cdot   \nabla' \psi_{0j} )  +  \bond  \ii \bk_j \cdot \nabla' ( \nablabo\om_j  \cdot \nabla' \psi_{0j} ) 
% +  \ii \frac{ b_j }{2\om_j}  \tilde P_j    - \frac12 \tilde Q_j,
%\\
%& = -  \ii \frac{ b_j }{2\om_j} \nablabo\om_j  \cdot \nabla' (b_j \nablabo\om_j \cdot   \nabla' \psi_{0j} )   +  \ii \frac{ b_j }{2\om_j}  \tilde P_j    - \frac12 \tilde Q_j
%\\& \quad
%-  \ii \bond \nablabo\om_j  \cdot \nabla' ( \xi_j \cdot   \nabla' \psi_{0j} )  
%+ \ii \bond \bk_j \cdot \nabla' ( \nablabo\om_j  \cdot \nabla' \psi_{0j} ) 
%\\
%& = -  \ii \frac{ b_j }{2\om_j} \nablabo\om_j  \cdot \nabla' (b_j \nablabo\om_j \cdot   \nabla' \psi_{0j} )   +  \ii \frac{ b_j }{2\om_j}  \tilde P_j    - \frac12 \tilde Q_j
%\end{align*}
%since $a\cdot \nabla (b\cdot\nabla f) = b\cdot \nabla (a\cdot \nabla f)$

At leading order in $\ep$ we obtain from the two equations  (consecutively)
\begin{align*}
&  \pl_t'  \psi_{0j}  + \nabla\om_j \cdot   \nabla'  \psi_{0j} = O(\ep), \quad    b_j \zeta_{1j} - \ii \omj  \psi_{1j}   -  b_j \nablabo\om_j \cdot \nabla' \psi_{0j} = O(\ep)
\end{align*}
(in analogy to \eqref{nonrestransport}, \eqref{zeta1j}).  
Thus, requiring that the equation on the right is satisfied \emph{exactly} (i.e., without $O(\ep)$-terms), we obtain \eqref{zeta1j} and 
\begin{align*}
&  ( \pl_t'  + \nabla\om_j  \cdot   \nabla' ) \psi_{0j} 
- \ep \ii \frac{1}{2 \om_j}  ( \pl_t' + \nabla\om_j \cdot \nabla' ) ( b_j  \zeta_{1j}  + \ii \om_j \psi_{1j} )  = \ep E_j + O(\ep^2), 
\displaybreak[0] \\ & 
% ( \pl_t'  + \nabla\om_j \cdot   \nabla' ) \psi_{0j}  +  \ep ( b_j \zeta_{2j}  - \ii \omj  \psi_{2j}  )  + \ep ( \pl_t' \psi_{1j}   -  \bond   2 \ii \bk_j \cdot \nabla' \zeta_{1j}  + \tilde Q_j ) = O(\ep^2)
b_j \zeta_{2j}  - \ii \omj  \psi_{2j}  = -\frac1\ep ( \pl_t'  + \nabla\om_j \cdot   \nabla' ) \psi_{0j}  -  \pl_t' \psi_{1j} + \bond   2 \ii \bk_j \cdot \nabla' \zeta_{1j} -  Q_j  +   O(\ep).
\end{align*}
The second equation plays the same r\^ole as \eqref{zeta2j}.  Choosing $b_j \zeta_{2j}  + \ii \omj  \psi_{2j}$ arbitrarily, we can determine $\zeta_{2j}$ and $\psi_{2j}$. 
% THE SAME ??? 

Now, we require  
\begin{equation}\label{BRchoice}
 (\pl_t' + \nabla\om_j \cdot \nabla' ) ( b_j  \zeta_{1j}  + \ii \om_j \psi_{1j} ) =0.
\end{equation}
Together with \eqref{zeta1j}, this determines $\zeta_{1j}$ and $\psi_{1j}$ (in difference to our approach above, using \eqref{zeta1j}, \eqref{psi1j} for that),
%to determine $\zeta_{1j}$ and $\psi_{1j}$
and yields
\begin{align}\label{BR}
 ( \pl_t'  + \nabla\om_j  \cdot   \nabla' ) \psi_{0j} = \ep E_j + O(\ep^2).
\end{align}
According to Lannes,  see \cite[\S 8.2.3., fn.\ 9]{Lannes} and \cite{Lannes98},
% and \cite{L295}, 
% this choice 
the choice \eqref{BRchoice} 
is the only possible in order to avoid the secular growth of  $\ep(b_j \zeta_{1j} + \ii \om_j  \psi_{1j} )$ on time scales of order $t'' = \ep t' = \ep^2 t$.
This implies that the results obtained in the present paper are relevant only on time scales of order $t'=\ep t$.
Note, that with 
%the choice 
\eqref{BRchoice} we obtain from \eqref{zeta1j} and \eqref{BR}
\begin{equation*}%\label{psi1jBR} 
( \pl_t'  + \nabla\om_j \cdot \nabla' ) \psi_{1j}  = \ep  \ii \frac{b_j}{2\omj} \nablabo\om_j  \cdot \nabla' E_j +\O(\ep^2)
\end{equation*}
instead of \eqref{psi1j} with our approach. However, in both cases we obtain
\begin{equation*}
 (\pl_t'   +  \nabla\om_j  \cdot \nabla' ) (\psi_{0j} + \ep \psi_{1j})   = \ep E_j + \O(\ep^2).
 \end{equation*} 

The system 
%\begin{align*}
%&  ( \pl_t'  + \nabla\om_j  \cdot   \nabla' ) \psi_{0j} = \ep \tilde E_j,
%\\ 
%& \pl_t'  \zeta_{10} + \sqrt{\mu} \Delta' \psi_{00}  =  -  2 \sum_j  \frac{\om_j}{b_j}  \bk_j \cdot\nabla' |\psi_{0j}|^2,
%\\
%& \pl_t' \psi_{00} + \zeta_{10} =   \sum_j  (g_j^2-|\xi_j|^2) |\psi_{0j}|^2
%\end{align*}
\begin{align*}
\begin{cases}
( \pl_t'  + \nabla\om_j  \cdot   \nabla' ) \psi_{0j} = \ep E_j,
\\ 
\pl_t'  \zeta_{10} + \sqrt{\mu} \Delta' \psi_{00}  =  -  2 \sum_j  \frac{\om_j}{b_j}  \bk_j \cdot\nabla' |\psi_{0j}|^2,
\\
\pl_t' \psi_{00} + \zeta_{10} =   \sum_j  (g_j^2-|\xi_j|^2) |\psi_{0j}|^2
\end{cases}
\end{align*}
with $E_j$ as in \eqref{tildeEj} (in the form mentioned above) can be seen as the
\emph{Benney-Roskes system for three capillary-gravity water waves}. The original Benney-Roskes system was derived for a single gravity water wave (see \cite{BR} and \cite[(8.34)]{Lannes}),
and can be recovered from the system above by setting the surface tension to zero and assuming that two of the three waves are identically vanishing. For its well-posedness we refer to \cite{PonceSaut05}. 
Note, that in contrast to the system \eqref{mainsystem} that we consider in this paper, the equations of the Benney-Roskes system are coupled.     
%
% The coefficients of the higher-order harmonics are obtained as presented above. 
% IST DAS SICHER ??? SOLLTE MAN FESTSTELLEN
\hfill$\square$

%
%%% WAS SIND DIE UNTERSCHIEDE LETZTENDLICH FUER UNSEREN APPROACH ? NICHT DAS DER AM ENDE NICHT GILT...

%%%%%%%%%%%%%%%%%%%%%%%%%%%%%%%%%%%%%%%%%%%%%%%%%%%%%%%%%
% \newpage

\section{Justification in the case without surface tension} \label{SectionJustification}

For the justification of the modulation equations derived in Section \ref{SectionFormalDerivation} in the case of pure gravity waves (i.e.\ without surface tension, $\sigma = \bond = 0$)
we use the well-posedness  result for gravity water waves of finite depth obtained by Lannes and Alvarez-Samaniego in \cite{Lannes05, ASL08, ASL08b, Lannes}
% in the form 
as presented in \cite[Theorems 4.16, 4.18]{Lannes}. 
% MANCHES VON DEM FOLGENDEN BRAUCHT MAN EVENTUELL SCHON IN DER INTRODUCTION BZW. SOLLTE ES DORT HINTUN
%
%
There, the well-posedness 
% of the gravity water waves problem of finite depth  
is established with respect to the energy norm 
\begin{equation}\label{energyN}
\mcE_\ep^N (U) = |\mfP \psi|_{H^{t_0+3/2}}^2 + \sum_{\alpha\in\N_0^d, |\alpha| \le N} | \partial^\alpha \zeta |_2^2 + | \mfP \psi_{\ep, (\alpha)} |_2^2, \qquad N\in \N,
\end{equation}
with  $U=(\zeta,\psi)^T$, 
\begin{align}
\notag
& 
 \psi_{(0)}=\psi, \qquad  
\psi_{\ep, (\alpha)} = \pl^\alpha \psi - (w[\ep\zeta] \ep \psi) \pl^\alpha\zeta\quad\text{for}\quad \alpha\neq 0,
\\
% \label{w}
\label{mfP}
% & w[\ep\zeta] \psi = \frac{ \mcG[\ep\zeta]\psi +\ep\nabla\zeta\cdot\nabla\psi}{1+\ep^2|\nabla\zeta|^2}
& w[\zeta] \psi = \frac{ \mcG[\zeta]\psi + \nabla\zeta\cdot\nabla\psi}{1+ |\nabla\zeta|^2}
\qquad\text{and}\qquad 
\mfP=\frac{|D|}{(1+\sqrt{\mu}|D|)^{1/2}} 
\end{align}
(recall the Fourier-multiplier notation \eqref{FourierMultiplier}).

$H^s(\R^d)$ ($s\in\R$) are the fractional Sobolev spaces (see, e.g., \cite{Hoermander})
\begin{equation}\label{Hs}
H^s(\R^d) = \{ u \in \mfS'(\R^d): |u|_{H^s} = |\Lambda^s u |_2<\infty\},\qquad |u|_2^2 = \int_{\R^d} |u(X)|^2 d\, X,
\end{equation}
where $\mfS'(\R^d)$ is the space of tempered distributions and
$\Lambda  
= \langle D \rangle 
= (1+|D|^2)^{1/2}$ is the fractional derivative.

Associated to the energy norm \eqref{energyN} is the space of time-dependent functions that remain bounded with respect to this norm up to the time $T>0$
\begin{equation*}
E^N_{\ep,T} =\{ U \in C([0,T]; H^{t_0+2} \times \dot H^2 (\R^d)) : \mcE^N_\ep (U(\cdot))\in L^\infty( [0,T] )\}
\end{equation*}
and the space of initial data
\begin{equation*}
E^N_{\ep,0} =\{ U^0 \in H^{t_0+2} \times \dot H^2 (\R^d) : \mcE^N_\ep(U^0) < \infty \}.
\end{equation*}
% where 
Here, $\dot H^{s+1}(\R^d)$ ($s\in\R$) are the Beppo-Levi  topological vector spaces
\begin{equation*}
\dot H^{s+1}(\R^d) = \{f\in L^2_{loc}(\R^d): \nabla f \in H^s(\R^d)^d\}
\end{equation*}
endowed with the semi-norm $| f |_{\dot H^{s+1}} = | \nabla f |_{H^s}$, which implies that $\dot H^{s+1}(\R^d)/\R$ are Banach spaces 
% with respect to the quotient norm  (???)
(see, e.g., \cite[\S 2.1.2.]{Lannes} and \cite{DenyLions}). 
%
% Hier P und Alignac's good unknown erklaeren
% 

\

% Following Lannes, the 
The
well-posedness result of \cite[Th.\  4.16, 4.18]{Lannes} can be adapted to the case of deep water 
%($\mu\gg1$) 
($\mu\ge1$)
of finite depth (with flat bottom, and full transversality of the 
waves if $d=2$) as follows (see also Remark \ref{SectionJustification}.3.\ below).
%This result, adapted to our case here 
%(flat bottom and full transversality of the waves, i.e.\ in the notation of \cite{Lannes}: $b=0$, $\beta=0$ and $\gamma=1$),  states the following:

\begin{theorem} \label{JustificationTheorem}
Let $t_0>d/2$, $t_0\ge 1$,  $N\ge t_0 + t_0\vee2 + 3/2$, $1\le \mu\le \mu_{\max} < \infty $, and $0<\ep \le 1$.
Assume that $U^0 = (\zeta^0,\psi^0)^T \in E_{\ep,0}^N$ with 
\begin{equation}\label{initialdataconditions}
1 - \ep |\zeta^0|_{\infty} \ge h_{\min}>0
%, \qquad \exists \ a_0>0:\quad \mathfrak{a}(U^0)\ge a_0. 
\qquad\text{and}\qquad 
\mathfrak{a}_\ep (U^0)  \ge a_0 >0,
\end{equation} 
where $\mathfrak{a}_\ep (U^0)$ is defined in Remark 4.2 below.
 Then, there exists $T>0$ and 
 %for every $\ep\in(0,\ep_0)$ 
 a unique solution $U=(\zeta,\psi)^T \in E^N_{\ep, T/\ep}$ to the water waves problem 
% to be adapted;  (4.1) of \cite{Lannes}; write it with tilde over the N's maybe
\begin{align}\label{wweTheorem}
\partial_t U +  \mcN_{\ep,0} (U)  = 0,\qquad
\mcN_{\ep, 0} (U) & = \begin{pmatrix}  - \mcG[\ep\zeta]\psi 
\\
 \zeta 
%+ \frac\ep2 \left( |\nabla\psi|^2 - \frac{( \mcG[\ep\zeta]\psi + \ep  \nabla\zeta\cdot \nabla \psi)^2}{1 + \ep^2 |\nabla\zeta|^2} \right)
+ \frac\ep2 |\nabla\psi|^2 - \frac\ep2 \frac{( \mcG[\ep\zeta]\psi + \ep  \nabla\zeta\cdot \nabla \psi)^2}{1 + \ep^2 |\nabla\zeta|^2} 
 \end{pmatrix}
\end{align} 
with initial data $U^0$, and 
 % Moreover,  $\frac1T= c_1$ and $ \sup_{t\in [0,T/\ep]} \mcE^N (U(t))  = c_2$, 
\begin{equation*}
\frac1{T} = c_1,  \qquad  \sup_{t\in [0,T/\ep]} \mcE^N_\ep (U(t))  = c_2, 
\end{equation*}
where the constants $c_j = C(\mcE^N_\ep(U^0), \mu_{\max}, h_{\min}^{-1}, a_0^{-1})$, $j=1,2$, are non-decreasing functions of their arguments. 
% and $a_0>0$ is given in Remark \ref{SectionJustification}.2.\ below.

Furthermore, if  there exists $U_{app}=(\zeta_{app},\psi_{app})^T \in E^N_{\ep, T/\ep}$ 
such that 
\begin{equation*}
% \inf_{t\in [0,T/\ep]}   \inf_{x\in\R^d} (1+\frac{\ep}{\sqrt{\mu}} \zetaa) 
% \inf_{t\in [0,T/\ep]}  (1-\ep |\zeta_{app}|_\infty)
1-\ep \sup_{t\in [0,T/\ep]} |\zeta_{app}(t)|_\infty
> 0 \qquad\text{and}\qquad \pl_t U_{app} + \mcN_{\ep, 0} (U_{app}) = (r^1,r^2)^T 
\end{equation*}
with $(r^1,\mfP r^2)\in L^\infty ([0,T/\ep]; H^N(\R^d)^2)$,
and a constant $c_{app}>0$ such that
\begin{equation*}
\sup_{t\in [0,T/\ep]} \mcE^N_\ep (U_{app}(t))\le c_{app},
\end{equation*}
then the error 
$\mfe = U -U_{app}$ satisfies for all $t\in [0,T/\ep]$
\begin{equation*}
\mcE^{N-1}_\ep (\mfe(t))^{1/2} \le C( 
% T, 
%%% IST IN LANNES BOOK UND IM BEWEIS NICHT DRIN; NOCHMAL CHECKEN WARUM ICH DAS REINGEMACHT HABE
%c_1,
%%% UNNOETIG, DA DAS GLEICHE WIE c_2
c_2, c_{app}) \Big(\mcE^{N-1}_\ep(\mfe(0))^{1/2} + t |(r^1,\mfP r^2)|_{L^\infty ([0,t]; H^N)} \Big)
\end{equation*}
with $C$ a non-decreasing function of its arguments.
\end{theorem}

{\sc Remark \ref{SectionJustification}.1.}\ 
We introduced the index $\ep$ in the notation of the energy norm, the corresponding spaces and $\mcN_{\ep,0}$, 
in order to point out their dependence on this parameter.
(Our justification result below uses the above theorem with $\ep=1$.) 
However, it is important to note that the time $T$ and the constants $c_j$, $j=1,2$, in the statement of the theorem are independent of 
$\ep\in (0,1]$, 
since $\mcE_\ep^N(U)$ is bounded with respect to $\ep$.
This is shown easily  
%which is shown easily 
 in the case 
%for
$U= (\zeta,\psi)^T \in H^{N+1} \times \dot H^{N+1}(\R^d)$. Indeed, from \eqref{mfP} we obtain
\begin{equation}\label{estimatemfP}
|\mfP u|_{H^s} \le |\nabla u|_{H^{s}} \quad \forall\ s\in \R,\ \mu\ge 0,
\end{equation}
and hence, with  Lemma \ref{ProductEstimates} (1) and (4), and \eqref{mcGestimate} for $\mu\ge 1$ and  $t_0\ge 1$
% MAN SOLLTE AUCH DORT VERWENDEN, DASS $\mu\ge 1$
\begin{equation*}
\big|\mfP \big((w[\ep\zeta]\psi) \pl^\alpha\zeta\big)\big|_2 \le C(h_{\min}^{-1},\mu_{\max}, |\ep \zeta|_{H^{t_0+1}})  |\nabla\psi|_{H^{t_0}} |\pl^\alpha\zeta|_{H^1}.
% \qquad t_0\ge 1.
\end{equation*}
This yields in particular, 
\begin{equation*}
\mcE^N_\ep (U) = \mcE^N_0 (U) + O(\ep)\quad\text{as $\ep\to0$} 
\end{equation*} 
but also 
\begin{equation}\label{energyestimate}
% \mcE^N_\ep (U)  \le C \left( h_{\min}^{-1} ,  \mu_{\max} , | \ep \zeta|_{H^{t_0 + 1}} ,|\ep \nabla \psi |_{H^{t_0}} \right) ( |\zeta |_{H^{N+1}}^2 + |\nabla\psi|_{H^N}^2) 
\mcE^N_\ep (U)  \le C \left( h_{\min}^{-1} ,  \mu_{\max} , | \ep 
%( \zeta, \psi) 
U 
|_{H^{t_0 + 1} \times \dot H^{t_0+1}} \right) |
%(\zeta,\psi)
U
 |_{H^{N+1}\times \dot H^{N+1}}^2 
%, \quad N\ge t_0+3/2.
\end{equation}
for $N\ge t_0+3/2$. 
\hfill$\square$

\

%%%%%%%%%%%%%%%%%%%%%%%%%%%%%%%%%%%%%%%%%

{\sc Remark \ref{SectionJustification}.2.}\ 
The first condition in \eqref{initialdataconditions} implies that the water height never vanishes. 
(Recall that $\zeta^0\in H^{t_0}(\R^d) \subset C\cap L^{\infty}(\R^d)$ for $t_0>d/2$.) 
%From a physical point of view this condition is always given in the dispersive deep water case considered here.
%Mathematically 
For given $\zeta^0$, it is uniformly satisfied for $0<\ep\le \ep_0$, with $0<\ep_0\le 1$, if $1 - \ep_0 |\zeta^0|_{\infty} \ge h_{\min}>0$.
However, note that $h_{\min}>0$ influences the existence time $T>0$ of solutions to the water waves problem. 

Moreover, the well-posedness of the water-waves problem and the existence time of its solution depend crucially on the validity of the (strict-)hyperbolicity condition 
\begin{equation}\label{Levycondition}
\mathfrak{a}_\ep (U^0)  = \mathfrak{a}_1 (\ep U^0) = 1 -   \mathfrak{b}_1 (\ep U^0)  \ge a_0 >0
\end{equation}
with
%\begin{align*}
%\mathfrak{a}_\ep(U) 
%& = 1 - \ep w [\ep \zeta] \mcN_{\ep,0}^2(U) - \ep d_\zeta  w (\mcN_{\ep, 0}^1(U) )\psi 
%+ \ep^2 \left(\nabla\psi - \ep (w[\ep\zeta] \psi)\nabla\zeta\right)  \cdot \nabla ( w[\ep\zeta] \psi)
%\\
%& = 1 - w [\ep \zeta] \mcN_{1,0}^2(\ep U) 
%% - d_\zeta  w (\mcN_{\ep, 0}^1(U) ) \ep \psi 
%- d_{\ep\zeta}  w (\mcN_{1, 0}^1(\ep U) ) \ep \psi 
%+  \left(\nabla \ep \psi - (w[\ep\zeta] \ep \psi)\nabla \ep \zeta\right)  \cdot \nabla ( w[\ep\zeta] \ep \psi)
%\end{align*} 
\begin{multline*}
\mathfrak{b}_1( U) 
 = 
 % 1 - 
 w [\zeta] \mcN_{1,0}^2(U) 
%+
-  \big(\nabla\psi - (w[\zeta] \psi)\nabla \zeta\big)  \cdot \nabla ( w[\zeta] \psi)
\\ % & 
% - d_{\zeta}  w (\mcN_{1, 0}^1( U) )\psi 
% -
 + \frac{ \mcG[\zeta] \big( ( w[\zeta] \psi)   \mcG[\zeta]\psi \big)  
+  ( \mcG[\zeta] \psi) \nabla\cdot  \big(\nabla\psi - (w[\zeta] \psi)\nabla \zeta\big) 
+ (\nabla\zeta \cdot \nabla  \mcG[\zeta] \psi ) w[\zeta] \psi   }{1 + |\nabla\zeta|^2},
\end{multline*} 
%with $(\mcN_{\ep,0}^1(U), \mcN_{\ep,0}^2(U))^T = \mcN_{\ep,0}(U)$ given in \eqref{wweTheorem}, 
%and where $d_\zeta w(h)\psi$ is the derivative of $w[\ep \cdot]\psi$  at $\zeta$ in the direction $h$, see \eqref{mfP}.
%
where $ \mcN_{1,0}^2(U) $ is the second component of $\mcN_{1,0}(U)$ given in \eqref{wweTheorem}.

Using the estimates
\begin{equation*}
| \mcG[\zeta] \psi|_{H^{s}} , \  | w[\zeta] \psi|_{H^{s}} \le C(h_{\min}^{-1},\mu_{\max}, |\zeta|_{H^{s\vee t_0+1}}) |\nabla \psi|_{H^s}, 
\qquad s\ge 0,
\end{equation*}
which follow from \eqref{mcGestimate} and Lemma \ref{ProductEstimates} (1) and (4),
we obtain 
\begin{align*}
|\mathfrak{b}_1( U) |_\infty 
& \le C |\mathfrak{b}_1( U) |_{H^{t_0}} 
%\\
%& \le 
%% C(h_{\min}^{-1},\mu_{\max}, |\zeta|_{H^{t_0+1}})  | \mcN_{1,0}^2(U) |_{H^{t_0+1}}
%% C(h_{\min}^{-1},\mu_{\max}, |\zeta|_{H^{t_0+1}}) 
%C(t_0+1) ( | \zeta |_{H^{t_0+1}} +   (1 + C(t_0+2) ) |\nabla\psi |_{H^{t_0+1}}^2 )
%\\&
%% + ( |\nabla\psi|_{H^{t_0}} +  C(t_0+1) |\nabla \psi|_{H^{t_0}}|\zeta|_{H^{t_0+1}}) C(t_0+2) |\nabla \psi|_{H^{t_0+1}}
%+ (1+  C(t_0+1) )  |\nabla\psi|_{H^{t_0}}  C(t_0+2) |\nabla \psi|_{H^{t_0+1}}
%\\&
%+ C(|\zeta|_{H^{t_0+1}} )( C(t_0+1) C(t_0+2) |\nabla \psi  |_{H^{t_0+1}}C(t_0+2) |\nabla \psi  |_{H^{t_0+1}} 
%\\&
%+  C(t_0+1) |\nabla \psi  |_{H^{t_0}} ( 1+ C(t_0+2) ) | \nabla\psi  |_{H^{t_0+1}} 
%\\&
%+ C(t_0+2) |\nabla \psi  |_{H^{t_0+1}} C(t_0+1) |\nabla \psi  |_{H^{t_0}} )
%\\
%& \le 
%% C(h_{\min}^{-1},\mu_{\max}, |\zeta|_{H^{t_0+1}}) 
%%C(t_0+1)  
%%| \zeta |_{H^{t_0+1}} +  C(t_0+2) |\nabla\psi |_{H^{t_0+1}}^2 
% C(h_{\min}^{-1},\mu_{\max}, |\zeta|_{H^{t_0+1}})   | \zeta |_{H^{t_0+1}} 
% +   C(h_{\min}^{-1},\mu_{\max}, |\zeta|_{H^{t_0+2}})  |\nabla\psi |_{H^{t_0+1}}^2 
%\\
%& 
\le 
 C(h_{\min}^{-1},\mu_{\max}, |\zeta|_{H^{t_0+2}})  ( | \zeta |_{H^{t_0+1}} +   |\nabla\psi |_{H^{t_0+1}}^2 ) ,
\end{align*}
and, hence,
\begin{equation*}%\label{Levycondition}
\mathfrak{a}_\ep (U^0)  
\ge 1 -  \ep C(h_{\min}^{-1},\mu_{\max}, |\zeta^0|_{H^{t_0+2}})  ( | \zeta^0 |_{H^{t_0+1}} +  \ep  |\nabla\psi^0 |_{H^{t_0+1}}^2 ) \ge a_0 >0.
\end{equation*}
Thus, also here, for given $U^0$, there exists an $\ep_0\le 1$, such that the second condition in \eqref{initialdataconditions}
holds true uniformly for $0\le \ep \le \ep_0$. 
However, the constant $a_0>0$ influences the existence time $T>0$ of the solution of \eqref{wweTheorem}.
%
%As proved by Lannes in \cite[Proposition 4.32]{Lannes}, condition \eqref{Levycondition} is always satisfied in the case of finite depth with flat bottom considered in the present article for $U^0\in E_{\ep,0}^N$ with \eqref{initialdataconditions}. 
%%
For more information on the definition, the r\^ole, and the properties of $\mathfrak{a}_\ep (U)$ we refer the reader to 
\cite[\S\S\ 4.2.3, 4.3.1, 4.3.5]{Lannes}.
\hfill$\square$

\

{\sc Remark \ref{SectionJustification}.3.}\ 
As mentioned above, Theorem \ref{JustificationTheorem} is an adaptation of  \cite[Th.\ 4.16, 4.18]{Lannes}. 
There, the main focus are applications to shallow water theory, which corresponds to $\nu \sim 1$
for the parameter $\nu = \tanh (2\pi \sqrt{\mu}) / (2\pi\sqrt{\mu})$
 in the general nondimensionalized form for the water waves problem set up in \cite{Lannes, ASL08}, viz.\
\begin{align*}%\label{wweTheoremNu}
\partial_t U + \tilde \mcN_\nu (U)  = 0,\qquad
\tilde \mcN_\nu (U) & = \begin{pmatrix}  - \frac1{\mu \nu} \mcG [ \ve \zeta ] \psi 
\\
 \zeta + \frac\ve{2\nu} |\nabla\psi|^2 - \frac\ve{2\mu \nu } \frac{( \mcG[ \varepsilon\zeta] \psi + \ve\mu \nabla\zeta\cdot \nabla \psi)^2}{1+\ve^2\mu |\nabla\zeta|^2}
 \end{pmatrix},
\end{align*} 
and hence Theorems 4.16, 4.18 in \cite{Lannes} are formulated for the case $\nu=1$.
However, these results, \emph{as well as their method of proof}, 
% as explained there (see \cite[Rem.\ 8.7]{Lannes}), Theorem \ref{JustificationNu=1} 
hold true also for \emph{deep water of finite depth}, $\nu\sim (2\pi\sqrt{\mu})^{-1}$ or
(after suitable renormalization of the equations, see \cite[Ch.\ 4, fn.\ 9]{Lannes}),  equivalently, $\nu\sim \mu^{-1/2}$, leading to 
\begin{align*}%\label{wweTheoremNu}
\tilde \mcN_{\frac1{\sqrt{\mu}}} (U) & = \begin{pmatrix}  - \frac1{\sqrt{\mu}} \mcG [ \frac{\ep}{\sqrt{\mu}} \zeta] \psi 
\\
 \zeta + \frac\ep{2} |\nabla\psi|^2 - \frac\ep{2} \frac{( \frac{1}{\sqrt{\mu}}\mcG[ \frac\ep{\sqrt{\mu}}\zeta] \psi + \ep \nabla\zeta\cdot \nabla \psi)^2}{1+\ep^2 |\nabla\zeta|^2}
 \end{pmatrix}
 = \mcN_{\ep,0}(U), 
 \qquad  \ep = \ve \sqrt{\mu},
\end{align*} 
provided  \emph{$\mu$ is bounded}, since then, $\nu=\tanh (2\pi \sqrt{\mu}) / (2\pi\sqrt{\mu})$ remains of order $\O(1)$ with respect to 
$\mu$ (see \cite[Remark 8.7]{Lannes}).
%Here, $\mcN_0(U)$ equals $\mcN_\sigma(U)$ in \eqref{wwe} in the case without surface tension, $\sigma=0$,
%and Theorem  \ref{JustificationNu=1} holds true verbatim for the water waves problem \eqref{wwe} when $\ve$ is replaced by $\ep$ with $0<\ep\le \ep_0$, $\ep_0$ being an additional argument  of the constants $c_j$. 
%\hfill$\square$
%
Nevertheless, an analogous result holds true if $\ep = \eps \sqrt{\mu}$ is bounded (even when $\mu\to\infty$), 
although the proof for this needs a different approach (see \cite[\S 4.4.3]{Lannes} and \cite{ASL08, ASL08b}).
However, since here we are interested in the case of finite depth $\sqrt{\mu}<\infty$, 
we use the well-posedness result as presented above,
and hence in our justification result we can not simply take the limit 
%  and hence the obtained results are not suitable for 
$\mu\to\infty$.

\

Since \eqref{wweTheorem} is equivalent to $\partial_t (\ep U) +  \mcN_{1,0} (\ep U)  = 0$ 
(note the index $\ep =1$ in $\mcN_{1,0}$) 
with $\mcE^N_1(\ep U) = \ep^2 \mcE^N_\ep(U)$ and $\mathfrak{a}_1 (\ep U) = \mathfrak{a}_\ep (U) $,
% IST DAS HILFREICH ODER EIN GEGENARGUMENT? ANDERS: STIMMT DAS JETZT SO, ODER NICHT?
we realize that the solution $U$ in the first part of the theorem corresponds to the unique solution $\ep U \in E^N_{1,T/\ep}$ of the initial value problem 
\begin{equation*}
\partial_t (\ep U) +  \mcN_{1,0} (\ep U)  = 0, \quad \ep U(0) = \ep U^0 \in E^N_{1,0}
\end{equation*}
with $T,c_2$ as in the theorem, i.e.\ 
%where $T$ depends on $U^0$ but is 
independent of $\ep\in (0,1]$. 
%
% IST DAS JETZT DAS GLEICHE MIT ODER DAS GEGETEILIGE VON COROLLARY 4.22 ? (ICH GLAUBE EHER DAS GEGENTEIL) 

Moreover, 
%setting $U_{app} = \ep U_a$ with $U_a$ as in \eqref{Uappfinal} we obtain 
from \eqref{UaConsistency} we obtain
\begin{equation*} % \label{UappConsistency}
% \partial_t U_a +\mcN_{\ep,\sigma} (U_a)  = \ep^3 ( r_2^1 , r_2^2 )^T, 
\partial_t (\ep U_a) +\mcN_{1,0} (\ep U_a)  = \ep^4 ( r_2^1 , r_2^2 )^T. 
\end{equation*}
Assuming that the estimates
\begin{equation}\label{macroest1}
1 - \ep \sup_{t\in [0,T_0/\ep]} |\zetaa(t)|_\infty > 0 \quad \text{$\forall$ $\ep\le\ep_0$ with some $\ep_0\le 1$}
\end{equation}
and
\begin{equation}\label{macroest23}
\sup_{t\in [0,T_0/\ep]} \mcE^N_1 (\ep U_a (t))\le 
%c_{app},
c_a,
\qquad
|(r_2^1,\mfP r_2^2)|_{L^\infty ([0,T_0/\ep]; H^N)} \le \ep^{-d/2} 
%c_r
c_a
\end{equation} 
are satisfied (with $T_0, 
%c_{app}, c_r
c_a
 >0 $ depending only on $U_a$),
% Postponing for the moment the proof of the three last estimates, 
we obtain from the stability part of the 
%Theorem \ref{JustificationTheorem} 
theorem, in its version for $\ep=1$,  
the estimate for the error $\mfe= \ep U - \ep U_a$
%\begin{multline*}
%\mcE^{N-1}(\ep ( U  - U_a)(t) )^{1/2} 
%\\
% \le C(T, c_1, c_2, C_{app},\ep_0) (\mcE^{N-1}(\ep ( U  - U_a)(0) )^{1/2} + t \ep^{4 - d/2} C_r )
%\quad \text{for $\ep t\le T_\ast = \min\{T,T_0\}$} 
%\end{multline*}
\begin{align}\label{errorest1}
\mcE^{N-1}_1 \big(\mfe(t) \big)^{1/2} 
 \le C( 
 %T,
%%% IST IN LANNES BOOK UND IM BEWEIS NICHT DRIN; NOCHMAL CHECKEN WARUM ICH DAS REINGEMACHT HABE
 % c_1, 
%%% UNNOETIG, DA DAS GLEICHE WIE c_2
 c_2, 
 % c_{app}
 c_a
 %,\ep_0
 %%% HAENGT NICHT DAVON AB; ZUR NOT SETZT MAN ES GLEICH 1
 ) \Big(\mcE^{N-1}_1\big(\mfe(0)\big)^{1/2} + t \ep^{4 - d/2} 
 % c_r 
 c_a 
 \Big)
%\quad \text{for $t\le T_\ast/\ep$,} 
%\quad \text{for $\ep t\le T_\ast = \min\{T,T_0\}$} 
\end{align}
for $\ep\le \ep_0$ and $t\le T_\ast/\ep$ with $T_\ast = \min\{T,T_0\}$ and the $T, 
%c_1, 
%%% UNNOETIG, DA DAS GLEICHE WIE c_2
c_2$ of Theorem \ref{JustificationTheorem}.  

\

We derive now sufficient conditions on the approximation $U_a$ given by \eqref{Uappfinal} and on the residuals $r_2^1$, $r_2^2$ given by \eqref{r21}, such that the estimates \eqref{macroest1}, \eqref{macroest23} are satisfied. 
From the form of $U_a$ and of the residuals it is clear that actually we need conditions on the macroscopic functions comprising them. 
As these functions are determined through classical partial differential equations, we prefer the conditions to be expressed 
rather in terms of  $|\cdot|_{H^s}$-norms than in terms of the energy norm $\mcE^N_1$ or $|\mfP\cdot|_{H^s}$. 

Denoting with $\tilde\zetaa$, $\tilde\psia$ the vectors of all macroscopic functions of $\zetaa$, $\psia$, respectively,  
% (as in \eqref{formu}) 
%and with $\xi$ the vector 
%%of wave vectors 
%$(\xi_1,\xi_2,\xi_3)$, 
we obtain from \eqref{energyestimate} and \eqref{ModulationEstimate} 
\begin{align} \label{mcEN1epUa}
\mcE^N_1(\ep U_a)  
& \le C \left( h_{\min}^{-1} ,  \mu_{\max} , | \ep \zetaa|_{H^{t_0 + 1}} ,|\ep \nabla \psia |_{H^{t_0}} \right) ( |\ep\zetaa |_{H^{N+1}}^2 + |\ep\nabla\psia|_{H^N}^2) 
\displaybreak[0] \\ \notag
& \le C \left( h_{\min}^{-1} ,  \mu_{\max} , | \tilde \zetaa|_{H^{t_0 + 1}} ,|\tilde \psia |_{H^{t_0+1}}, %\ep_0, 
|\xi_j| \right) ( |\tilde\zetaa |_{H^{N+1}}^2 + |\tilde\psia|_{H^{N+1}}^2)
%, \quad N\ge t_0+3/2.
\end{align}
for $\ep\le 1$, and, analogously, 
% KOENNTE MAN VIELLEICHT BESSER ERKLAEREN, ABER IST OK SO
% Analogously, we obtain 
for \eqref{r21}, recalling \eqref{estimateR21}, \eqref{estimateR22},
%from 
and with \eqref{estimatemfP},
%\eqref{ModulationEstimate} 
\begin{multline}\label{resepUa}
\ep^{ d/2}   
 |(r_2^1,\mfP r_2^2)|_{H^N} 
%\le |(r_2^1, \nabla r_2^2)|_{H^N} ,
%\le C( |(r_2^1|_{H^N} + |\nabla r_2^2)|_{H^N} ) 
%\le  &  \ep^{-d/2}  C(|\xi|, | \pl_t' \tilde \zeta_{2j} |_{H^N} , | \pl_t' \tilde \zeta_{2ji} |_{H^N} , | \pl_t'  \tilde \zeta_{2jik} |_{H^N}, | \pl_t' \tilde \zeta_{20} |_{H^N})
%\\ & +  \ep^{-d/2}  C(|\xi|,  | \pl_t' \tilde \psi_{2ji} |_{H^{N+1}},   | \pl_t' \tilde  \psi_{2jik} |_{H^{N+1}} )
%\\ & + \ep^{-d/2}  M(N+3,\tilde\zetaa) C \big( | \tilde \psi_0|_{H^{N+4}},  | \tilde \psi_1|_{H^{N+3}},  |\tilde \psi_2|_{H^{N+2}} \big) 
\le   
%\ep^{-d/2}  
C(h_{\min}^{-1} ,  \mu_{\max} , 
%\ep_0, 
|\xi_j|, |\tilde \zetaa|_{H^{N+3}},  | \tilde \psi_0|_{H^{N+4}},  | \tilde \psi_1|_{H^{N+3}},  
% \\
|\tilde \psi_2|_{H^{N+2}}, 
\\
% |\tilde \zetaa|_{H^{N+3}}, 
 | \pl_t' \tilde \zeta_{2j} |_{H^N} , | \pl_t' \tilde \zeta_{2ji} |_{H^N} , | \pl_t'  \tilde \zeta_{2jik} |_{H^N}, | \pl_t' \tilde \zeta_{20} |_{H^N}, 
| \pl_t' \tilde \psi_{2ji} |_{H^{N+1}},   | \pl_t' \tilde  \psi_{2jik} |_{H^{N+1}} ).
\end{multline} 
% FUER WELCHE N, FUER WELCHE t ? 

By a careful count of derivatives in the formulas for the macroscopic functions of $U_a$ 
%\eqref{Uappfinal} 
as obtained in Section \ref{SectionFormalDerivation}, we obtain from standard results of qualitative theory  (see, e.g., \cite[\S 2.1, \S 7.2]{Evans2})  for the linear homogeneous and inhomogeneous transport equations and the linear inhomogeneous wave equation of the system \eqref{mainsystem}  
that, for initial data of the form 
\begin{equation*}
\psi_{0j}^0 \in H^{s+4}(\R^d), 
%\qquad \psi_{00}^0=\pl_t' \psi_{00}^0=\psi_{1j}=0
\quad
%\psi_{00}^0 \in H^{s+3}(\R^d), 
%\ 
%\pl_t' \psi_{00}^0 \in H^{s+2}(\R^d),
%\quad 
( \psi_{00}^0, \pl_t' \psi_{00}^0 ) \in H^{s+3}\times H^{s+2} (\R^d), 
\quad 
\psi_{1j}^0 \in H^{s+2}(\R^d)
%\quad  s\in \N,\ s\ge t_0\vee 2,
\end{equation*}
% with $s\in \N$, $s\ge t_0\vee 2$,
with $s\in \N$, $s\ge 2 > t_0 = 3/2$,
we have for every $T_0>0$ and $t'\le T_0$ the estimates 
%\begin{align*}
%|\tilde U_2 |_{H^s} & \le C(|\xi|, |\psi_{0j}|_{H^{s+2}}, |\psi_{1j}|_{H^{s+1}}, |\pl_t' \psi_{00}|_{H^{s}}
%%,  |\pl_t' \psi_{1ji}|_{H^{s}}, |\pl_t' \zeta_{1ji}|_{H^{s}} 
%)
%\\
%|\tilde U_1 |_{H^s} & \le C(|\xi|, |\psi_{0j}|_{H^{s+1}}, |\psi_{1j}|_{H^{s}}, |\pl_t' \psi_{00}|_{H^{s}})
%\\
%|\tilde U_0 |_{H^s} & \le C(|\xi|, |\psi_{0j}|_{H^{s}}, |\psi_{00}|_{H^{s}})
%\\
%|\pl_t' \tilde U_{2ji,2jik} |_{H^s} & \le C(|\xi|, |\psi_{0j}|_{H^{s+2}}, |\pl_t' \psi_{1j}|_{H^{s}})
%\\
%|\pl_t' \tilde \zeta_{20} |_{H^s} & \le C(|\xi|,  |\psi_{0j}|_{H^{s+2}}, |\pl_t' \psi_{1j}|_{H^{s}} )
%\\
%|\pl_t' \tilde \zeta_{2j} |_{H^s} & \le C(|\xi|, |\psi_{0j}|_{H^{s+3}},  |\pl_t' \nabla' \psi_{1j}|_{H^{s}} , |\pl_t' \pl_t' \psi_{00}|_{H^{s}} )
%\\
%|\psi_{1j} |_{H^s} & \le C(|\xi|,   |E_j|_{H^s})
%\\
%|\pl_t' \tilde \psi_{1j} |_{H^s} & \le C(|\xi|, |\psi_{1j} |_{H^{s+1}},   |E_j|_{H^{s+1}} )
%\\
%|\pl_t' \nabla' \tilde \psi_{1j} |_{H^s} & \le C(|\xi|, |\psi_{1j} |_{H^{s+2}},   |E_j|_{H^{s+2}} )
%\\
%|E_j |_{H^s} & \le C(|\xi|,  |\psi_{0j}|_{H^{s+2}}, |\pl_t' \psi_{00}|_{H^{s}} 
%% , |\nabla' \psi_{00}|_{H^{s}} 
%, |\psi_{00}|_{H^{s+1}} 
%)
%\end{align*}
%\begin{align*}
%| (\tilde\zeta_0,\tilde\psi_0)  |_{H^s} & \le C( T_0, \mu_{\max}, |\xi|, |\psi_{0j}^0|_{H^{s}},  |( \psi_{00}^0, \pl_t' \psi_{00}^0 )|_{H^{s} \times H^{s-1}} ) ,
%\\
%|(\tilde\zeta_1,\tilde\psi_1)  |_{H^s} & \le C( T_0, \mu_{\max}, |\xi|, |\psi_{0j}^0|_{H^{s+2}},  |( \psi_{00}^0, \pl_t' \psi_{00}^0 )|_{H^{s+1} \times H^{s}}, |\psi_{1j}^0|_{H^{s}}),
%\end{align*}
\begin{align*}
| (\tilde\zeta_0,\tilde\psi_0)  |_{H^s} & \le C( T_0, \mu_{\max}, |\xi_j|, |\psi_{0j}^0|_{H^{s}},  |( \psi_{00}^0, \pl_t' \psi_{00}^0 )|_{H^{s} \times H^{s-1}} ) ,
\end{align*}
\begin{align*}
|(\tilde\zeta_1,\tilde\psi_1)  |_{H^s} & \le C( T_0, \mu_{\max}, |\xi_j|, |\psi_{0j}^0|_{H^{s+2}},  |( \psi_{00}^0, \pl_t' \psi_{00}^0 )|_{H^{s+1} \times H^{s}}, |\psi_{1j}^0|_{H^{s}}),
\end{align*}
\begin{multline*}
|(\tilde\zeta_2,\tilde\psi_2) |_{H^s} , 
| \pl_t' \tilde \zeta_{2ji} |_{H^s} , | \pl_t'  \tilde \zeta_{2jik} |_{H^s}, | \pl_t' \tilde \zeta_{20} |_{H^s}, 
| \pl_t' \tilde \psi_{2ji} |_{H^{s}},   | \pl_t' \tilde  \psi_{2jik} |_{H^{s}} 
\\
\le C( T_0, \mu_{\max},  |\xi_j|, |\psi_{0j}^0|_{H^{s+3}},  |( \psi_{00}^0, \pl_t' \psi_{00}^0 )|_{H^{s+2} \times H^{s+1}} , |\psi_{1j}^0|_{H^{s+1}} ), 
\end{multline*}
\begin{align*}
|\pl_t' \tilde \zeta_{2j} |_{H^s} & \le C(T_0, \mu_{\max}, |\xi_j|,
|\psi_{0j}^0|_{H^{s+4}},  |( \psi_{00}^0 , \pl_t' \psi_{00}^0) |_{ H^{s+3} \times H^{s+2}},  |\psi_{1j}^0|_{H^{s+2}} ).
\end{align*}

Hence, we obtain from \eqref{mcEN1epUa} and \eqref{resepUa} that,  for initial data
\begin{equation}\label{macroinitialdata}
\psi_{0j}^0 \in H^{N+6}(\R^d),
\quad
( \psi_{00}^0, \pl_t' \psi_{00}^0 ) \in H^{N+5}\times H^{N+4} (\R^d), 
\quad 
\psi_{1j}^0 \in H^{N+4}(\R^d)
\end{equation}
with $N\in\N$ as in Theorem \ref{JustificationTheorem},
% with $s\in \N$, $s\ge t_0\vee 2$,
the approximation $U_a$ of \eqref{Uappfinal} 
%with the macroscopic functions determined as in Section \ref{SectionFormalDerivation}, and in particular by \eqref{mainsystem},
satisfies for every $T_0>0$ the assumptions \eqref{macroest23}
with 
%$c_{app} = c_r = c_a$, where 
\begin{equation*}%\label{ca}
c_a = C ( T_0,   h_{\min}^{-1} , \mu_{\max}, |\xi_j|, 
%\ep_0,  
%%% WE SET THIS =1
|\psi_{0j}^0|_{H^{N+6}},  |( \psi_{00}^0 , \pl_t' \psi_{00}^0) |_{ H^{N+5} \times H^{N+4}},  |\psi_{1j}^0|_{H^{N+4}} ),
\end{equation*}
% Moreover, since $|\zetaa(t)|_\infty  \le  C |\tilde\zetaa(t')|_\infty \le C |\tilde\zetaa(t')|_{H^{t_0}} \le c_a$,
% we can choose for every $T_0>0$ some $0<\ep_0\le 1$ such that \eqref{macroest1} is satisfied for all $\ep\le \ep_0$.
%
where, since $|\zetaa(t)|_\infty  \le  C |\tilde\zetaa(t')|_\infty \le C |\tilde\zetaa(t')|_{H^{t_0}}$,
we can find for every $T_0>0$ some $0<\ep_0\le 1$ such that \eqref{macroest1} is satisfied uniformly for  $\ep\le \ep_0$, with $h_{\min}>0$ as a lower bound. 
Thus, for the initial data \eqref{macroinitialdata} and sufficiently small $\ep_0$, the error estimate \eqref{errorest1} holds indeed true. 

If one wants to reformulate the  estimate \eqref{errorest1} in terms of Sobolev norms, one needs to estimate its right-hand side from above 
and its left-hand side from below by such norms. 
To this end, 
% Furthermore, 
since
\begin{equation}\label{error}
%\mfe  = \ep U - \ep U_a = \ep U - \ep U_{a,1} - \ep^3 U_2  = \mfe_1 - \ep^3 U_2, \qquad U_2 = (\zeta_2,\psi_2)^T, 
\mfe  = \ep U - \ep U_a = \ep U - \ep U_{a,1} - \ep^3 (\zeta_2,\psi_2)^T,  
\end{equation}
we obtain from \eqref{energyestimate} for $\ep=1$ and $U^0 = (\zeta^0,\psi^0)^T \in H^{N+1}\times \dot H^{N+1}(\R^d)$
\begin{equation*} 
% \mcE^N_\ep (U)  \le C \left( h_{\min}^{-1} ,  \mu_{\max} , | \ep ( \zeta, \psi) |_{H^{t_0 + 1} \times \dot H^{t_0+1}} \right)
%  |(\zeta,\psi) |_{H^{N+1}\times \dot H^{N+1}}^2 
% \mcE^{N-1}_1 (\mfe(0))  \le C \left( h_{\min}^{-1} ,  \mu_{\max} , | \mfe(0) |_{H^{t_0 + 1} \times \dot H^{t_0+1}} \right) 
% |\mfe(0)|_{H^{N}\times \dot H^{N}}^2 
% \mcE^{N-1}_1 (\mfe(0))  \le C \left( c_2, c_a \right)  ( |\mfe_1(0)|_{H^{N}\times \dot H^{N}}^2  + |\ep^3 U_2|_{H^{N}\times \dot H^{N}}^2  )
% \mcE^{N-1}_1 (\mfe(0))  \le C \left( c_2, c_a \right)  ( | \ep U^0 - \ep U_{a,1}^0|_{H^{N}\times \dot H^{N}}^2  + |\ep^3 (\zeta_2,\psi_2)|_{H^{N}\times \dot H^{N}}^2  )
%\mcE^{N-1}_1 (\mfe(0))^{1/2}  \le C \left( c_2, c_a \right)  
% ( | \ep U^0 - \ep U_{a,1}^0|_{H^{N}\times \dot H^{N}} + \ep^{3-d/2}|(\tilde \zeta_2, \tilde \psi_2)|_{H^{N}\times \dot H^{N}})
\mcE^{N-1}_1 (\mfe(0))^{1/2}  \le C ( c_2, c_a )  \Big( | \ep U^0 - \ep U_{a,1}(0,\cdot)|_{H^{N}\times \dot H^{N}} + \ep^{3-d/2} c_a \Big)
\end{equation*}
with $c_2$ as in Theorem \ref{JustificationTheorem} and 
% the $c_a$ of \eqref{ca}.
$c_a$ as above.
%
% Finally, 
Moreover, 
since we get from 
%the definition 
\eqref{mfP} 
%of $\mfP$ 
%we get
\begin{equation} \label{estimatemfP1}
|\nabla u|_{H^{s}} \le (1+\sqrt{\mu})^{1/2}|\mfP u|_{H^{s+1/2}} , \quad  s\in \R,\quad \mu\ge 0,
\end{equation}
and
\begin{equation}\label{Pmu_above} 
|\mfP u|_{H^s} \le \max\{1,\mu^{-1/4}\} |\nabla u|_{H^{s - 1/2}} \le   |u|_{H^{s+1/2}} , \quad  s\in \R,\quad \mu\ge 1,
\end{equation}
we obtain from \eqref{estimatemfP}, \eqref{energyN} and Lemma \ref{ProductEstimates}(1)
\begin{align*} 
|\nabla\psi|_{H^{N-1}} 
& \le C(\mu_{\max}) |\mfP\psi|_{H^{N-1/2}}
%\\&
%\le C(\mu_{\max}) \sum_{|\alpha| \le N-1} 
% \Big( | \mfP \psi_{\ep, (\alpha)} |_{H^{1/2}} + \big|\mfP \big( \ep (w[\ep\zeta]\psi) \pl^\alpha\zeta\big)\big|_{H^{1/2}} \Big)
\displaybreak[0] \\&
 \le C(\mu_{\max}) \sum_{|\alpha| \le N-1} 
 \Big( | \mfP \psi_{\ep, (\alpha)} |_{H^1} + \big|\mfP \big( (w[\ep\zeta]\ep\psi) \pl^\alpha\zeta\big)\big|_{H^{1/2}} \Big)
\displaybreak[0] \\&
 \le C(\mu_{\max}) \Big( \sum_{|\alpha| \le N} | \mfP \psi_{\ep, (\alpha)} |_2
 + \sum_{|\alpha| \le N-1} \sum_{|\beta| =1} \big|\mfP \big( \big(\pl^\beta ( w[\ep\zeta] \ep \psi  ) \big) \pl^\alpha\zeta\big)\big|_2
 \\& \phantom{\le C(\mu_{\max}) \Big( \ }
 + \sum_{|\alpha| \le N-1}  \big|  (w[\ep\zeta] \ep \psi) \pl^\alpha\zeta \big|_{H^1} \Big),
\displaybreak[0] \\&
 \le C(\mu_{\max}) \Big( \mcE_\ep^N (\zeta,\psi)^{1/2}
 + \sum_{|\alpha| \le N-1} \sum_{|\beta| =1} \big|\big( \pl^\beta ( w[\ep\zeta]\ep\psi ) \big) \pl^\alpha\zeta \big|_{H^1}
 \\& \phantom{\le C(\mu_{\max}) \Big( \ }
 + \sum_{|\alpha| \le N-1} \big| w[\ep\zeta] \ep \psi\big|_{H^{t_0}} |\pl^\alpha\zeta|_{H^1} \Big),
%\\&
% \le C(\mu_{\max}) \Big( \sum_{|\beta| =1} \big| \ep \pl^\beta ( w[\ep\zeta]\psi ) \big|_{H^{t_0}}
%  + \big| \ep w[\ep\zeta]\psi\big|_{H^{t_0}} \Big) \mcE_\ep^N (\zeta,\psi)^{1/2},
\displaybreak[0] \\&
 \le C(\mu_{\max})  \big| w[\ep\zeta] (\ep \psi)  \big|_{H^{t_0+1}}  \mcE_\ep^N (\zeta,\psi)^{1/2},
\end{align*}
and, hence, from Lemma \ref{ProductEstimates}~(4)~and~(3), and \eqref{mcGestimate}, \eqref{estimatemfP1}
\begin{align*}
|\nabla\psi|_{H^{N-1}} 
 \le C( h_{\min}^{-1}, \mu_{\max}, | \ep \zeta|_{H^{t_0 + 2}}, |\mfP ( \ep \psi) |_{H^{t_0+3/2}}  ) \mcE_\ep^N (\zeta,\psi)^{1/2}.
\end{align*}
Since, obviously, also $|\zeta|_{H^N} \le \mcE_\ep^N(\zeta,\psi)^{1/2}$, we obtain for $U=(\zeta,\psi)^T$
\begin{align}\label{energy_above}  
|U|_{H^N\times\dot H^N} 
\le C( h_{\min}^{-1}, \mu_{\max}, | \ep \zeta|_{H^{t_0 + 2}}, |\mfP ( \ep \psi) |_{H^{t_0+3/2}}  ) \mcE_\ep^N (U)^{1/2},
\end{align}
and, in particular, for $\ep=1$ and the error $\mfe$ of \eqref{error}
\begin{align*} 
|\mfe(t)|_{H^{N-1}\times\dot H^{N-1}} 
% \le C( h_{\min}^{-1}, \mu_{\max}, | \zeta|_{H^{t_0 + 2}}, |\mfP \psi |_{H^{t_0+3/2}}  ) \mcE_1^{N-1} (\mfe(t))^{1/2}.
\le C( c_2,c_a) \mcE_1^{N-1} (\mfe(t))^{1/2}
\end{align*}
with $N$, $c_2$ as in Theorem \ref{JustificationTheorem}, 
%and $c_a$ as in \eqref{ca},
and, hence, by the 
triangle inequality,
\begin{align*} 
|\ep U(t) - \ep U_{a,1}(t,\cdot)|_{H^{N-1}\times\dot H^{N-1}} 
\le C( c_2,c_a) \mcE_1^{N-1} (\mfe(t))^{1/2} + \ep^{3-d/2} c_a.
\end{align*}

\

Summarizing the above analysis, we obtain as the main result of this article the following justification theorem.
\begin{theorem} \label{JustificationNonresonantInteraction}
Under Notation \ref{approximationnotation} (with $\sigma=\bond = 0$) and its assumptions, 
let $N \ge t_0 + t_0\vee 2 + 3/2$ with $t_0=3/2 > d/2$,   
%$1\le \mu\le \mu_{\max}$,  
and let
\begin{equation*}%\label{macroinitialdata}
\psi_{0j}^0 \in H^{N+6}(\R^d),
\quad
( \psi_{00}^0, \pl_t' \psi_{00}^0 ) \in H^{N+5}\times H^{N+4} (\R^d), 
\quad 
\psi_{1j}^0 \in H^{N+4}(\R^d), 
\end{equation*}
%with 
% $j\in J = \{1,2,3\}$ 
$j=1,2,3$,
be the initial data for the system 
\begin{align*} %\label{mainsystemfinal}
\begin{cases}
\pl_t' \psi_{0j}   +  \nabla\om_j \cdot   \nabla' \psi_{0j} = 0,
\\ \ds
{\pl_t'}^2  \psi_{00} -  \sqrt{\mu} \Delta' \psi_{00}   =  
% \sum_{j\in J}  
%\sum_j
\sum_{j=1}^3
\left(  (
%g_j^2 
\om_j^4
- |\xi_j|^2) \pl_t' + 2  
%\frac{\om_j}{b_j} 
\om_j 
\xi_j  \cdot  \nabla'  \right)  
|\psi_{0j}|^2,
\\
\pl_t'   \psi_{1j} + \nabla\om_j \cdot \nabla' \psi_{1j}   = E_j
\end{cases}
\end{align*}
with 
$1 \le  \mu \le \mu_{\max}<\infty$
and
\begin{align*} % \label{tildeEj}
 E_j 
= &\  \ii \frac12   \nabla' \cdot \mcH_\om(\bk_j) \nabla' \psi_{0j} 
-  \ii  \psi_{0j} \Big( \frac{1}{2\om_j}  (\om_j^4 - |\xi_j|^2)   \pl_t' +  \bk_j  \cdot  \nabla'\Big) \psi_{00}
+ \tilde E_j,
\end{align*}
where $\tilde E_j$ consists of cubic products of $\psi_{0j}$, see \eqref{truetildeEj}.
Let also 
%\begin{align*} %\label{Uappfinal}
%U_{a,1} = 
%& \sum_j \begin{pmatrix}
%%\ii
%%%\frac{\om_j}{b_j} 
%%\om_j
%%\psi_{0j} +  \ep 
%%% \zeta_{1j} 
%%% \zeta_{1j} = \ii \frac{ \om_j }{b_j} \psi_{1j} + \nablabo\om_j  \cdot \nabla' \psi_{0j}
%%% \zeta_{1j} = 
%%( \ii \om_j \psi_{1j} + 
%%%\nablabo\om_j  
%%\nabla\om_j
%%% , \qquad \nablabo\om_j =\frac1{b_j}\Big( \nabla\om_j  - \bond 2\frac{  \om_j}{b_j}  \bk_j \Big).
%%\cdot \nabla' \psi_{0j} 
%%)
%%
%%\ii \om_j \psi_{0j} +  \ep ( \ii \om_j \psi_{1j} + \nabla\om_j\cdot \nabla' \psi_{0j} )
%%
%\ii \om_j ( \psi_{0j} +  \ep \psi_{1j} ) + \ep \nabla\om_j\cdot \nabla' \psi_{0j} 
%%
%\\  \psi_{0j} + \ep \psi_{1j} \end{pmatrix} \ee_j 
%+ \sum_{ji}  \begin{pmatrix} \ep \zeta_{1ji} \\ \ep \psi_{1ji} \end{pmatrix}  \ee_{ji}
%+ \cc 
%+ \begin{pmatrix}  \ep \zeta_{10} \\ \psi_{00} \end {pmatrix}
%\end{align*}
\begin{align*} %\label{Uappfinal}
U_{a,1} (t,X) = 
& % \sum_{j\in J} 
\sum_{j=1}^3 \begin{pmatrix}
\ii \om_j ( \psi_{0j} +  \ep \psi_{1j} ) + \ep \nabla\om_j\cdot \nabla' \psi_{0j} 
\\  \psi_{0j} + \ep \psi_{1j} \end{pmatrix} 
%(\ep t, \ep X) 
(t',X')
%\ee_j 
\ee^{\ii (\xi_j\cdot X -\om_j t)}
\\& 
+ 
%\sum_{ji}  
%\sum_{(j,i) \in I}\begin{pmatrix} \ep \zeta_{1ji} \\ \ep \psi_{1ji} \end{pmatrix} 
\ep \sum_{(j,i) \in I}\begin{pmatrix} \zeta_{1ji} \\ \psi_{1ji} \end{pmatrix} 
%(\ep t, \ep X)
(t',X')
% \ee_{ji}
\ee^{\ii \left((\xi_j+\xi_i)\cdot X - (\om_j +\om_i) t\right)}
+ \cc 
+ \begin{pmatrix}  \ep \zeta_{10} \\ \psi_{00} \end {pmatrix} 
(t',X')
%(\ep t, \ep X)
\end{align*}
with $0 \le t'  = \ep t\le T_0$,  $X' =\ep X \in \R^d$, $0< \ep\le 1$,  
where $\zeta_{1ji}, \psi_{1ji}$ consist of quadratic products of $\psi_{0j}$, see  \eqref{zetapsi1ji}, and 
\begin{align*}
%\label{zeta10}
\zeta_{10} = - \pl_t' \psi_{00} +  \sum_j  ( \om_j^4 - |\xi_j|^2 ) |\psi_{0j}|^2.
\end{align*}

Then, for any $c_0>0$ there exists an $\ep_0\in (0,1]$ and a $T>0$ such that for all 
$\ep \in (0,\ep_0]$ and all 
$U^0 = (\zeta^0,\psi^0)^T \in H^{N+1}\times \dot H^{N+1}(\R^d)$ with
\begin{equation*}
1 - \ep |\zeta^0|_{\infty} \ge h_{\min}>0,  
\quad \mathfrak{a}_\ep (U^0)  \ge a_0 >0,
\quad
%\text{and}\quad
 | \ep U^0 - 
 %\ep U_{a,1}^0
 \ep U_{a,1} (0,\cdot)
 |_{H^{N}\times \dot H^{N}} \le c_0 \ep^{3-d/2}
% \quad \forall\ \ep\le \ep_0
% DA KANN MAN DAS \ep NICHT WEGKANZELN, DA U_a|_{t=0} NICHT UNBEDINGT EINE \mcE^{N-1} NORM HAT
\end{equation*} 
there exists a unique solution $U = (\zeta,\psi)^T \in E^N_{1,T/\ep}$ 
to the 
%gravity 
water-waves problem 
\begin{equation*}
% \pl_t ( \ep U ) +\mcN_{1,0}(\ep U)=0,\quad  \ep U|_{t=0} = \ep U^0,
% \pl_t U  +\mcN_{1,0} ( U ) =0,\quad  U|_{t=0} = \ep U^0,
\begin{cases}
\ds \partial_t \zeta - \mcG[\zeta]\psi = 0,  
\\
\ds \partial_t \psi + \zeta 
+ \frac12 |\nabla\psi|^2 - \frac{( \mcG[\zeta]\psi + \nabla\zeta\cdot \nabla \psi)^2}{ 2 (1 + |\nabla\zeta|^2)} = 0, 
\end{cases}
\qquad  U(0) = \ep U^0,
\end{equation*} 
which satisfies for all $t\le T_\ast/\ep$ with $T_\ast =  \min\{T,T_0\}$ the estimate
%\begin{align*}
%\mcE^{N-1}\big( \widetilde U (t) - \ep U_a (t) \big)^{1/2} 
% \le C(T, c_1, c_2, C_{app},\ep_0) (C_0 + T_\ast C_r )\ep^{3 - d/2}  
%%\quad \text{for $t\le T_\ast/\ep$, $T_\ast =  \min\{T,T_0\}$ } 
%%\quad \text{for $\ep t\le T_\ast = \min\{T,T_0\}$} 
%\end{align*}
%with $T,c_1,c_2$ as in the theorem. 
\begin{align*} 
%| % \ep U(t) 
%U(t)
%- \ep U_{a,1}(t)|_{H^{N-1}\times\dot H^{N-1}} 
| % \ep U(t) 
U(t)
- \ep U_{a,1}(t,\cdot)|_{H^{N-1}\times\dot H^{N-1}} 
% \le C( c_2,c_a) \mcE_1^{N-1} (\mfe(t))^{1/2} + \ep^{3-d/2} c_a.
%\le C( c_2,c_a) \Big(\mcE^{N-1}_1\big(\mfe(0)\big)^{1/2} + t \ep^{4 - d/2}  c_a \Big) + \ep^{3-d/2} c_a.
%\le C( c_2,c_a) \Big( C( c_2,c_a)  \big( | \ep U^0 - \ep U_{a,1}^0|_{H^{N}\times \dot H^{N}} + \ep^{3-d/2} c_a \big) + t \ep^{4 - d/2}  c_a \Big) + \ep^{3-d/2} c_a.
% \le C( c_2,c_a) \Big(  | \ep U^0 - \ep U_{a,1}^0|_{H^{N}\times \dot H^{N}} + \ep^{3-d/2} c_a + t \ep^{4 - d/2} + \ep^{3-d/2}  \Big) 
\le C( c, c_a,c_0
%,T_\ast
) \ep^{3 - d/2}, 
\end{align*}
where
\begin{align*}
c & = C(|U^0|_{H^{N+1}\times\dot H^{N+1}}, h_{\min}^{-1}, a_0^{-1} , \mu_{\max}),
\\
%\label{ca}
c_a &= C ( T_0,   h_{\min}^{-1} , \mu_{\max}, |\xi_j|, 
%\ep_0,  
%%% WE SET THIS =1
|\psi_{0j}^0|_{H^{N+6}},  |( \psi_{00}^0 , \pl_t' \psi_{00}^0) |_{ H^{N+5} \times H^{N+4}},  |\psi_{1j}^0|_{H^{N+4}} ).
\end{align*}
%where $\ds 0< a_0 = \inf_{X\in\R^d} \mathfrak{a}_1(\ep U^0) = \inf_{X\in\R^d} \mathfrak{a}_\ep (U^0)$ 
%% is given in 
%is obtained as in Remark \ref{SectionJustification}.2. 
\end{theorem}

\

%%% KOMMENTARE UEBER DAS RESULTAT, WIE SIE AM ENDE DER INTRO UND WEITER OBEN STEHEN

We conclude this article, with some comments on our justification result. 

{\sc Remark \ref{SectionJustification}.4.}\ 
\begin{enumerate}
\item 
For $d=1,2$, the theorem holds true with $N=5$.
\item
An analogous justification result can be obtained for the transport equations \eqref{nonrestransport} by the leading-order-approximation 
$\ep U_{a,0}$, see Remark \ref{SectionFormalDerivation}.2., with an error of order $O(\ep^{2-d/2})$.
\item
The justification relies on the stability of the (original) water-waves equation, see Theorem \ref{JustificationTheorem}, and not on the stability of the derived system. 
In the case of finite depth, 
such a stability result for the water-waves equation does not exist for time-scales of higher order, e.g.,  $O(1/\ep^2)$.

Moreover, the well-posedness of the macroscopic linear transport equations and the macroscopic wave equation up to any time $T_0$ imply that the only restriction on the time of validity of the justification result is due to the existence time $T$ for the water-waves problem. 
However, $T_0$ influences the constant of the error estimate.
\item
There is a difference of one order between the Sobolev space in which the initial data of the approximation are assumed to exist 
and the order of the norm in which the initial distance to the original solution is measured, 
which is one order higher than the norm of the error for $t>0$. 
The latter difference results from the different estimates of the energy norm $\mcE_\ep^N (U)$ from above and from below, see \eqref{energyestimate} and \eqref{energy_above}, while the former one results from the stability result itself, see Theorem \ref{JustificationTheorem}, and for a more detailed analysis, \cite{Lannes}. 
However, an optimization of the regularity assumptions was not our main focus in this article. 
Recall, here, that $|U|_{H^N\times\dot H^N}^2  = |\zeta|_{H^N}^2  + |\nabla \psi|_{H^{N-1}}^2  $.
% \item

\end{enumerate}

%%%%%%%%%%%%%%%%%%%%%%%%%%%%%%%%%%%%%%%%%%%%%%%%%%%%%%%%
% \newpage

\section{Appendix}
We give here the definitions of the 
%maps $F_1$, $F_2$ and the 
functions $C$, $D$ appearing in Proposition \ref{finalexpansion}
with the abbreviations of Notation \ref{approximationnotation} and Proposition \ref{LannesResiduals}\eqref{propG0expansion}.
% We use the abbreviations of Notations \ref{approximationnotation} and \ref{notationprime}.
%
\begin{align*}%\label{F}
% F_2(\zeta_{00}) 
D_{00}
= & \bond  \Delta' \zeta_{00}  -  \big( G_0(\zeta_{00} G_0 \psi_0)  +  \zeta_{00}  \psi_0''   \big) G_0\psi_0,
\displaybreak[0]\\
% F_1(\zeta_{00})
C_{00}
= & - G_1 (\zeta_{00} G_0 \psi_0) - G_0(\zeta_{00} G_1 \psi_0)   - G_0 (\zeta_{00} G_0 \psi_1) 
 \displaybreak[0] \\& \notag
 %\\& 
 +  G_0 ( \zeta_{00}  G_0 ( 
 % \hat \zeta_0 
 ( \zeta_0 -\zeta_{00})
 G_0 \psi_0) )   
 +  G_0 ( \zeta_0  G_0 ( \zeta_{00}  G_0 \psi_0) )  
\displaybreak[0] \\&
-  \nabla' \zeta_{00} \cdot  \psi_0' 
 -  2 \zeta_{00} \nabla' \cdot \psi_0'
 -  \zeta_{00}  \psi_1''  
\displaybreak[0] \\&
+  {\ts\frac12} (  \zeta_{00}(
%\hat \zeta_0 + \zeta_0
2\zeta_0 - \zeta_{00}
)  G_0 \psi_0)''   +  {\ts\frac12} G_0 ( \zeta_{00}( 
%\hat \zeta_0 + \zeta_0
2\zeta_0 -\zeta_{00}
) \psi_0''), 
\displaybreak[0]  \\[.5em]
% D_0  = &   \sum_j \big(   ( \nabla' \psi_{0j} + \ii \xi_j \psi_{1j} ) \cdot    \ii \xi_j \overline{ \psi_{0j} }  +  ( -\ii\nabla g_j \cdot \nabla' \psi_{0j} +  g_j \psi_{1j}) g_j \overline{ \psi_{0j} }  \big)  + \cc ,
D_0  = &  {\ts \sum_j } \big(     \ii   \overline{ \psi_{0j} } ( \xi_j - g_j g'_j ) \cdot \nabla' \psi_{0j}  + ( - |\xi_j|^2 + g_j^2 )  \overline{ \psi_{0j} }  \psi_{1j} \big)  + \cc + D_{00} ,
\displaybreak[0]\ \\
C_0 = &   {\ts \sum_j }    \ii\xi_j  \cdot  \nabla'  \big(   \zeta_{0j} \overline{ \psi_{0j} }  \big)  +\cc + C_{00},
\displaybreak[0] \\[.5em]
%\end{align*}
%
% $\hat \zeta_0 = \zeta_0 -\zeta_{00}$, $\hat \zeta_1 = \zeta_1 -\zeta_{10}$, $\hat \psi_0 = \psi_0 -\psi_{00}$
%
%Using the formula 
%\begin{align*}
%\Big(\sum_j a_j \ee_j + \cc \Big)  \Big(\sum_i b_i \ee_i + \cc \Big) = \sum_{ji} c_{ji} \ee_{ji} + \cc + c_0
%\end{align*}
%with
%\begin{align*}
%c_{jj} = a_j b_j, \quad c_{ji} = a_j b_i + a_i b_j ,\quad c_{j,-i} = a_j \bar b_i + \bar a_i b_j, \quad c_0 = \sum_j a_j \bar b_j  +\cc,
%\end{align*}
%we obtain
% \begin{align*}
D_{jj}   = &  -    \ii \psi_{0j} ( \xi_j   +   g_j  g'_j ) \cdot \nabla'  \psi_{0j} + ( |\xi_j|^2 + g_j^2  )  \psi_{0j} \psi_{1j} , 
\displaybreak[0]\\\
C_{jj} =  & \ \ii ( g_j g'_{jj} - \xi_j ) \cdot \nabla' (\zeta_{0j} \psi_{0j})
+ \ii  \zeta_{0j}  ( g_{jj}  g'_j  - 2 \xi_j) \cdot \nabla' \psi_{0j}
\displaybreak[0] \\&
- ( g_j g_{jj}  - 2 |\xi_j|^2 ) (\psi_{0j} \zeta_{1j}  +  \zeta_{0j}  \psi_{1j} )  
\qquad \text{for $j\in J % = \{1,2,3\}
$,}
% \intertext{for $j\in J = \{1,2,3\}$,}
 \displaybreak[0]\\[.5em]
D_{ji}   = &  -   \ii  \psi_{0i} (  \xi_i + g_i g'_j ) \cdot \nabla' \psi_{0j}  -  \ii  \psi_{0j} ( \xi_j  + g_j g'_i ) \cdot \nabla' \psi_{0i}  
\displaybreak[0] \\& + ( \xi_j\cdot  \xi_i  +  g_j g_i ) ( \psi_{0i}  \psi_{1j}   +  \psi_{0j}  \psi_{1i} )  , 
\displaybreak[0]\ \\
C_{ji} = &
\  \ii  ( g_i g'_{ji} -  \xi_i ) \cdot \nabla' ( \zeta_{0j}\psi_{0i} ) 
+ \ii ( g_j g'_{ji} -   \xi_j) \cdot  \nabla'  ( \zeta_{0i} \psi_{0j} ) 
\displaybreak[0]\\&
+ \ii \zeta_{0j} ( g_{ji} g'_i  - \xi_{ji} ) \cdot \nabla' \psi_{0i}   
+  \ii \zeta_{0i} ( g_{ji} g'_j -  \xi_{ji}) \cdot \nabla' \psi_{0j} 
\displaybreak[0]\\& 
- ( g_i  g_{ji}   -  \xi_i \cdot  \xi_{ji} ) ( \psi_{0i}  \zeta_{1j} + \zeta_{0j}   \psi_{1i} )   
\displaybreak[0]\\& 
- (g_j  g_{ji}  -  \xi_j \cdot \xi_{ji} ) ( \psi_{0j}  \zeta_{1i} +  \zeta_{0i}  \psi_{1j}  )
\qquad\text{for $(j,i)\in I_< % = \{(1,2), (1,3), (2,3) \}
$, }
\displaybreak[0]\\[.5em]
D_{j,-i} = & \  \ii  \overline{  \psi_{0i} } (  \xi_i - g_i g'_j ) \cdot \nabla' \psi_{0j}  -   \ii \psi_{0j} (  \xi_j  - g_j g'_i ) \cdot \nabla'  \overline{\psi_{0i} } 
\displaybreak[0]\\& +  (  -  \xi_j  \cdot \xi_i +   g_j g_i  ) ( \overline{ \psi_{0i} }  \psi_{1j} + \psi_{0j} \overline{ \psi_{1i} } ),
\displaybreak[0]\ \\
C_{j,-i} = & 
 \ \ii ( g_i g'_{j,-i} + \xi_i ) \cdot  \nabla' ( \zeta_{0j} \overline{ \psi_{0i}} ) 
 +\ii ( g_j g'_{j,-i} - \xi_j ) \cdot   \nabla'  ( \overline{  \zeta_{0i} } \psi_{0j} )
\displaybreak[0]\\ &
- \ii \zeta_{0j}  (g_{j,-i}  g'_i +  \xi_{j,-i} ) \cdot \nabla'  \overline{ \psi_{0i}}
+ \ii    \overline{ \zeta_{0i} } ( g_{j,-i}  g'_j -  \xi_{j,-i} ) \cdot \nabla' \psi_{0j} 
\displaybreak[0]\\ &
- ( g_i g_{j,-i}   + \xi_i \cdot  \xi_{j,-i} ) ( \overline{ \psi_{0i} } \zeta_{1j} +  \zeta_{0j} \overline{  \psi_{1i} } ) 
\displaybreak[0]\\ &
- ( g_j g_{j,-i}  - \xi_j \cdot \xi_{j,-i} ) ( \psi_{0j}  \overline{ \zeta_{1i} } +  \overline{ \zeta_{0i} } \psi_{1j} ) 
\qquad\text{for $(j,i)\in I_< % = \{(1,2), (1,3), (2,3) \}
$, }
%\intertext{for $(j,i)\in I_< = \{(1,2), (1,3), (2,3) \}$, }
\displaybreak[0] \\[.5em]
D_j  = & - d_j^{(1)} + d_j^{(2)}  + g_j \big( {\ts \sum_i }  |\xi_i|^2 \zeta_{0i}  \overline{ \psi_{0i}} +\cc \big)  \psi_{0j}  
\displaybreak[0] \\ & 
+ \bond \Big(  \ii\xi_j \cdot d_j^{(3)}  +  |\xi_j|^2 \big(  {\ts  \sum_i } |\xi_i|^2  |\zeta_{0i}|^2 \big)   \zeta_{0j} \Big) ,
 \displaybreak[0]  \\
 C_j  = &  - g_j \big( d_j^{(4)} + d_j^{(5)}   \big)   - \big( d_j^{(6)}  +  d_j^{(7)}   + d_j^{(8)}   + d_j^{(9)} \big) 
\displaybreak[0]\\&  -  {\ts\frac12} |\xi_j|^2 d_j^{(10)} 
+  {\ts\frac12}   g_j d_j^{(11)} 
-   2 |\xi_j|^2 g_j   \big( { \ts \sum_i } |\zeta_{0i}|^2 \big) \psi_{0j} 
\qquad\text{for $j\in J$,} 
\displaybreak[0]  \\[.5em] 
D_{jik}  = & - d_{jik}^{(1)} + d_{jik}^{(2)}  + \bond  \ii \xi_{jik} \cdot d_{jik}^{(3)}, 
\displaybreak[0]\ \\ 
 C_{jik}   = &  - g_{jik} \big( d_{jik}^{(4)} +  d_{jik}^{(5)}  \big) 
 - \big( d_{jik}^{(6)}  + d_{jik}^{(7)}  + d_{jik}^{(8)} + d_{jik}^{(9)} \big)
\displaybreak[0] \\&
 -  {\ts\frac12} |\xi_{jik}|^2 d_{jik}^{(10)}
+  {\ts\frac12}   g_{jik} d_{jik}^{(11)} \qquad \text{for $ (j,i,k)\in K$.}
\end{align*}
The $d_j^{(n)}$, $d_{jik}^{(n)}$ ($n\in \{1,\ldots,11\}$) can be calculated from $c_{ji}^{(n)}$, $a_k^{(n)}$ by the formulas 
%
%Moreover, using the formula 
%\begin{align*}
%\Big(\sum_{ji} c_{ji} \ee_{ji} + \cc \Big)  \Big(\sum_k a_k \ee_k + \cc \Big)  =  \sum_j d_j \ee_j  + \sum_{jik} d_{jik} \ee_{jik} + \cc 
%\end{align*}
%with 
\begin{align*}
&
d_1 = {\ts \sum_i} c_{1i} \bar a_i  + {\ts \sum_{i=2,3} } c_{1,-i} a_i   ,
\displaybreak[0] \\&
d_2 =  c_{12} \bar a_1 + c_{22} \bar a_2  + c_{23} \bar a_3 + \bar c_{1,-2} a_1 + c_{2,-3} a_3    ,
\displaybreak[0] \\&
d_3 = {\ts  \sum_j } c_{j3} \bar a_j + {\ts  \sum_{j=1,2} }  \bar c_{j,-3} a_j  ,
\displaybreak[0] \\
& d_{jjj} = c_{jj} a_j \qquad \text{for $j\in J$,}
\displaybreak[0] \\& 
d_{jji} = c_{jj} a_i + c_{ji} a_j , \quad d_{iij} = c_{ii} a_j + c_{ji} a_i \qquad \text{for $(j,i)\in I_<$,}
\displaybreak[0] \\ & 
d_{jj,-i} = c_{jj}\bar a_i + c_{j,-i} a_j , \quad d_{ii,-j} = c_{ii}\bar a_j + \bar c_{j,-i} a_i \qquad \text{for $(j,i)\in I_<$,}
\displaybreak[0] \\ &
 d_{123}  
 = c_{23} a_1 + c_{13} a_2 +  c_{12} a_3 , 
\displaybreak[0]  \\& 
 d_{12,-3} =  c_{2,-3} a_1 + c_{1,-3} a_2  + c_{12} \bar a_3 ,
\displaybreak[0]  \\&
 d_{13,-2} =  \bar c_{2,-3} a_1 + c_{13} \bar a_2 + c_{1,-2} a_3 ,
 \displaybreak[0] \\&
 d_{23,-1} = c_{23} \bar a_1 + \bar c_{1,-3} a_2 + \bar c_{1,-2} a_3  
  \end{align*}
with
\begin{alignat*}{2}
c_{ji}^{(1)} & =  \ii \xi_{ji} \psi_{1ji} , &\qquad  a_k^{(1)} & = \ii \xi_k \psi_{0k},
\displaybreak[0]  \\ 
c_{ji}^{(2)} & = \gamma_{ji}^{(1)} ,  & \qquad  a_k^{(2)} & = g_k  \psi_{0k},  
\displaybreak[0]  \\
c_{ji}^{(3)} & = \gamma_{ji}^{(2)} ,  & \qquad   a_k^{(3)} & = \ii \xi_k \zeta_{0k} , 
\displaybreak[0] \\
% c_{ji}^{(4)} & =  , &\qquad  a_k^{(4)} & = ,
% \\
c_{ji}^{(4)} & =  \zeta_{1ji} 
, &\qquad  a_k^{(4)} & = g_k \psi_{0k} ,
 \displaybreak[0] \\
c_{ji}^{(5)} & =  g_{ji} ( \psi_{1ji} -    \gamma_{ji}^{(3)}  )
, &\qquad  a_k^{(5)} & = \zeta_{0k} ,
% \\
%c_{ji}^{(5)} & =  g_{ji} \psi_{1ji}   
%, &\qquad  a_k^{(5)} & = \zeta_{0k} 
%\\
%c_{ji}^{(6)} & =  g_{ji} \gamma_{ji}^{(3)}   
%, &\qquad  a_k^{(6)} & = \zeta_{0k} 
\displaybreak[0]  \\
c_{ji}^{(6)} & =  \ii \xi_{ji} \zeta_{1ji}  
, &\qquad  a_k^{(6)} & = \ii \xi_k  \psi_{0k} ,
\displaybreak[0]  \\
c_{ji}^{(7)} & =   \zeta_{1ji}   
, &\qquad  a_k^{(7)} & =   - |\xi_k|^2 \psi_{0k} ,
\displaybreak[0] \\
c_{ji}^{(8)} & =   \ii \xi_{ji} \psi_{1ji} 
, &\qquad  a_k^{(8)} & =   \ii\xi_k  \zeta_{0k} ,
\displaybreak[0] \\
c_{ji}^{(9)} & = - |\xi_{ji}|^2 \psi_{1ji} 
, &\qquad  a_k^{(9)} & =\zeta_{0k},
\displaybreak[0] \\
c_{ji}^{(10)} & = \gamma_{ji}^{(4)} 
, &\qquad  a_k^{(10)} & = g_k \psi_{0k},
\displaybreak[0] \\
c_{ji}^{(11)} & =  \gamma_{ji}^{(4)} 
, &\qquad  a_k^{(11)} & = -|\xi_k|^2 \psi_{0k} 
%\qquad \text{for $(j,i)\in I$}, 
\end{alignat*}
for $(j,i)\in I$, $k\in J$,  where
\begin{align*}
\gamma_{jj}^{(1)} & = g_{jj} \psi_{1jj} + (  |\xi_j|^2  - g_{jj}  g_j )  \zeta_{0j} \psi_{0j}  , 
\displaybreak[0] \\
\gamma_{jj}^{(2)} & =  {\ts \frac12} |\xi_j|^2 \zeta_{0j}^2 , \quad 
\gamma_{jj}^{(3)} = g_j \zeta_{0j} \psi_{0j} , \quad 
\gamma_{jj}^{(4)} = \zeta_{0j}^2 \qquad \text{for $j\in J$,}
\displaybreak[0] \\
\gamma_{ji}^{(1)} & = g_{ji} \psi_{1ji} +( |\xi_i|^2  - g_{ji} g_i  ) \zeta_{0j} \psi_{0i} + (   |\xi_j|^2 - g_{ji} g_j )  \zeta_{0i}\psi_{0j}  , 
\displaybreak[0] \\
\gamma_{ji}^{(2)}  & = \xi_j  \cdot   \xi_i  \zeta_{0j} \zeta_{0i}  , \quad
\gamma_{ji}^{(3)} = g_i \zeta_{0j} \psi_{0i}  + g_j \zeta_{0i}  \psi_{0j}  , \quad 
\gamma_{ji}^{(4)} = 2 \zeta_{0j} \zeta_{0i} 
% \qquad \text{for $(j,i)\in I_<$,}
, \quad (j,i)\in I_<, 
\displaybreak[0] \\
\gamma_{j,-i}^{(1)} & = g_{j,-i} \psi_{1j,-i} +  ( |\xi_i|^2  - g_{j,-i} g_i ) \zeta_{0j}  \overline{ \psi_{0i}} +  ( |\xi_j|^2 - g_{j,-i}   g_j ) \overline{\zeta_{0i}}    \psi_{0j} , 
\displaybreak[0] \\
 \gamma_{j,-i}^{(2)} & = -  \xi_j   \cdot  \xi_i  \zeta_{0j}\overline{\zeta_{0i} } , \quad
\gamma_{j,-i}^{(3)} = g_i \zeta_{0j} \overline{ \psi_{0i} }  + g_j \overline{\zeta_{0i}}  \psi_{0j}  , \quad
\gamma_{j,-i}^{(4)} = 2 \zeta_{0j} \overline{\zeta_{0i} } 
, \quad (j,i)\in I_<.
\end{align*}
%\begin{align*}
%\gamma_{jj}^{(1)} & = g_{jj} \psi_{1jj} + (  |\xi_j|^2  - g_{jj}  g_j )  \zeta_{0j} \psi_{0j}  , 
%\\
%\gamma_{ji}^{(1)} & = g_{ji} \psi_{1ji} +( |\xi_i|^2  - g_{ji} g_i  ) \zeta_{0j} \psi_{0i} + (   |\xi_j|^2 - g_{ji} g_j )  \zeta_{0i}\psi_{0j}  , 
%\\
%\gamma_{j,-i}^{(1)} & = g_{j,-i} \psi_{1j,-i} +  ( |\xi_i|^2  - g_{j,-i} g_i ) \zeta_{0j}  \overline{ \psi_{0i}} +  ( |\xi_j|^2 - g_{j,-i}   g_j ) \overline{\zeta_{0i}}    \psi_{0j} , 
%\\
%\gamma_{jj}^{(2)} & =  {\ts \frac12} |\xi_j|^2 \zeta_{0j}^2 , \quad 
%\gamma_{ji}^{(2)}  = \xi_j  \cdot   \xi_i  \zeta_{0j} \zeta_{0i}  , \quad
% \gamma_{j,-i}^{(2)}  = -  \xi_j   \cdot  \xi_i  \zeta_{0j}\overline{\zeta_{0i} } , 
% \\
%\gamma_{jj}^{(3)} &= g_j \zeta_{0j} \psi_{0j} , \quad 
%\gamma_{ji}^{(3)} = g_i \zeta_{0j} \psi_{0i}  + g_j \zeta_{0i}  \psi_{0j}  , \quad 
%\gamma_{j,-i}^{(3)} = g_i \zeta_{0j} \overline{ \psi_{0i} }  + g_j \overline{\zeta_{0i}}  \psi_{0j}  ,
%\\
%\gamma_{jj}^{(4)}& = \zeta_{0j}^2 , \quad 
%\gamma_{ji}^{(4)} = 2 \zeta_{0j} \zeta_{0i} ,\quad 
%\gamma_{j,-i}^{(4)} = 2 \zeta_{0j} \overline{\zeta_{0i} } 
%\end{align*}

% \end{document}

%%%%%%%%%%%%%%%%%%%%%%%%%%%%%%%%%%%%%%%%%%%%%%%%%%%%%%%%%%%%%%%%%%%%%%%%%%%%%
 \bibliographystyle{amsplain}

\end{document}